\documentclass[12pt,twoside]{report}
\usepackage{amsmath,amssymb,bbm,amsthm,graphicx,url}
\usepackage{euler}

\newtheorem{theorem}{Theorem}[chapter]
\newtheorem{corollary}[theorem]{Corollary}
\newtheorem{lemma}[theorem]{Lemma}
\newtheorem{proposition}[theorem]{Proposition}
\newtheorem{conjecture}[theorem]{Conjecture}

\theoremstyle{definition}

\newtheorem{example}[theorem]{Example}

\newcommand{\C}{{\mathbb C}}
\newcommand{\Cd}{{\C^d}}

\newcommand{\N}{{\mathbb N}}
\newcommand{\R}{{\mathbb R}}
\newcommand{\Rd}{{\R^d}}
\newcommand{\T}{{\mathbb T}}

\newcommand{\Z}{{\mathbb Z}}

\newcommand{\e}{{\varepsilon}}
\newcommand{\lam}{{\lambda}}

\newcommand{\Interior}{\operatorname{Int}}
\newcommand{\res}{\operatorname{res}}
\newcommand{\sech}{\operatorname{sech}}
\newcommand{\spec}{\operatorname{spec}}
\newcommand{\Ran}{\operatorname{Ran}}
\renewcommand{\Re}{\operatorname{Re}}
\renewcommand{\Im}{\operatorname{Im}}

\newcommand{\LHP}{\text{LHP}}
\newcommand{\RHP}{\text{RHP}}

\newcommand{\la}{\langle}
\newcommand{\ra}{\rangle}

\newcommand{\lv}{\lVert}
\newcommand{\rv}{\rVert}

\title{\textbf{Spectral Theory of Partial Differential Equations} \\ \ \\ \textbf{Lecture Notes} \\ \ \\ University of Illinois \\ at Urbana--Champaign}

\author{Richard S. Laugesen
\footnote{Copyright \copyright\ 2011, Richard S. Laugesen
(Laugesen@illinois.edu). This work is licensed under the Creative
Commons Attribution--Noncommercial--Share Alike 3.0 Unported
License. To view a copy of this license, visit
\protect\url{http://creativecommons.org/licenses/by-nc-sa/3.0/}.} }

\begin{document}

\maketitle

\section*{Preface}

A \emph{textbook} presents far more material than any professor can cover in class. These \emph{lecture notes} present only somewhat more than I covered during the half-semester course Spectral Theory of Partial Differential Equations (Math 595 STP) at the University of Illinois, Urbana--Champaign, in Fall 2011.

\vspace{6pt} \noindent I make no claims of originality for the material presented other than some originality of emphasis: I emphasize computable examples before developing the general theory. This approach leads to occasional redundancy, and sometimes we use ideas before they are properly defined, but I think students gain a better understanding of the purpose of a theory after they are first well grounded in specific examples.

Please email me with corrections, and suggested improvements.

\vspace{24pt} \noindent Richard S. Laugesen \qquad \qquad \qquad
\textsc{Email:} Laugesen\@@illinois.edu
\\ Department of Mathematics \\ University of Illinois at Urbana--Champaign, U.S.A.

\newpage
\section*{Prerequisites and notation}

We assume familiarity with elementary Hilbert space theory: inner product, norm, Cauchy--Schwarz, orthogonal complement, Riesz Representation Theorem, orthonormal basis (ONB), bounded operators, and compact operators. Our treatment of discrete spectra builds on the spectral theorem for compact, selfadjoint operators.

All functions are assumed to be measurable. We use the function spaces
\begin{align*}
L^1 & = \text{integrable functions,} \\
L^2 & = \text{square integrable functions,} \\
L^\infty & = \text{bounded functions,}
\end{align*}
but we have no need of general $L^p$ spaces.

Sometimes we employ the $L^2$-theory of the Fourier transform,
\[
\widehat{f}(\xi) = \int_\Rd f(x) e^{-2\pi i \xi \cdot x} \, dx .
\]
Only the basic facts are needed, such as that the Fourier transform preserves the $L^2$ norm and maps derivatives in the spatial domain to multipliers in the frequency domain.

We use the language of Sobolev spaces throughout. Readers unfamiliar with this language can proceed unharmed: we mainly need that
\begin{align*}
H^1 = W^{1,2} & = \{ \text{$L^2$-functions with $1$ derivative in $L^2$} \} , \\
H^2 = W^{2,2} & = \{ \text{$L^2$-functions with $2$ derivatives in $L^2$} \} ,
\end{align*}
and
\[
H^1_0 = W^{1,2}_0 = \{ \text{$H^1$-functions that equal zero on the boundary} \} .
\]
(These characterizations are not mathematically precise, but they are good enough for our purposes.) Later we will recall the standard inner products that make these spaces into Hilbert spaces.

For more on Sobolev space theory, and related concepts of weak solutions and elliptic regularity, see \cite{E}.

\newpage
\section*{Introduction}

Spectral methods permeate the theory of partial differential equations. One solves linear PDEs by separation of variables, getting eigenvalues when the spectrum is discrete and continuous spectrum when it is not. Linearized stability of a steady state or traveling wave of a nonlinear PDE depends on the sign of the first eigenvalue, or on the location of the continuous spectrum in the complex plane.

This minicourse aims at highlights of spectral theory for \textbf{selfadjoint} partial differential operators, with a heavy emphasis on problems with \textbf{discrete spectrum}.

\medskip
\emph{Style of the course.} Research work differs from standard course work. Research often starts with questions motivated by analogy, or by trying to generalize special cases. Normally we find answers in a nonlinear fashion, slowly developing a coherent theory by linking up and extending our scraps of known information. We cannot predict what we will need to know in order to succeed, and we certainly do not have enough time to study all relevant background material. To succeed in research, we must develop a rough mental map of the surrounding mathematical landscape, so that we know the key concepts and canonical examples (without necessarily knowing the proofs). Then when we need to learn more about a topic, we know where to begin.

This course aims to develop your mental map of spectral theory in partial differential equations. We will emphasize computable examples, and will be neither complete in our coverage nor completely rigorous in our approach. Yet you will finish the course having a much better appreciation of the main issues and techniques in the subject.

\medskip
\emph{Closing thoughts.} If the course were longer, then we could treat topics such as nodal patterns, geometric bounds for the first eigenvalue and the spectral gap, majorization techniques (passing from eigenvalue sums to spectral zeta functions and heat traces), and inverse spectral problems. And we could investigate more deeply the spectral and scattering theory of operators with continuous spectrum, giving applications to stability of traveling waves and similarity solutions. These fascinating topics must await another course\ldots

\tableofcontents

\part{Discrete Spectrum}
\label{part:discrete}

\chapter{ODE preview}

\subsubsection*{Goal}

To review the role of eigenvalues and eigenvectors in solving 1st and 2nd order systems of linear ODEs; to interpret eigenvalues as decay rates, frequencies, and stability indices; and to observe formal analogies with PDEs.

\subsubsection*{Notational convention} Eigenvalues are written with multiplicity, and are listed in increasing order (when real-valued):
\[
\lam_1 \leq \lam_2 \leq \lam_3 \leq \cdots
\]

\subsubsection*{Spectrum of a real symmetric matrix}
If $A$ is a real symmetric $d \times d$ matrix (\emph{e.g.}\, $A = \left[ \begin{smallmatrix} a & b \\ b & c \end{smallmatrix} \right]$ when $d=2$) or Hermitian matrix then its \emph{spectrum} is the collection of eigenvalues:
\[
\spec(A) = \{ \lam_1,\ldots,\lam_d \} \subset \R
\]
\begin{figure}
\begin{center}
\includegraphics[width=0.4\textwidth]{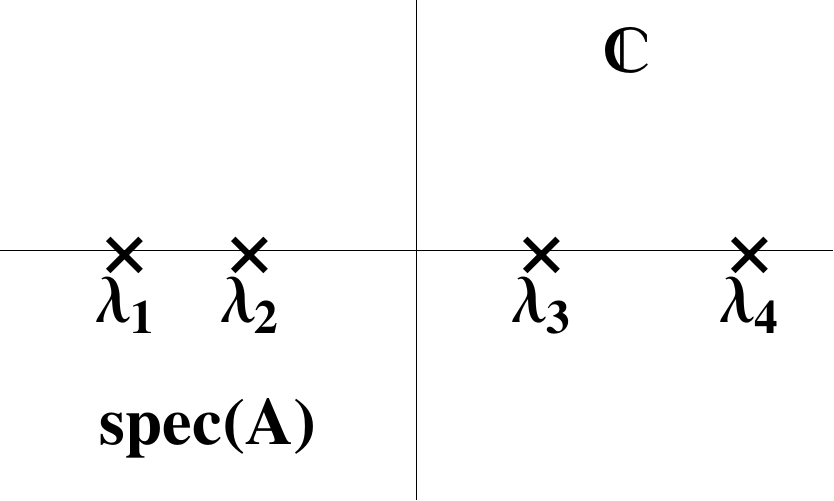}
\end{center}
\end{figure}
(see the figure). Recall that
\[
A v_j = \lam_j v_j
\]
where the eigenvectors $\{ v_1,\ldots,v_d\}$ can be chosen to form an ONB for $\Rd$.

Observe $A : \Rd \to \Rd$ is diagonal with respect to the eigenbasis:
\begin{align*}
A \big( \sum c_j v_j \big) & = \sum \lam_j c_j v_j \\ \ \\
\left[ \begin{matrix} \lam_1 & & 0 \\ & \ddots & \\ 0 & & \lam_d \end{matrix} \right] \! \left[ \begin{matrix} c_1 \\ \vdots \\ c_d \end{matrix}\right] & = \left[ \begin{matrix} \lam_1 c_1 \\ \vdots \\ \lam_d c_d \end{matrix}\right]
\end{align*}

\subsection*{What does the spectrum tell us about linear ODEs?}
\begin{example}[1st order] \label{ex1-1} The equation
\begin{align*}
\frac{dv}{dt} & = -Av \\
v(0)& = \sum c_j v_j
\end{align*}
has solution
\[
v(t) = e^{-At}v(0) \overset{\text{def}}{=} \sum e^{-\lam_j t} c_j v_j .
\]
Notice $\lam_j=$\textbf{decay rate} of the solution in direction $v_j$ if $\lam_j>0$, or \textbf{growth rate} (if $\lam_j<0$).

Long-time behavior: the solution is dominated by the first mode, with
\[
v(t) \sim e^{-\lam_1 t} c_1 v_1 \qquad \text{for large $t$,}
\]
assuming $\lam_1<\lam_2$ (so that the second mode decays faster than the first). The rate of collapse onto the first mode is governed by the \textbf{spectral gap} $\lam_2-\lam_1$ since
\begin{align*}
v(t)
& = e^{-\lam_1 t} \big( c_1 v_1 + \sum_{j=2}^d e^{-(\lam_j-\lam_1)t} c_j v_j \big) \\
& \sim e^{-\lam_1 t} \big( c_1 v_1 + O(e^{-(\lam_2-\lam_1)t}) .
\end{align*}
\end{example}
\begin{example}[2nd order] \label{ex1-2} Assume $\lam_1>0$, so that all the eigenvalues are positive. Then
\begin{align*}
\frac{d^2v}{dt^2} & = -Av \\
v(0)& = \sum c_j v_j \\
v^\prime(0)& = \sum c_j^\prime v_j
\end{align*}
has solution
\begin{align*}
v(t)
& = \cos(\sqrt{A}t)v(0) + \frac{1}{\sqrt{A}} \sin(\sqrt{A}t) v^\prime(0) \\
& \overset{\text{def}}{=} \sum \cos(\sqrt{\lam_j}t) c_j v_j + \sum \frac{1}{\sqrt{\lam_j}} \sin(\sqrt{\lam_j}t) c_j^\prime v_j .
\end{align*}
Notice $\sqrt{\lam_j}=$\textbf{frequency} of the solution in direction $v_j$.
\end{example}
\begin{example}[1st order imaginary] \label{ex1-3} The equation
\begin{align*}
i \frac{dv}{dt} & = Av \\
v(0)& = \sum c_j v_j
\end{align*}
has complex-valued solution
\[
v(t) = e^{-iAt}v(0) \overset{\text{def}}{=} \sum e^{-i\lam_j t} c_j v_j .
\]
This time $\lam_j=$\textbf{frequency} of the solution in direction $v_j$.
\end{example}

\subsection*{What does the spectrum tell us about nonlinear ODEs? In/stability!}
\begin{example}[1st order nonlinear] \label{ex1-4} Suppose
\[
\frac{dv}{dt} = F(v)
\]
where the vector field $F$ satisfies $F(0)=0$, with first order Taylor expansion
\[
F(v) = Bv + O(|v|^2)
\]
for some matrix $B$ having $d$ linearly independent eigenvectors $v_1,\ldots,v_d$ and corresponding eigenvalues $\lam_1,\ldots,\lam_d \in \C$. (The eigenvalues come in complex conjugate pairs, since $B$ is real.)

Clearly $v(t) \equiv 0$ is an equilibrium solution. Is it stable? To investigate, we linearize the ODE around the equilibrium to get $\frac{dv}{dt} = Bv$, which has solution
\[
v(t) = e^{Bt}v(0) = \sum e^{\lam_j t} c_j v_j .
\]
Notice $v(t) \to 0$ as $t \to \infty$ if $\Re(\lam_j)<0$ for all $j$, whereas $|v(t)| \to \infty$ if $\Re(\lam_j)>0$ for some $j$ (provided the corresponding coefficient $c_j$ is nonzero, and so on). Hence the equilibrium solution $v(t) \equiv 0$ is:
\begin{itemize}
\item \emph{linearly asymptotically stable} if $\spec(B) \subset \LHP$,
\begin{figure}[h]
\begin{center}
\includegraphics[width=0.4\textwidth]{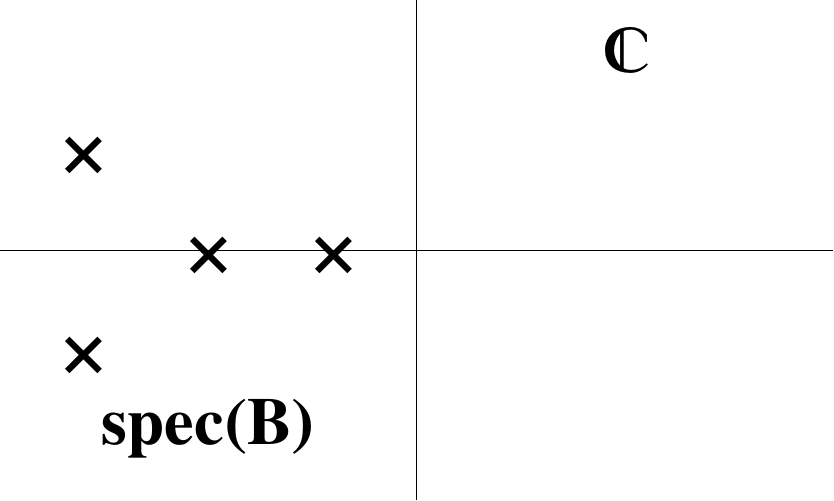}
\end{center}
\end{figure}
\item \emph{linearly unstable} if $\spec(B) \cap \RHP \neq \emptyset$.
\begin{figure}[h]
\begin{center}
\includegraphics[width=0.4\textwidth]{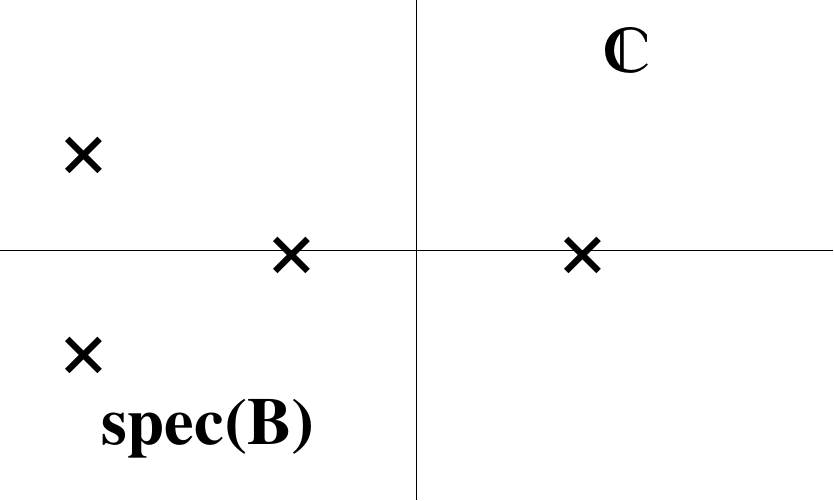}
\end{center}
\end{figure}
\end{itemize}
The \emph{Linearization Theorem} guarantees that the nonlinear ODE indeed behaves like the linearized ODE near the equilibrium solution, in the stable and unstable cases.

The nonlinear ODE's behavior requires further investigation in the neutrally stable case where the spectrum lies in the closed left half plane and intersects the imaginary axis ($\Re(\lam_j) \leq 0$ for all $j$ and $\Re(\lam_j)=0$ for some $j$).
\begin{figure}[h]
\begin{center}
\includegraphics[width=0.4\textwidth]{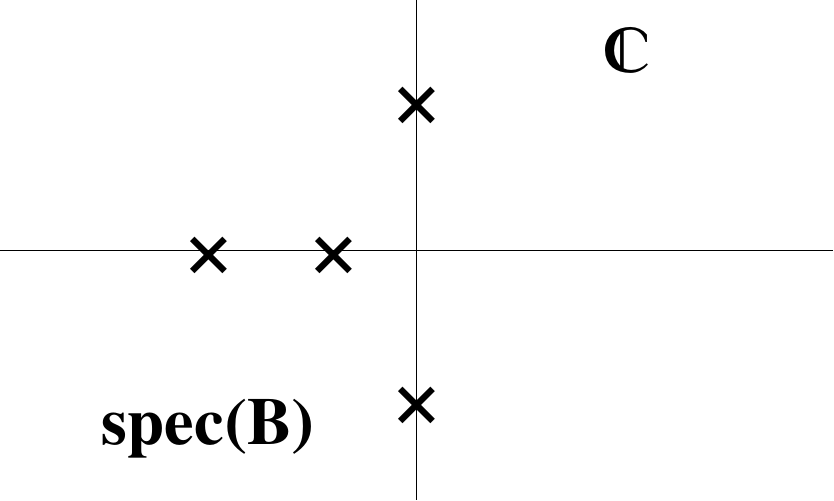}
\end{center}
\end{figure}

For example, if $B=\left[ \begin{smallmatrix} 0 & -1 \\ 1 & 0\end{smallmatrix} \right]$ (which has eigenvalues $\pm i$), then the phase portrait for $\frac{dv}{dt}=Bv$ consists of circles centered at the origin, but the phase portrait for the nonlinear system $\frac{dv}{dt} = F(v)$ might spiral in towards the origin (stability) or out towards infinity (instability), or could display even more complicated behavior.
\end{example}

\subsection*{Looking ahead to PDEs} Now suppose $A$ is an elliptic operator on a domain $\Omega \subset \Rd$. For simplicity, take $A=-\Delta$. Assume boundary conditions that make the operator self-adjoint (we will say more about boundary conditions later). Then the eigenvalues $\lam_j$ and eigenfunctions $v_j(x)$ of the Laplacian satisfy
\[
-\Delta v_j = \lam_j v_j \qquad \text{in $\Omega$}
\]
and the spectrum increases to infinity:
\[
\lam_1 \leq \lam_2 \leq \lam_3 \leq \cdots \to \infty .
\]
The eigenfunctions form an ONB for $L^2(\Omega)$.

Substituting $A=-\Delta$ into the ODE Examples~\ref{ex1-1}--\ref{ex1-3} transforms them into famous partial differential equations for the function $v(x,t)$. We solve these PDEs formally by separation of variables:
\begin{itemize}
\item Example \protect\ref{ex1-1} --- diffusion equation $v_t=\Delta v$. Separation of variables gives the solution
\[
v = e^{\Delta t}v(\cdot,0) \overset{\text{def}}{=} \sum e^{-\lam_j t} c_j v_j
\]
where the initial value is $v(\cdot,0) = \sum c_j v_j$. Here $\lam_j=$decay rate.

\item Example \protect\ref{ex1-2} --- wave equation $v_{tt}=\Delta v$. Separation of variables gives
\begin{align*}
v
& = \cos(\sqrt{-\Delta}t)v(\cdot,0) + \frac{1}{\sqrt{-\Delta}} \sin(\sqrt{-\Delta}t) v_t(\cdot,0) \\
& \overset{\text{def}}{=} \sum \cos(\sqrt{\lam_j}t) c_j v_j + \sum \frac{1}{\sqrt{\lam_j}} \sin(\sqrt{\lam_j}t) c_j^\prime v_j .
\end{align*}
So $\sqrt{\lam_j}=$frequency and $v_j=$mode of vibration.

\item Example \protect\ref{ex1-3} --- Schr\"{o}dinger equation $iv_t=-\Delta v$. Separation of variables gives
\[
v = e^{i\Delta t}v(\cdot,0) \overset{\text{def}}{=} \sum e^{-i\lam_j t} c_j v_j .
\]
Here $\lam_j=$frequency or energy level, and $v_j=$quantum state.
\end{itemize}

\medskip
We aim in what follows to analyze not just the Laplacian, but a whole family of related operators including:
\begin{align*}
A & = -\Delta && \text{Laplacian,} \\
A & = -\Delta + V(x) && \text{Schr\"{o}dinger operator,} \\
A & = (i\nabla+\vec{A})^2 && \text{magnetic Laplacian,} \\
A & = (-\Delta)^2= \Delta \Delta && \text{biLaplace operator.}
\end{align*}
The spectral theory of these operators helps explain the stability of different kinds of ``equilibria'' for evolution equations: steady states, standing waves, traveling waves, and similarity solutions.

\chapter{Laplacian --- computable spectra} \label{ch:lce}

\subsubsection*{Goal} To develop a library of explicitly computable spectra, which we use later to motivate and understand the general theory. The examples are classical and so proofs are left to the reader, or else omitted, except that Weyl's asymptotic law is proved in detail for rectangles.

\paragraph*{References} \cite{S} Chapters 4, 10; \cite{F} Lesson 30.

\subsubsection*{Notation} Let $\Omega$ be a bounded domain in $\Rd, d \geq 1$. Fix $L>0$.

Abbreviate ``boundary condition'' as ``BC'':
\begin{itemize}
\item \emph{Dirichlet} BC means $u=0$ on $\partial \Omega$,
\item \emph{Robin} BC means $\frac{\partial u}{\partial n} + \sigma u =  0$ on $\partial \Omega$ (where $\sigma \in \R$ is the Robin constant),
\item \emph{Neumann} BC means $\frac{\partial u}{\partial n}=0$ on $\partial \Omega$.
\end{itemize}

\subsection*{Spectra of the Laplacian}
\[
\Delta = \nabla \cdot \nabla = \Big( \frac{\partial\ \,}{\partial x_1} \Big)^{\! 2} + \cdots + \Big( \frac{\partial\ \,}{\partial x_d} \Big)^{\! 2}
\]
Eigenfunctions satisfy $\boxed{-\Delta u = \lam u}$, and we order the eigenvalues in increasing order as
\[
\lam_1 \leq \lam_2 \leq \lam_3 \leq \cdots \to \infty .
\]
To get an ONB one should normalize the eigenfunctions in $L^2$, but for simplicity, we will not normalize the following examples.

\paragraph*{One dimension} $\boxed{-u^{\prime \prime}=\lam u}$

\bigskip
\noindent 1. \emph{Circle} $\T = \R / 2\pi \Z$, periodic BC: $u(-\pi)=u(\pi),u^\prime(-\pi)=u^\prime(\pi)$.

Eigenfunctions $e^{ijx}$ for $j \in \Z$, or equivalently $1,\cos(jx),\sin(jx)$ for $j \geq 1$.

Eigenvalues $\lam_j=j^2$ for $j \in \Z$, or $\lam=0^2,1^2,1^2,2^2,2^2,\ldots$

\bigskip
\noindent 2. \emph{Interval} $(0,L)$

\medskip
\noindent (a) Dirichlet BC: $u(0)=u(L)=0$.

Eigenfunctions $u_j(x)=\sin(j\pi x/L)$ for $j \geq 1$.

Eigenvalues $\lam_j=(j\pi/L)^2$ for $j \geq 1$, \emph{e.g.}\ $L=\pi \ \Rightarrow \ \lam=1^2,2^2,3^2,\ldots$
\begin{figure}[h]
\begin{center}
\includegraphics[width=0.4\textwidth]{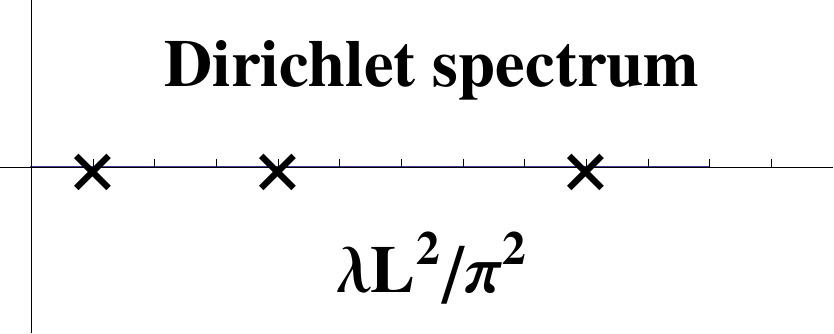} \qquad
\includegraphics[width=0.4\textwidth]{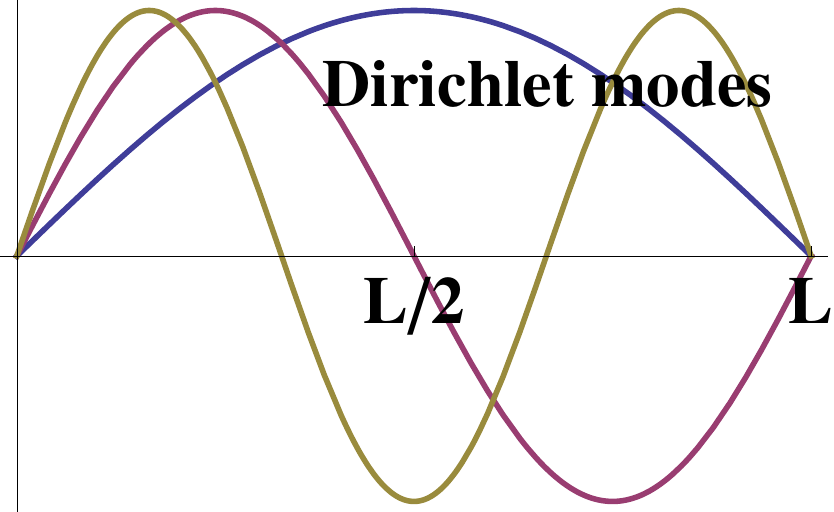}
\end{center}
\end{figure}

\medskip
\noindent (b) Robin BC: $-u^\prime(0)+\sigma u(0)=u^\prime(L)+\sigma u(L)=0$.

Eigenfunctions $u_j(x)=\sqrt{\rho_j} \cos(\sqrt{\rho_j}x)+\sigma \sin(\sqrt{\rho_j}x)$.

Eigenvalues $\rho_j=$ $j$th positive root of $\tan(\sqrt{\rho}L)=\frac{2\sigma \sqrt{\rho}}{\rho - \sigma^2}$ for $j \geq 1$.
\begin{figure}[h]
\begin{center}
\includegraphics[width=0.4\textwidth]{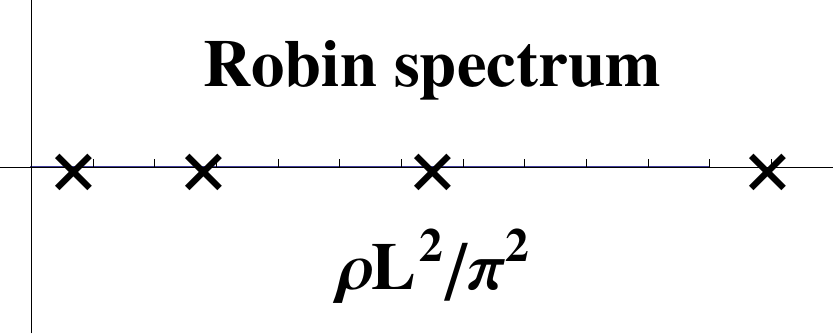} \qquad
\includegraphics[width=0.4\textwidth]{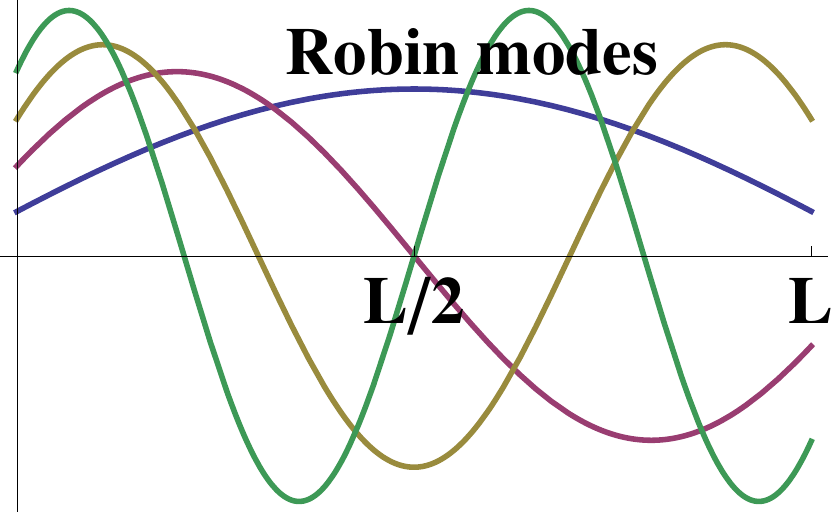}
\end{center}
\end{figure}

\medskip
\noindent (c) Neumann BC: $u^\prime(0)=u^\prime(L)=0$.

Eigenfunctions $u_j(x)=\cos(j\pi x/L)$ for $j \geq 0$ (note $u_0 \equiv 1$).

Eigenvalues $\mu_j=(j\pi/L)^2$ for $j \geq 0$, \emph{e.g.}\ $L=\pi \ \Rightarrow \ \lam=0^2,1^2,2^2,3^2,\ldots$
\begin{figure}[h]
\begin{center}
\includegraphics[width=0.4\textwidth]{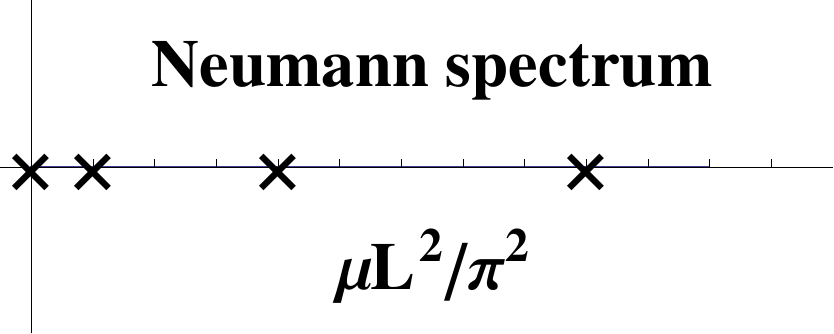} \qquad
\includegraphics[width=0.4\textwidth]{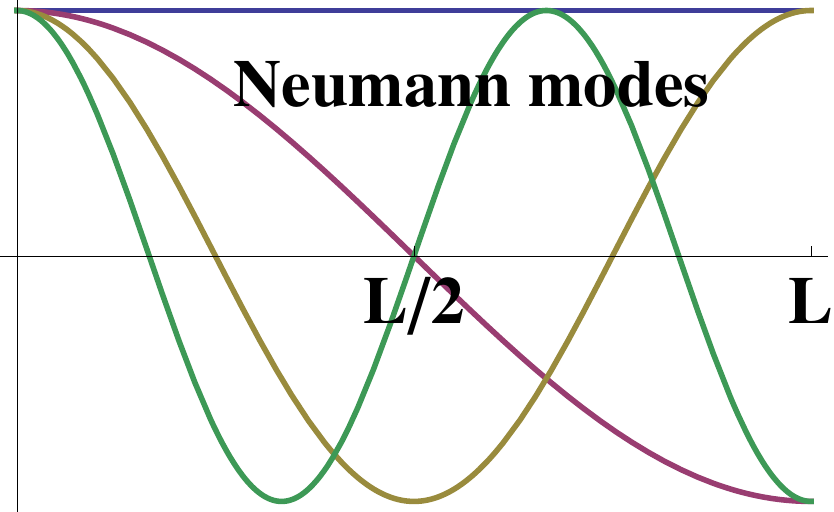}
\end{center}
\end{figure}

\paragraph*{Spectral features in $1$ dim} \

i. Scaling: eigenvalue must balance $d^2/dx^2$, and so $\lam \sim (\text{length scale})^{-2}$.

\qquad \qquad Precisely, $\lam_j \big( (0,tL) \big)=\lam_j \big( (0,L) \big) /t^2$.

ii. Asymptotic: eigenvalues grow at a regular rate, $\lam_j \sim (\text{const.})j^2$

iii. Robin spectrum lies between Neumann and Dirichlet:
\[
\text{Neumann} \xleftarrow{\ \sigma \to 0 \ } \text{Robin} \xrightarrow{\ \sigma \to \infty\ } \text{Dirichlet}
\]
as one sees formally by letting $\sigma$ approach $0$ or $\infty$ in the Robin BC $\tfrac{\partial u}{\partial n}+\sigma u=0$.

\paragraph*{Two dimensions} $\boxed{-\Delta u=\lam u}$

\bigskip
\noindent 1. \emph{Rectangle} $\Omega = (0,L) \times (0,M)$ (product of intervals).

Separate variables using rectangular coordinates $x_1,x_2$. See the figures at the end of the chapter!

(Note that every rectangle can be reduced to a rectangle with sides parallel to the coordinate axes because the Laplacian, and hence its spectrum, is rotationally and translationally invariant.)

\medskip
\noindent (a) Dirichlet BC: $u=0$

Eigenfunctions $u_{jk}(x)=\sin(j\pi x_1/L) \sin(k\pi x_2/M)$ for $j,k \geq 1$.

Eigenvalues $\lam_{jk}=(j\pi/L)^2+(k\pi/M)^2$ for $j,k \geq 1$,

\emph{e.g.}\ $L=M=\pi \ \Rightarrow \ \lam=2,5,5,8,10,10,\ldots$

\medskip
\noindent (b) Neumann BC: $\tfrac{\partial u}{\partial n}=0$

Eigenfunctions $u_{jk}(x)=\cos(j\pi x_1/L) \cos(k\pi x_2/M)$ for $j,k \geq 0$.

Eigenvalues $\mu_{jk}=(j\pi/L)^2+(k\pi/M)^2$ for $j,k \geq 0$,

\emph{e.g.}\ $L=M=\pi \ \Rightarrow \ \lam=0,1,1,2,4,4,\ldots$

\bigskip
\noindent 2. \emph{Disk} $\Omega = \{ x \in \R^2 : |x|<R \}$.

Separate variables using polar coordinates $r,\theta$.

\medskip
\noindent (a) Dirichlet BC: $u=0$

Eigenfunctions
\begin{align*}
& \text{$J_0(rj_{0,m}/R)$ for $m \geq 1$,} \\
& \text{$J_n(rj_{n,m}r/R) \cos (n\theta)$ and $J_n(rj_{n,m}r/R) \sin (n\theta)$ for $n \geq 1, m \geq 1$.}
\end{align*}
Notice the modes with $n=0$ are purely radial, whereas when $n \geq 1$ the modes have angular dependence.

Eigenvalues $\lam=(j_{n,m}/R)^2$ for $n \geq 0, m \geq 1$, where
\begin{align*}
J_n & = \text{Bessel function of order $n$, and} \\
j_{n,m} & = \text{$m$-th positive root of $J_n(r)=0$.}
\end{align*}
The eigenvalue $\lam_{n,m}$ has multiplicity $2$ when $n \geq 1$, associated to both cosine and sine modes.

From the graphs of the Bessel functions $J_0,J_1,J_2$ we can read off the first 4 roots:
\[
j_{0,1} \simeq 2.40, \quad j_{1,1} \simeq 3.83 \quad j_{2,1} \simeq 5.13 \quad j_{1,2} \simeq 5.52 .
\]
These roots generate the first 6 eigenvalues (remembering the eigenvalues are double when $n \geq 1$).
\begin{figure}[h]
\begin{center}
\includegraphics[width=0.5\textwidth]{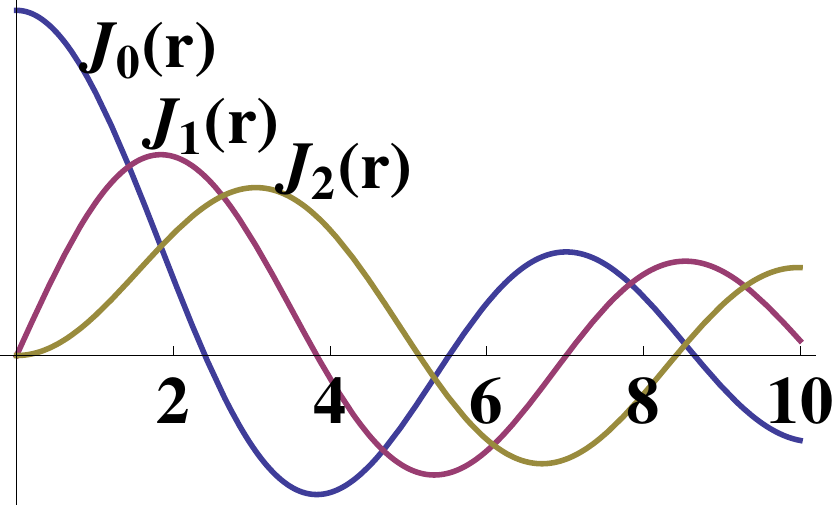}
\end{center}
\end{figure}

\medskip
\noindent (b) Neumann BC: $\tfrac{\partial u}{\partial n}=0$

Use roots of $J_n^\prime(r)=0$. See \cite[Chapter III]{B}.

\bigskip
\noindent 3. \emph{Equilateral triangle} of sidelength $L$.

Separation of variables fails, but one may reflect repeatedly to a hexagonal lattice whose eigenfunctions are trigonometric.

Dirichlet eigenvalues $\lam_{jk}=\tfrac{16\pi^2}{9L^2}(j^2+jk+k^2)$ for $j,k \geq 1$.

Neumann eigenvalues $\mu_{jk}=\tfrac{16\pi^2}{9L^2}(j^2+jk+k^2)$ for $j,k \geq 0$.

\noindent See \cite{MW,McC}.

\paragraph*{Spectral features in $2$ dim} \

\medskip
i. Scaling: eigenvalue must balance $\Delta$, and so $\lam \sim (\text{length scale})^{-2}$.

\qquad \qquad Precisely, $\lam_j(t\Omega)=\lam_j(\Omega)/t^2$.

\medskip
ii. Dirichlet and Neumann spectra behave quite differently when the domain degenerates. Consider the rectangle, for example. Fix one side length $L$, and let the other side length $M$ tend to $0$. Then the first positive Dirichlet eigenvalue blows up: taking $j=k=1$ gives eigenvalue $(\pi/L)^2+(\pi/M)^2 \to \infty$. The first positive Neumann eigenvalue is constant (independent of $M$): taking $j=1,k=0$, gives eigenvalue $(\pi/L)^2$.

\medskip
iii. Asymptotic: eigenvalues of the rectangle grow at a regular rate.
\begin{proposition}(Weyl's law for rectangles) \label{pr:rectweyl}
The rectangle $(0,L) \times (0,M)$ has
\[
\lam_j \sim \mu_j \sim \frac{4\pi j}{\text{Area}} \qquad  \text{as $j \to \infty$,}
\]
where $\text{Area}=LM$ is the area of the rectangle and $\lam_1,\lam_2,\lam_3,\ldots$ and $\mu_1,\mu_2,\mu_3,\ldots$ are the Dirichlet and Neumann eigenvalues respectively, in increasing order.
\end{proposition}
\begin{proof} We give the proof for Dirichlet eigenvalues. (The Neumann case is similar.) Define for $\alpha>0$ the eigenvalue counting function
\begin{align*}
N(\alpha)
& = \# \{ \text{eigenvalues} \leq \alpha \} \\
& = \# \big\{ j,k \geq 1 : \frac{j^2}{\alpha L^2/\pi^2} + \frac{k^2}{\alpha M^2/\pi^2} \leq 1 \big\} \\
& = \# \big\{ (j,k) \in \N \times \N : (j,k) \in E \big\}
\end{align*}
where $E$ is the ellipse $(x/a)^2+(y/b)^2 \leq 1$ and $a=\sqrt{\alpha}L/\pi, b=\sqrt{\alpha}M/\pi$.
\begin{figure}[h]
\begin{center}
\includegraphics[width=0.4\textwidth]{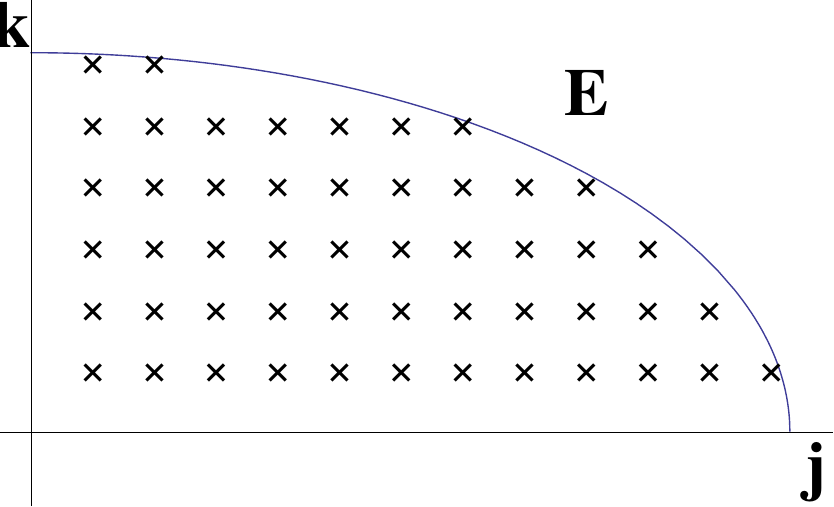} \qquad
\includegraphics[width=0.4\textwidth]{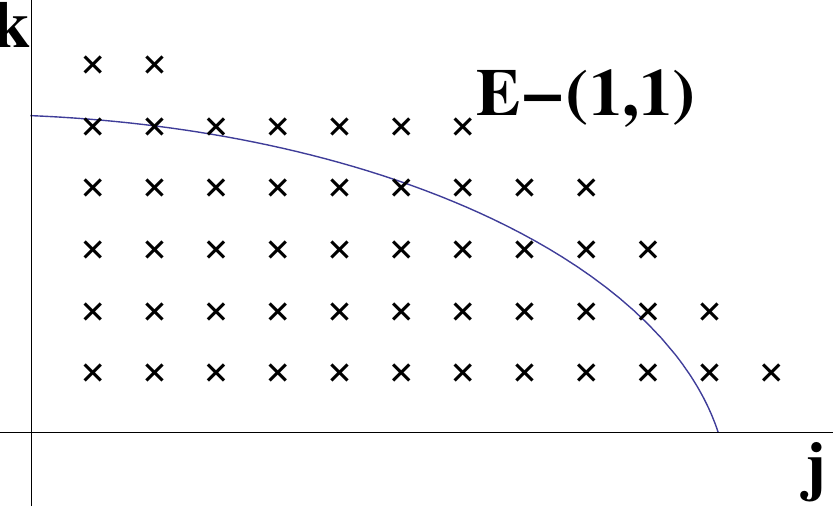}
\end{center}
\end{figure}

We associate each lattice point $(j,k) \in E$ with the square
\[
S(j,k)=[j-1,j] \times [k-1,k]
\]
whose upper right corner lies at $(j,k)$. These squares all lie within $E$, and so by comparing areas we find
\[
N(\alpha) \leq (\text{area of $E$ in first quadrant}) = \frac{1}{4} \pi ab = \frac{\text{Area}}{4\pi} \alpha .
\]
On the other hand, a little thought shows that the union of the squares covers a copy of $E$ shifted down and left by one unit:
\[
\cup_{(j,k) \in E} \, S(j,k) \supset \big( E - (1,1) \big) \cap (\text{first quadrant}) .
\]
Comparing areas shows that
\begin{align*}
N(\alpha)
& \geq \frac{1}{4} \pi ab - a - b \\
& = \frac{LM}{4\pi} \alpha - \frac{L+M}{\pi} \sqrt{\alpha} \\
& = \frac{\text{Area}}{4\pi} \alpha - \frac{\text{Perimeter}}{2\pi} \sqrt{\alpha} .
\end{align*}
Combining our upper and lower estimates shows that
\[
N(\alpha) \sim \frac{\text{Area}}{4\pi} \alpha
\]
as $\alpha \to \infty$. To complete the proof we simply invert this last asymptotic, with the help of the lemma below.
\end{proof}
\begin{lemma}(Inversion of asymptotics) \label{le:inversion} Fix $c>0$. Then:
\[
N(\alpha) \sim \frac{\alpha}{c} \qquad \Longrightarrow \qquad \lam_j \sim cj .
\]
\end{lemma}
\begin{proof}
Formally substituting $\alpha=\lam_j$ and $N(\alpha)=j$ takes us from the first asymptotic to the second. The difficulty with making this substitution rigorous is that if $\lam_j$ is a multiple eigenvalue, then $N(\lam_j)$ can exceed $j$.

To circumvent the problem, we argue as follows. Given $\e>0$ we know from $N(\alpha) \sim \alpha/c$ that
\[
(1-\e) \frac{\alpha}{c} < N(\alpha) < (1+\e) \frac{\alpha}{c}
\]
for all large $\alpha$. Substituting $\alpha=\lam_j$ into the right hand inequality implies that
\[
j < (1+\e) \frac{\lam_j}{c}
\]
for all large $j$. Substituting $\alpha=\lam_j-\delta$ into the left hand inequality implies that
\[
(1-\e) \frac{\lam_j-\delta}{c} < j
\]
for each large $j$ and $0<\delta<1$, and hence (by letting $\delta \to 0$) that
\[
(1-\e) \frac{\lam_j}{c} \leq j .
\]
We conclude that
\[
\frac{1}{1+\e} < \frac{\lam_j}{cj} \leq \frac{1}{1-\e}
\]
for all large $j$, so that
\[
\lim_{j \to \infty} \frac{\lam_j}{cj}  = 1
\]
as desired.
\end{proof}

Later, in Chapter~\ref{ch:weyl}, we will prove \textbf{Weyl's Asymptotic Law} that
\[
\lam_j \sim 4\pi j/\text{Area}
\]
for all bounded domains in $2$ dimensions, regardless of shape or boundary conditions.

\paragraph*{Question to ask yourself} What does a ``typical'' eigenfunction look like, in each of the examples above? See the following figures.

\begin{figure}[h]
\begin{center}
\includegraphics[width=0.4\textwidth]{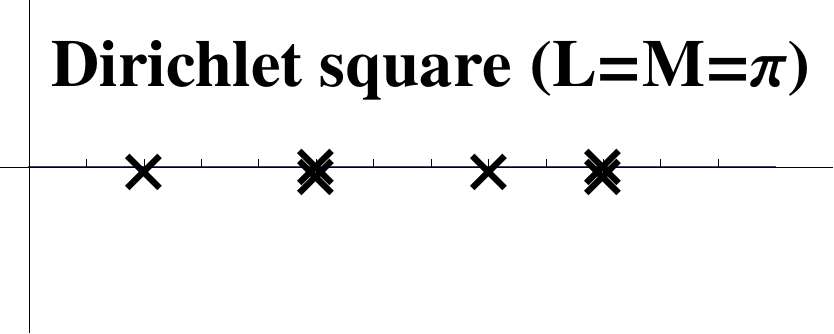} \qquad
\includegraphics[width=0.4\textwidth]{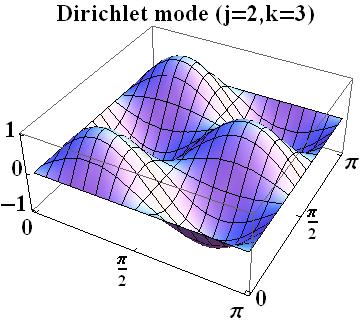}
%

\vspace*{36pt}
%
\includegraphics[width=0.4\textwidth]{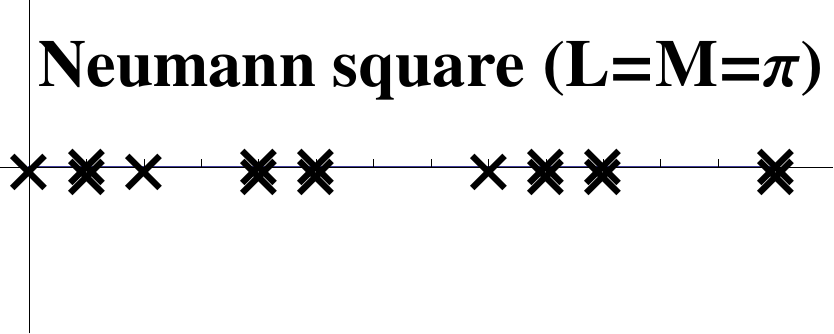} \qquad
\includegraphics[width=0.4\textwidth]{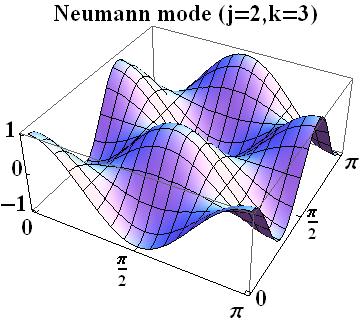}
%

\vspace*{36pt}
%
\includegraphics[width=0.4\textwidth]{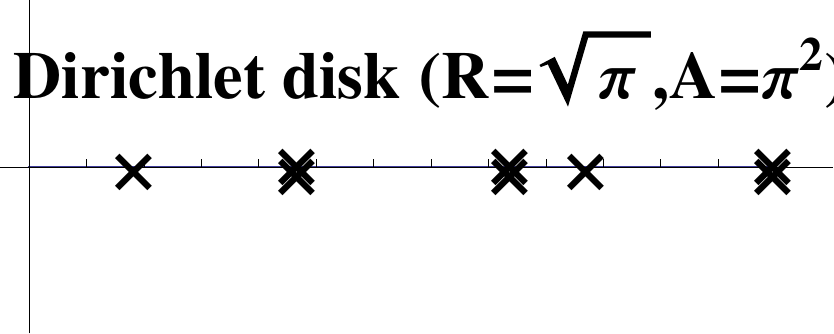} \qquad
\includegraphics[width=0.4\textwidth]{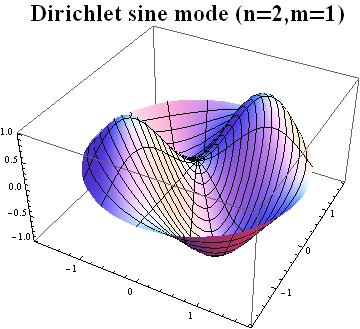}
\end{center}
\end{figure}

\chapter{Schr\"{o}dinger --- computable spectra} \label{ch:soce}

\subsubsection*{Goal} To study the classic examples of the harmonic oscillator ($1$ dim) and hydrogen atom ($3$ dim).

\paragraph*{References} \cite{S} Sections 9.4, 9.5, 10.7; \cite{GS} Section 7.5, 7.7

\subsection*{Harmonic oscillator in $1$ dimension $\boxed{-u^{\prime \prime}+x^2 u = E u}$}

Boundary condition: $u(x) \to 0$ as $x \to \pm \infty$. (Later we give a deeper perspective, in terms of a weighted $L^2$-space.)

Eigenfunctions $u_k(x)=H_k(x)e^{-x^2/2}$ for $k \geq 0$, where $H_k=$ $k$-th Hermite polynomial.

Eigenvalues $E_k=2k+1$ for $k \geq 0$, or $E=1,3,5,7,\ldots$

\noindent \emph{Examples.} $H_0(x)=1, H_1(x)=2x, H_2(x)=4x^2-2, H_k(x)=(-1)^k e^{x^2} \big( \tfrac{d\ }{dx} \big)^k e^{-x^2}$

Ground state: $u_0(x)=e^{-x^2/2}=$ Gaussian. (Check: $-u_0^{\prime \prime}+x^2 u_0 = u_0$)

\begin{figure}[h]
\begin{center}
\includegraphics[width=0.4\textwidth]{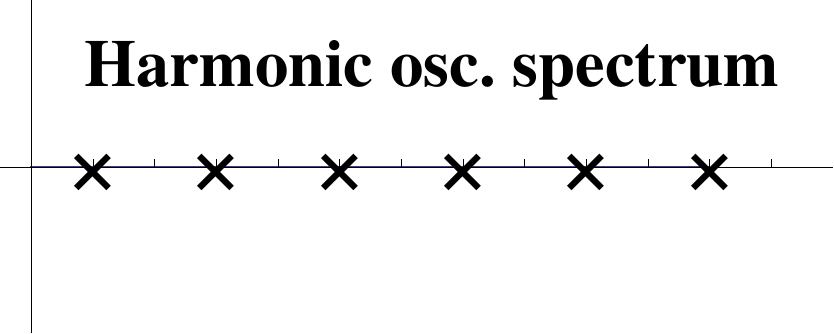} \qquad
\includegraphics[width=0.4\textwidth]{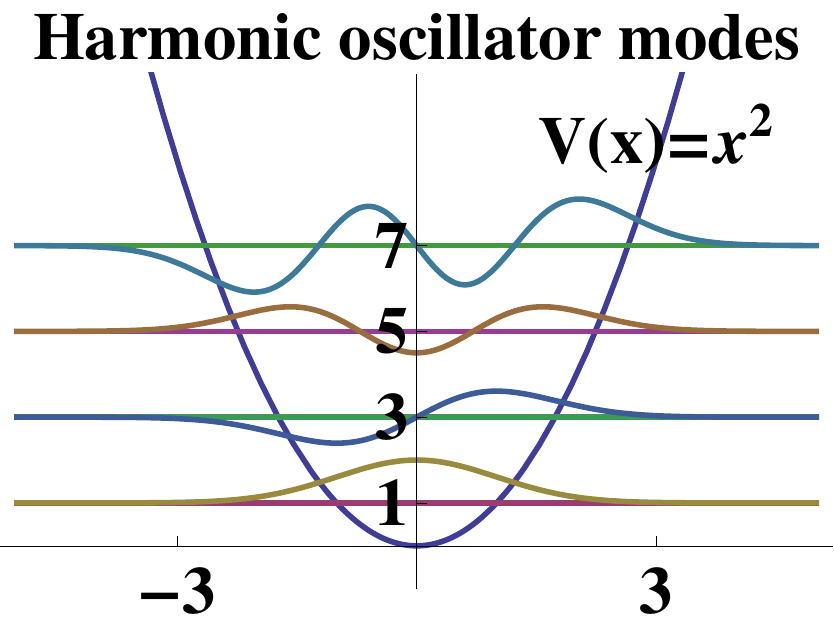}
\end{center}
\end{figure}

\subsubsection*{Quantum mechanical interpretation} If $u(x,t)$ solves the time-dependent Schr\"{od}inger equation
\[
iu_t = -u^{\prime \prime} + x^2 u
\]
with potential $V(x)=x^2$ and $u$ has $L^2$ norm equal to $1$, then $|u|^2$ represents the probability density for the location of a particle in a quadratic potential well.

The $k$-th eigenfunction $u_k(x)$ is called the \emph{$k$-th excited state}, because it gives a ``standing wave'' solution
\[
u(x,t)=e^{-iE_k t} u_k(x)
\]
to the time-dependent equation. The higher the frequency or ``energy'' $E_k$ of the excited state, the more it can spread out in the confining potential well, as the solution plots show.

\subsubsection*{Harmonic oscillator investigations}

\noindent \underline{Method 1: ODEs} Since $u_0(x)=e^{-x^2/2}$ is an eigenfunction, we guess that all eigenfunctions decay like $e^{-x^2/2}$. So we try the change of variable $u=we^{-x^2/2}$. The eigenfunction equation becomes
\[
w^{\prime \prime} - 2x w^\prime + (E-1)w = 0 ,
\]
which we recognize as the \textbf{Hermite equation}. Solving by power series, we find that the only appropriate solutions have terminating power series: they are the \textbf{Hermite polynomials}. (All other solutions grow like $e^{x^2}$ at infinity, violating the boundary condition on $u$.)

\medskip
\noindent \underline{Method 2: Raising and lowering} Define
\begin{align*}
h^+ & = - \frac{d\ }{dx} + x \qquad \text{(raising or creation operator),} \\
h^- & = \frac{d\ }{dx} + x \qquad \text{(lowering or annihilation operator).}
\end{align*}
Write $H=-\tfrac{d^2\ }{dx^2} + x^2$ for the harmonic oscillator operator. Then one computes that
\begin{align*}
H
& = h^+h^- + 1 \\
& = h^-h^+ - 1
\end{align*}
\emph{Claim.} If $u$ is an eigenfunction with eigenvalue $E$ then $h^\pm u$ is an eigenfunction with eigenvalue $E\pm 2$. (In other words, $h^+$ ``raises'' the energy, and $h^-$ ``lowers'' the energy.)

\noindent \emph{Proof.}
\begin{align*}
H(h^+ u)
& = (h^+ h^- + 1)(h^+ u) \\
& = h^+(h^-h^+ +1)u \\
& = h^+ (H+2)u \\
& = h^+(E+2) u \\
& = (E+2) h^+u
\end{align*}
and similarly $H(h^-u)=(E-2)h^-u$ (exercise).

\smallskip The only exception to the Claim is that $h^- u$ will not be an eigenfunction if $h^- u \equiv 0$, which occurs precisely when $u=u_0=e^{-x^2/2}$. Thus the lowering operator annihilates the ground state.

\subsubsection*{Relation to classical harmonic oscillator} Consider a classical oscillator with mass $m=2$, spring constant $k=2$, and displacement $x(t)$, so that $2\ddot{x}=-2x$. The total energy is
\[
\dot{x}^2 + x^2 = \text{const.} = E .
\]
To describe a quantum oscillator, we formally replace the momentum $\dot{x}$ with the ``momentum operator'' $-i \tfrac{d\ }{dx}$ and let the equation act on a function $u$:
\[
\Big[ \big( -i \frac{d\ }{dx} \big)^2 + x^2 \Big] u = E u .
\]
This is exactly the eigenfunction equation $-u^{\prime \prime}+x^2 u = Eu$.

\subsection*{Harmonic oscillator in higher dimensions $\boxed{-\Delta u+|x|^2 u = E u}$}
Here $|x|^2=x_1^2+\cdots +x_d^2$. The operator separates into a sum of $1$ dimensional operators, and hence has product type eigenfunctions
\[
u=u_{k_1}(x_1) \cdots u_{k_d}(x_d) , \qquad E = (2k_1 + 1) + \cdots + (2k_d + 1) .
\]

\subsection*{Hydrogen atom in $3$ dimensions $\boxed{-\Delta u - \tfrac{2}{|x|} u = E u}$}
Here $V(x)=-2/|x|$ is an attractive electrostatic (``Coulomb'') potential created by the proton in the hydrogen nucleus. (Notice the gradient of this potential gives the correct $|x|^{-2}$ inverse square law for electrostatic force.)

Boundary conditions: $u(x) \to 0$ as $|x| \to \infty$ (we will say more later about the precise formulation of the eigenvalue problem).

Eigenvalues: $E=-1,-\tfrac{1}{4}, -\tfrac{1}{9}, \ldots$ with multiplicities $1,4,9,\ldots$

\noindent That is, the eigenvalue $E=-1/n^2$ has multiplicity $n^2$.

Eigenfunctions: $e^{-r/n} L^\ell_n(r) Y^m_\ell(\theta,\phi)$ for $0 \leq |m| \leq n-1$, where $Y^m_\ell$ is a spherical harmonic and $L^\ell_n$ equals $r^\ell$ times a Laguerre polynomial.

(Recall the spherical harmonics are eigenfunctions of the spherical Laplacian in $3$ dimensions, with $-\Delta_{\text{sphere}}Y^m_\ell = \ell(\ell+1)Y^m_\ell$. In $2$ dimensions the spherical harmonics have the form $Y=\cos(k\theta)$ and $Y=\sin(k\theta)$, which satisfy $-\tfrac{d^2 \ }{d\theta^2} Y = k^2 Y$.)

\smallskip
\noindent \emph{Examples.} The first three purely radial eigenfunctions ($\ell=m=0,n=1,2,3$) are $e^{-r}, e^{-r/2}(1-\frac{r}{2}), e^{-r/3}(1-\tfrac{2}{3}r+\tfrac{2}{27}r^2)$.

\begin{figure}[h]
\begin{center}
\includegraphics[width=0.4\textwidth]{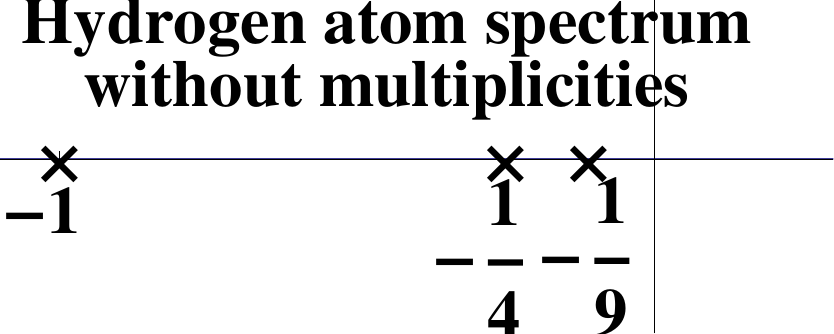} \qquad
\includegraphics[width=0.4\textwidth]{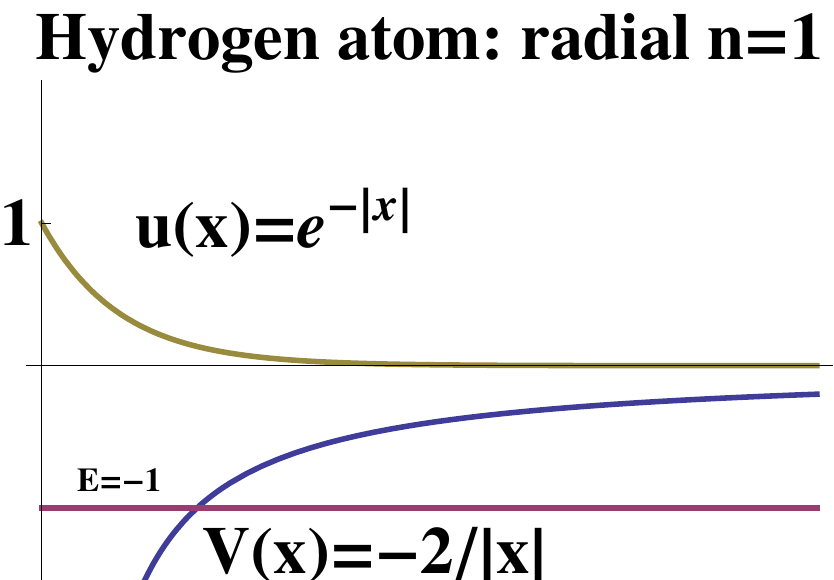} \qquad
\end{center}
\end{figure}
\begin{figure}[h]
\begin{center}
\includegraphics[width=0.4\textwidth]{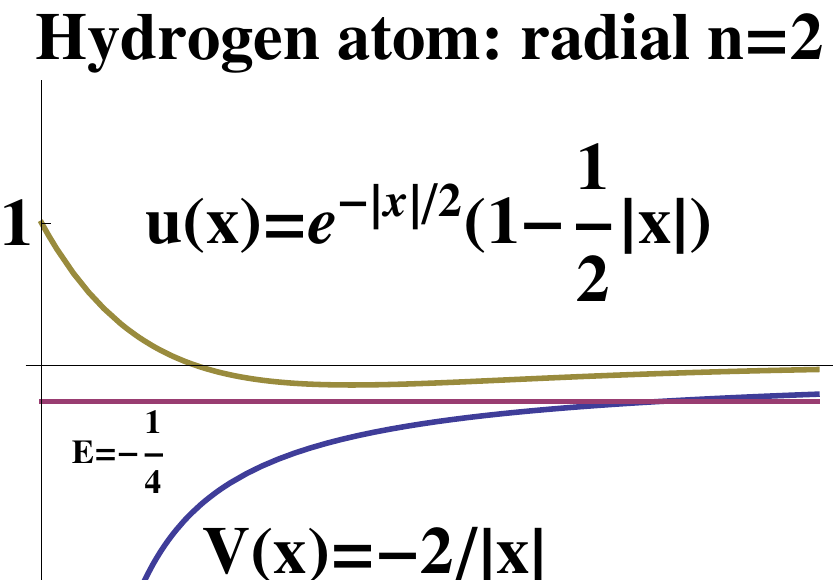} \qquad
\includegraphics[width=0.4\textwidth]{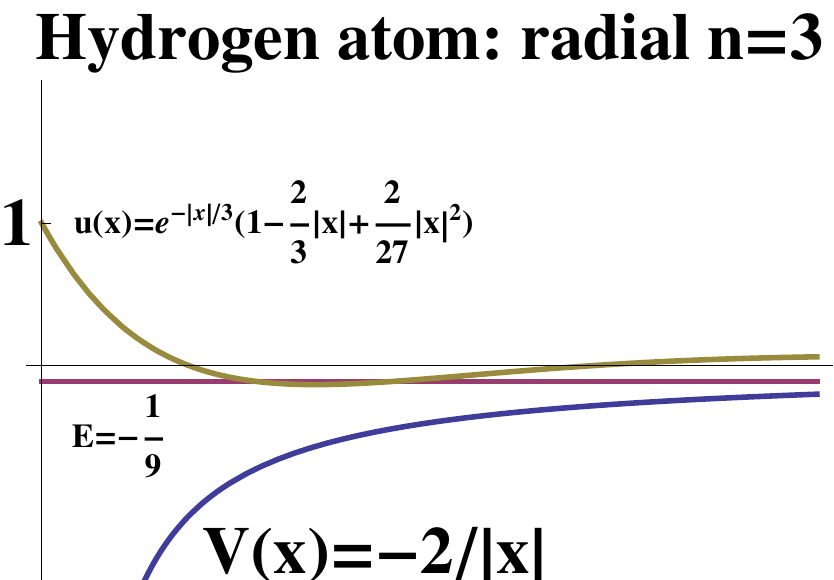}
\end{center}
\end{figure}

The corner in the graph of the eigenfunction at $r=0$ is caused by the singularity of the Coulomb potential.

\paragraph*{Continuous spectrum} Eigenfunctions with positive energy $E>0$ do exist, but they oscillate as $|x| \to \infty$, and thus do not satisfy our boundary conditions. They represent ``free electrons'' that are not bound to the nucleus. See our later discussion of continuous spectrum, in Chapter~\ref{ch:spectrum}.

\chapter{Discrete spectral theorem} \label{ch:eONB}

\subsubsection*{Goal} To state the spectral theorem for an elliptic
sesquilinear form on a dense, compactly imbedded Hilbert space, and to prove it using the spectral theorem for compact, selfadjoint operators. In later chapters we apply the spectral theorem to unify and extend the examples of Chapters~\ref{ch:lce} and \ref{ch:soce}.

\paragraph*{References} \cite{BlBl} Section 6.3

\subsection*{Matrix preview --- weak eigenvectors}
Consider a Hermitian matrix $A$, and suppose $u$ is an eigenvector with eigenvalue $\gamma$, so that $Au=\gamma u$. Take the dot product with an arbitrary vector $v$ to obtain
\[
Au \cdot \overline{v} = \gamma u \cdot \overline{v} , \qquad \forall v \in \Cd .
\]
We call this condition the ``weak form'' of the eigenvector equation. Clearly it implies the original ``strong'' form, because if $(Au-\gamma u) \cdot \overline{v} = 0$ for all $v$, then $Au-\gamma u=0$.

We will find the weak form useful below, when we generalize to Hilbert spaces.

The Hermitian nature of the matrix guarantees conjugate-symmetry of the left side of the weak equation, when $u$ and $v$ are interchanged:
\[
Au \cdot \overline{v} = \overline{Av \cdot \overline{u}} .
\]
This symmetry ensures that all eigenvalues are real, by choosing $v=u$. We will need a similar symmetry property in the Hilbert space setting.

\subsection*{PDE preview --- weak eigenfunctions}
Consider the eigenfunction equation $-\Delta u = \lam u$ for the
Laplacian, in a domain $\Omega$. Multiply by a function $v \in
H^1_0(\Omega)$, so that $v$ equals $0$ on $\partial \Omega$, and integrate
to obtain
\[
- \int_\Omega v \Delta u \, dx = \lam \int _\Omega uv \, dx .
\]
Assume $u$ and $v$ are real-valued. Green's theorem and the boundary condition on $v$ imply
\[
\int_\Omega \nabla u \cdot \nabla v \, dx = \lam \la u , v
\ra_{L^2(\Omega)} , \qquad \forall v \in H^1_0(\Omega).
\]
We call this condition the ``weak form'' of the eigenfunction equation. To
prove existence of ONBs of such weak eigenfunctions, we will generalize
to a Hilbert space problem.

Notice the left side of the weak eigenfunction equation, $\int_\Omega \nabla u \cdot \nabla v \, dx$, is  symmetric with respect to $u$ and $v$.

\subsection*{Hypotheses} Consider two infinite dimensional Hilbert spaces
$\mathcal{H}$ and $\mathcal{K}$ over $\R$ (or $\C$).

$\mathcal{H}$: inner product $\la u , v \ra_\mathcal{H}$, norm $\lv u
\rv_\mathcal{H}$

$\mathcal{K}$: inner product $\la u , v \ra_\mathcal{K}$, norm $\lv u
\rv_\mathcal{K}$

\medskip
\noindent Assume:

1. $\mathcal{K}$ is continuously and densely imbedded in $\mathcal{H}$,
meaning there exists a continuous linear injection $\iota : \mathcal{K}
\to \mathcal{H}$ with $\iota(\mathcal{K})$ dense in $\mathcal{H}$.

2. The imbedding $\mathcal{K} \hookrightarrow \mathcal{H}$ is
\textbf{compact}, meaning if $B$ is a bounded subset of $\mathcal{K}$ then
$B$ is precompact when considered as a subset of $\mathcal{H}$.
(Equivalently, every bounded sequence in $\mathcal{K}$ has a subsequence
that converges in $\mathcal{H}$.)

3. We have a map $a : \mathcal{K} \times \mathcal{K} \to \R$ (or $\C$)
that is sesquilinear, continuous, and \textbf{symmetric}, meaning
\begin{align*}
u \mapsto a(u,v) \ & \text{is linear, for each fixed $v$,} \\
v \mapsto a(u,v) \ & \text{is linear (or conjugate linear), for each fixed
$u$,} \\
|a(u,v)| & \leq (\text{const.}) \lv u \rv_\mathcal{K} \lv v
\rv_\mathcal{K} \\
a(v,u) & = a(u,v) \quad \text{(or $\overline{a(u,v)}$)}
\end{align*}

4. $a$ is \textbf{elliptic} on $\mathcal{K}$, meaning
\[
a(u,u) \geq c \lv u \rv_\mathcal{K}^2 \qquad \forall u \in \mathcal{K} ,
\]
for some $c>0$. Hence $a(u,u) \asymp \lv u \rv_\mathcal{K}^2$.

\medskip
An important consequence of symmetry and ellipticity is that:
\begin{quote}
$a(u,v)$ defines an inner product whose norm is equivalent to the $\lv
\cdot \rv_\mathcal{K}$-norm.
\end{quote}

\subsection*{Spectral theorem}
\begin{theorem} \label{th:spec}
Under the hypotheses above, there exist vectors $u_1,u_2,u_3,\ldots \in
\mathcal{K}$ and numbers
\[
0 < \gamma_1 \leq \gamma_2 \leq \gamma_3 \leq \cdots \to \infty
\]
such that:
\begin{itemize}
\item$u_j$ is an eigenvector of $a(\cdot,\cdot)$ with eigenvalue
$\gamma_j$, meaning
\begin{equation} \label{eONB:eig}
a(u_j,v) = \gamma_j \la u_j , v \ra_\mathcal{H} \qquad \forall v \in
\mathcal{K} ,
\end{equation}
\item $\{ u_j \}$ is an ONB for $\mathcal{H}$,
\item $\{ u_j/\sqrt{\gamma_j} \}$ is an ONB for $\mathcal{K}$ with respect
to the $a$-inner product.
\end{itemize}
The decomposition
\begin{equation} \label{eq:orthogtwice}
f = \sum_j \la f , u_j \ra_\mathcal{H} \, u_j
\end{equation}
converges in $\mathcal{H}$ for each $f \in \mathcal{H}$, and converges in
$\mathcal{K}$ for each $f \in \mathcal{K}$.
\end{theorem}
The idea is to show that a certain ``inverse'' operator associated with
$a$ is compact and selfadjoint on $\mathcal{H}$. This approach makes sense
in terms of differential equations, where $a$ would correspond to a
differential operator such as $-\Delta$ (which is unbounded) and the
inverse would correspond to an integral operator $(-\Delta)^{-1}$ (which
is bounded, and in fact compact, on suitable domains). Indeed, we will
begin by solving the analogue of $-\Delta u=f$ weakly in our Hilbert space
setting, with the help of the Riesz Representation Theorem.

\medskip \noindent \emph{Remark.} For a more general spectral theorem, readers may consult the recent paper \cite{Au} and certain references therein. Briefly, the eigenvectors there satisfy $a(u_j,v) = \gamma_j b(u_j,v)$ for all $v \in \mathcal{K}$, where the bilinear form $b$ is assumed to be weakly continuous on $\mathcal{K} \times \mathcal{K}$. In our situation, $b(u,v)=\la u,v \ra_\mathcal{H}$ and our assumption that $\mathcal{K}$ imbeds compactly into $\mathcal{H}$ implies weak continuity of $b$.

\medskip
\noindent \emph{Proof of Theorem~\ref{th:spec}.}
We first claim that for each $f \in \mathcal{H}$ there exists a unique $u
\in \mathcal{K}$ such that
\begin{equation} \label{eONB:weak}
a(u,v) = \la f , v \ra_\mathcal{H} \qquad \forall v \in \mathcal{K} .
\end{equation}
Furthermore, the map
\begin{align*}
B : & \mathcal{H} \to \mathcal{K} \\
& f \mapsto u
\end{align*}
is linear and bounded. To prove this claim, fix $f \in \mathcal{H}$ and
define a bounded linear functional $F(v)=\la v , f \ra_\mathcal{H}$ on
$\mathcal{K}$, noting for the boundedness that
\begin{align*}
|F(v)|
& \leq \lv v \rv_\mathcal{H} \lv f \rv_\mathcal{H} \\
& \leq (\text{const.}) \lv v \rv_\mathcal{K} \lv f \rv_\mathcal{H} \qquad
\text{since $\mathcal{K}$ is imbedded in $\mathcal{H}$} \\
& \leq (\text{const.}) a(v,v)^{1/2} \lv f \rv_\mathcal{H}
\end{align*}
by ellipticity. Hence by the Riesz Representation Theorem on $\mathcal{K}$
(with respect to the $a$-inner product and norm on $\mathcal{K}$), there
exists a unique $u \in \mathcal{K}$ such that $F(v) = a(v,u)$ for all $v
\in \mathcal{K}$. That is,
\[
\la v , f \ra_\mathcal{H} = a(v,u) \qquad \forall v \in \mathcal{K} ,
\]
as desired for \eqref{eONB:weak}. Thus the map $B:f \mapsto u$ is well
defined. Clearly it is linear. And
\[
a(u,u) = |F(u)| \leq (\text{const.}) a(u,u)^{1/2} \lv f \rv_\mathcal{H} .
\]
Hence $a(u,u)^{1/2} \leq (\text{const.}) \lv f \rv_\mathcal{H}$, so that
$B$ is bounded from $\mathcal{H}$ to $\mathcal{K}$, which proves our
initial claim.

Next, $B:\mathcal{H} \to \mathcal{K} \to \mathcal{H}$ is compact, since
$\mathcal{K}$ imbeds compactly into $\mathcal{H}$. Further, $B$ is
selfadjoint on $\mathcal{H}$, since for all $f,g \in \mathcal{H}$ we have
\begin{align*}
\la Bf , g \ra_\mathcal{H}
& = \overline{\la g , Bf \ra_\mathcal{H}} \\
& = \overline{a(Bg,Bf)} && \text{by definition of $B$,} \\
& = a(Bf,Bg) && \text{by symmetry of $a$,} \\
& = \la f , Bg \ra_\mathcal{H} && \text{by definition of $B$,}
\end{align*}
which implies $B^*=B$.

Hence the spectral theorem for compact, self-adjoint operators
\cite[App.~D]{E} provides an ONB for $\mathcal{H}$ consisting of
eigenvectors of $B$, with
\[
Bu_j = \widetilde{\gamma}_j u_j
\]
for some eigenvalues $\widetilde{\gamma}_j \to 0$. The decomposition
\eqref{eq:orthogtwice} holds in $\mathcal{H}$  because $\{ u_j \}$ forms
an ONB for $\mathcal{H}$.

The eigenvalues of $B$ are all nonzero, because $B$ is injective: $Bf=0$
would imply $\la f , v \ra_\mathcal{H} = 0$ for all $v \in \mathcal{K}$ by
\eqref{eONB:weak}, so that $f=0$ (using density of $\mathcal{K}$ in
$\mathcal{H}$). Thus $\widetilde{\gamma_j} \neq 0$.

Since we may divide by the eigenvalue we deduce that $u_j=
B(u_j/\widetilde{\gamma_j})$. Thus $u_j$ belongs to the range of $B$, and
so $u_j \in \mathcal{K}$.

The eigenvalues are all positive, since
\[
\widetilde{\gamma}_j a(u_j,v) = a(Bu_j,v) = \la u_j , v \ra_\mathcal{H}
\qquad \forall v \in \mathcal{K}
\]
and choosing $v=u_j \in \mathcal{K}$ and using ellipticity shows that
$\widetilde{\gamma}_j>0$. Thus we see that the reciprocal numbers
$0<\gamma_j \overset{\text{def}}{=} 1/\widetilde{\gamma}_j \to \infty$
satisfy
\[
a(u_j,v) = \gamma_j \la u_j , v \ra_\mathcal{H} \qquad \forall v \in
\mathcal{K} ,
\]
which is \eqref{eONB:eig}.

Finally, we have $a$-orthonormality of the set $\{ u_j/\sqrt{\gamma_j} \}$:
\begin{align*}
a(u_j,u_k)
& = \gamma_j \la u_j , u_k \ra_\mathcal{H} \\
& = \gamma_j \delta_{jk} \\
& = \sqrt{\gamma_j} \sqrt{\gamma_k} \, \delta_{jk} .
\end{align*}
This orthonormal set is complete in $\mathcal{K}$, because if $a(u_j,v)=0$
for all $j$ then $\la u_j , v \ra_\mathcal{H}=0$ for all $j$, by
\eqref{eONB:eig}, so that $v=0$. Therefore each $f \in \mathcal{K}$ can be
decomposed as
\[
f = \sum_j a(f,u_j/\sqrt{\gamma_j}) \, u_j/\sqrt{\gamma_j}
\]
with convergence in $\mathcal{K}$, and this decomposition reduces to
\eqref{eq:orthogtwice} because $a(f,u_j)=\gamma_j \la f,u_j
\ra_\mathcal{H}$.
\qed

\paragraph*{Remark.} Eigenvectors corresponding to distinct eigenvalues
are automatically orthogonal, since
\begin{align*}
(\gamma_j - \gamma_k) \la u_j , u_k \ra_\mathcal{H}
& = \gamma_j \la u_j , u_k \ra_\mathcal{H} - \overline{\gamma_k \la u_k ,
u_j \ra_\mathcal{H}} \\
& = a(u_j,u_k) - \overline{a(u_k,u_j)} \\
& = 0
\end{align*}
by symmetry of $a$.

\chapter[Laplace eigenfunctions]{Application: ONBs of Laplace eigenfunctions} \label{ch:aONB}

\subsubsection*{Goal} To apply the spectral theorem from the previous chapter to the Dirichlet, Robin and Neumann Laplacians, and to the fourth order biLaplacian.

\subsection*{Laplacian}

\paragraph*{Dirichlet Laplacian}
\begin{align*}
-\Delta u & = \lam u \qquad \text{in $\Omega$} \\
u & = 0 \qquad \ \ \text{on $\partial \Omega$}
\end{align*}

$\Omega=$bounded domain in $\Rd$.

$\mathcal{H}=L^2(\Omega)$, inner product $\la u , v \ra_{L^2} = \int_\Omega uv \, dx$.

$\mathcal{K}=H^1_0(\Omega)=$Sobolev space, which is the completion of $C^\infty_0(\Omega)$ (smooth functions equalling zero on a neighborhood of the boundary) under the inner product
\[
\la u , v \ra_{H^1} = \int_\Omega [\nabla u \cdot \nabla v + uv] \, dx.
\]

Density: $H^1_0$ contains $C^\infty_0$, which is dense in $L^2$.

Continuous imbedding $H^1_0 \hookrightarrow L^2$ is trivial:
\begin{align*}
\lv u \rv_{L^2}
& = \Big( \int_\Omega u^2 \, dx \Big)^{\! 1/2} \\
& \leq \Big( \int_\Omega [ \, |\nabla u|^2 +  u^2] \, dx \Big)^{\! 1/2} \\
& = \lv u \rv_{H^1}
\end{align*}

Compact imbedding: $H^1_0 \hookrightarrow L^2$ compactly by the Rellich--Kondrachov Theorem \cite[Theorem 7.22]{GT}.

Sesquilinear form: define
\[
a(u,v) = \int_\Omega \nabla u \cdot \nabla v \, dx + \int_\Omega uv \, dx = \la u , v \ra_{H^1}, \qquad u,v \in H^1_0(\Omega) .
\]
Clearly $a$ is symmetric and continuous on $H^1_0(\Omega)$.

Ellipticity: $a(u,u) = \lv u \rv_{H^1}^2$

\medskip
The discrete spectral Theorem~\ref{th:spec} gives an ONB $\{ u_j \}$ for $L^2(\Omega)$ and corresponding eigenvalues which we denote $\gamma_j=\lam_j+1>0$ satisfying
\[
\la u_j, v \ra_{H^1} = (\lam_j + 1) \la u_j , v \ra_{L^2} \qquad \forall v \in H^1_0(\Omega) .
\]
Equivalently,
\[
\int_\Omega \nabla u_j \cdot \nabla v \, dx = \lam_j \int_\Omega u_j v \, dx \qquad \forall v \in H^1_0(\Omega) .
\]
That is,
\[
-\Delta u_j = \lam_j u_j
\]
weakly, so that $u_j$ is a \emph{weak eigenfunction} of the Laplacian with eigenvalue $\lam_j$. Elliptic regularity theory gives that $u_j$ is $C^\infty$-smooth in $\Omega$ \cite[Corollary 8.11]{GT}, and hence satisfies the eigenfunction equation classically. The boundary condition $u_j=0$ is satisfied in the sense of Sobolev spaces (since $H^1_0$ is the closure of $C^\infty_0$), and is satisfied classically on any smooth portion of $\partial \Omega$, again by elliptic regularity.

The eigenvalues are nonnegative, with
\[
\lam_j = \frac{\int_\Omega |\nabla u_j|^2 \, dx}{\int_\Omega u_j^2 \, dx} \geq 0 ,
\]
as we see by choosing $v=u_j$ in the weak formulation.

Further, $\lam_j>0$ because: if $\lam_j=0$ then $|\nabla u_j| \equiv 0$ by the last formula, so that $u_j \equiv 0$ by the Sobolev inequality for $H^1_0$ \cite[Theorem 7.10]{GT}, but $u_j$ cannot vanish identically because it has $L^2$-norm equal to $1$. Hence
\[
0<\lam_1 \leq \lam_2 \leq \lam_3 \leq \cdots \to \infty .
\]

\medskip
\emph{Aside.} The Sobolev inequality we used is easily proved: for $u \in H^1_0(\Omega)$,
\begin{align*}
\lv u \rv_{L^2}^2
& = \int_\Omega u^2 \, dx \\
& = - \int_\Omega 2 x_i u \frac{\partial u}{\partial x_i} \, dx \qquad \text{by parts} \\
& \leq 2 (\max_{x \in \overline{\Omega}} |x|) \lv u \rv_{L^2} \lv \partial u/\partial x_i \rv_{L^2} \\
& \leq (\text{const.}) \lv u \rv_{L^2} \lv \nabla u \rv_{L^2}
\end{align*}
so that we have a Sobolev inequality
\[
\lv u \rv_{L^2} \leq (\text{const.}) \lv \nabla u \rv_{L^2} \qquad \forall u \in H^1_0(\Omega) ,
\]
where the constant depends on the domain $\Omega$. Incidentally, this Sobolev inequality provides another proof that $\lam_j>0$ for the Dirichlet Laplacian.

\paragraph*{Neumann Laplacian}
\begin{align*}
-\Delta u & = \mu u \qquad \text{in $\Omega$} \\
\frac{\partial u}{\partial n} & = 0 \qquad \ \ \text{on $\partial \Omega$}
\end{align*}

$\Omega=$bounded domain in $\Rd$ with Lipschitz boundary.

$\mathcal{H}=L^2(\Omega)$

$\mathcal{K}=H^1(\Omega)=$Sobolev space, which is the completion of $C^\infty(\overline{\Omega})$ under the inner product $\la u , v \ra_{H^1}$ (see \cite[p.\,174]{GT}).

Argue as for the Dirichlet Laplacian. The compact imbedding is provided by the Rellich--Kondrachov Theorem \cite[Theorem 7.26]{GT}, which relies on Lipschitz smoothness of the boundary.

One writes the eigenvalues in the discrete spectral Theorem~\ref{th:spec} as $\gamma_j=\mu_j+1>0$ and finds
\begin{equation} \label{eq:morethanweak}
\int_\Omega \nabla u_j \cdot \nabla v \, dx = \mu_j \int_\Omega u_j v \, dx \qquad \forall v \in H^1(\Omega) ,
\end{equation}
which implies that
\[
-\Delta u_j = \mu_j u_j
\]
weakly (and hence classically). In fact \eqref{eq:morethanweak} says a little more, because it holds for all $v \in H^1(\Omega)$, not just for $v \in H^1_0(\Omega)$ as needed for a weak solution. We will use this additional information in the next chapter to show that eigenfunctions automatically satisfy the Neumann boundary condition (even though we never imposed it)!

Choosing $v=u_j$ proves $\mu_j \geq 0$. The first Neumann eigenvalue is zero: $\mu_1=0$, with a constant eigenfunction $u_1 \equiv \text{const.} \neq 0$. (This constant function belongs to $H^1(\Omega)$, although not to $H^1_0(\Omega)$.) Hence
\[
0=\mu_1 \leq \mu_2 \leq \mu_3 \leq \cdots \to \infty .
\]

\paragraph*{Robin Laplacian}

\begin{align*}
-\Delta u & = \rho u \qquad \text{in $\Omega$} \\
\frac{\partial u}{\partial n} + \sigma u & = 0 \qquad \ \ \text{on $\partial \Omega$}
\end{align*}

$\Omega=$bounded domain in $\Rd$ with Lipschitz boundary.

$\mathcal{H}=L^2(\Omega)$

$\mathcal{K}=H^1(\Omega)$

$\sigma > 0$ is the Robin constant.

The density and compact imbedding conditions are as in the Neumann case above.




Before defining the sesquilinear form, we need to make sense of the boundary values of $u$. Sobolev functions do have well defined boundary values. More precisely, there is a bounded linear operator (called the \emph{trace operator}) $T : H^1(\Omega) \to L^2(\partial \Omega)$ such that
\begin{equation} \label{eq:trace-est}
\lv Tu \rv_{L^2(\partial \Omega)} \leq \tau \lv u \rv_{H^1(\Omega)}
\end{equation}
for some $\tau>0$, and with the property that if $u$ extends to a continuous function on $\partial \Omega$, then $Tu=u$ on $\partial \Omega$. (Thus the trace operator truly captures the boundary values of $u$.) Further, if $u \in H^1_0(\Omega)$ then $Tu=0$, meaning that functions in $H^1_0$ ``equal zero on the boundary''. For these trace results, see \cite[Section 5.5]{E} for domains with $C^1$ boundary, or \cite[{\S}4.3]{EG} for the slightly rougher case of Lipschitz boundary.

Sesquilinear form:
\[
a(u,v) = \int_\Omega \nabla u \cdot \nabla v \, dx + \sigma \int_{\partial \Omega} uv \, dS(x) + \int_\Omega uv \, dx
\]
(where $u$ and $v$ on the boundary should be interpreted as the trace values $Tu$ and $Tv$). Clearly $a$ is symmetric and continuous on $H^1(\Omega)$.

Ellipticity: $a(u,u) \geq \lv u \rv_{H^1}^2$, since $\sigma>0$.

\medskip
One writes the eigenvalues in the discrete spectral Theorem~\ref{th:spec} as $\gamma_j=\rho_j+1>0$ and finds
\[
\int_\Omega \nabla u_j \cdot \nabla v \, dx + \sigma \int_{\partial \Omega} u_j v \, dS(x) = \rho_j \int_\Omega u_j v \, dx \qquad \forall v \in H^1(\Omega) ,
\]
which implies that
\[
-\Delta u_j = \rho_j u_j
\]
weakly and hence classically. For the weak solution here we need (by definition) only to use trial functions $v \in H^1_0(\Omega)$ (functions equalling zero on the boundary). In the next chapter we use the full class $v \in H^1(\Omega)$ to show that the eigenfunctions satisfy the Robin boundary condition.

Choosing $v=u_j$ proves
\[
\rho_j = \frac{\int_\Omega |\nabla u_j|^2 \, dx + \sigma \int_{\partial \Omega} u_j^2 \, dS(x)}{\int_\Omega u_j^2 \, dx} \geq 0 ,
\]
using again that $\sigma>0$. Further, $\rho_j>0$ because: if $\rho_j=0$ then $|\nabla u_j| \equiv 0$ so that $u_j \equiv \text{const.}$, and this constant must equal zero because $\int_{\partial \Omega} u_j^2 \, dS(x) = 0$; but $u_j$ cannot vanish identically because it has $L^2$-norm equal to $1$. Hence when $\sigma>0$ we have
\[
0<\rho_1 \leq \rho_2 \leq \rho_3 \leq \cdots \to \infty .
\]

\paragraph*{Negative Robin constant: $\sigma<0$.} Ellipticity more difficult to prove when $\sigma<0$. We start by controlling the boundary values in terms of the gradient and $L^2$ norm. We have
\[
\int_{\partial \Omega} u^2 \, dS(x) \leq (\text{const.}) \int_\Omega \lvert \nabla u \rvert \lvert u \rvert \, dx + (\text{const.}) \int_\Omega u^2 \, dx ,
\]
as one sees by inspecting the proof of the trace theorem (\cite[{\S}5.5]{E} or \cite[{\S}4.3]{EG}). An application of Cauchy-with-$\e$ gives
\[
\int_{\partial \Omega} u^2 \, dS(x) \leq \e \lv \nabla u \rv_{L^2}^2 + C \lv u \rv_{L^2}^2
\]
for some constant $C=C(\e)>0$ (independent of $u$). Let us choose $\e = 1/2|\sigma|$, so that
\[
a(u,u) \geq \frac{1}{2} \lv u \rv_{H^1}^2 - C|\sigma| \, \lv u \rv_{L^2}^2 .
\]
Hence the new sesquilinear form $\widetilde{a}(u,v)=a(u,v)+C|\sigma| \la u,v \ra_{L^2}$ is elliptic. We apply the discrete spectral theorem to this new form, and then obtain the eigenvalues of $a$ by subtracting $C|\sigma|$ (with the same ONB of eigenfunctions).

\subsubsection*{Eigenfunction expansions in the $L^2$ and $H^1$ norms}

The $L^2$-ONB of eigenfunctions $\{ u_j \}$ of the Laplacian gives the decomposition
\begin{equation} \label{eq:ONBH1}
f = \sum_j \la f , u_j \ra_{L^2} \, u_j
\end{equation}
with convergence in the $L^2$ and $H^1$ norms, for all $f$ in the following spaces:
\[
f \in \begin{cases}
H^1_0(\Omega) & \text{for Dirichlet,} \\
H^1(\Omega) & \text{for Neumann,} \\
H^1(\Omega) & \text{for Robin.}
\end{cases}
\]
These claims follow immediately from the discrete spectral Theorem~\ref{th:spec}, in view of our applications above.

\subsubsection*{Invariance of eigenvalues under translation, rotation and reflection, and scaling under dilation}

Eigenvalues of the Laplacian remain invariant when the domain $\Omega$ is translated, rotated or reflected, as one sees by a straightforward change of variable in either the classical or weak formulation of the eigenvalue problem. Physically, this invariance simply means that a vibrating membrane is unaware of any coordinate system we impose upon it.

Dilations do change the eigenvalues, of course, by a simple rescaling relation: $\lam_j(t \Omega) = t^{-2} \lam_j(\Omega)$ for each $j$ and all $t>0$, and similarly for the Neumann eigenvalues. (We can understand this scale factor $t^{-2}$ physically, by recalling that large drums vibrate with low tones.) The Robin eigenvalues rescale the same way under dilation, provided the Robin parameter is rescaled to $\sigma/t$ on the domain $t\Omega$.

\subsection*{BiLaplacian --- vibrating plates}

The fourth order wave equation $\phi_{tt} = - \Delta \Delta \phi$ describes the transverse vibrations of a rigid plate. (In one dimension, this equation simplifies to the \emph{beam equation}: $\phi_{tt} = - \phi^{\prime \prime \prime \prime}$). After separating out the time variable, one arrives at the eigenvalue problem for the biLaplacian:
\[
\Delta \Delta u = \Lambda u \qquad \text{in $\Omega$.}
\]
We will prove existence of an orthonormal basis of eigenfunctions. For simplicity, we treat only the Dirichlet case, which has boundary conditions
\[
u=|\nabla u|=0 \qquad \text{on $\partial \Omega$.}
\]
(The Neumann ``natural'' boundary conditions are rather complicated, for the biLaplacian.)

\bigskip
$\Omega=$bounded domain in $\Rd$

$\mathcal{H}=L^2(\Omega)$

$\mathcal{K}=H^2_0(\Omega)=$ completion of $C^\infty_0(\Omega)$ under the inner product
\[
\la u , v \ra_{H^2} = \int_\Omega [ \, \sum_{m,n=1}^d u_{x_m x_n} v_{x_m x_n} + \sum_{m=1}^d u_{x_m} v_{x_m} + uv \, ] \, dx.
\]

Density: $H^2_0$ contains $C^\infty_0$, which is dense in $L^2$.

Compact imbedding: $H^2_0 \hookrightarrow H^1_0 \hookrightarrow L^2$ and the second imbedding is compact.

Sesquilinear form: define
\[
a(u,v) = \int_\Omega [ \, \sum_{m,n=1}^d u_{x_m x_n} v_{x_m x_n}  + uv \, ] \, dx, \qquad u,v \in H^2_0(\Omega) .
\]
Clearly $a$ is symmetric and continuous on $H^2_0(\Omega)$.

Ellipticity: $\lv u \rv_{H^2}^2 \leq (d+1) a(u,u)$, because integration by parts gives
\begin{align*}
\int_\Omega \sum_{m=1}^d u_{x_m}^2 \, dx
& = -\sum_{m=1}^d \int_\Omega u_{x_m x_m} u \, dx \\
& \leq \sum_{m=1}^d \int_\Omega [ \, u_{x_m x_m}^2 + u^2 \, ] \, dx \\
& \leq a(u,u) d .
\end{align*}

\medskip
The discrete spectral Theorem~\ref{th:spec} gives an ONB $\{ u_j \}$ for $L^2(\Omega)$ and corresponding eigenvalues which we denote $\gamma_j=\Lambda_j+1>0$ satisfying
\[
a(u_j, v) = (\Lambda_j + 1) \la u_j , v \ra_{L^2} \qquad \forall v \in H^2_0(\Omega) .
\]
Equivalently,
\[
\int_\Omega \sum_{m,n=1}^d (u_j)_{x_m x_n} v_{x_m x_n} \, dx = \Lambda_j \int_\Omega u_j v \, dx \qquad \forall v \in H^2_0(\Omega) .
\]
That is,
\[
\sum_{m,n=1}^d  (u_j)_{x_m x_m x_n x_n} = \Lambda_j u_j
\]
weakly, which says
\[
\Delta \Delta u_j = \Lambda_j u_j
\]
weakly. Hence $u_j$ is a \emph{weak eigenfunction} of the biLaplacian with eigenvalue $\Lambda_j$. Elliptic regularity  gives that $u_j$ is $C^\infty$-smooth, and hence satisfies the eigenfunction equation classically. The boundary condition $u_j=|\nabla u_j|=0$ is satisfied in the sense of Sobolev spaces (since $u_j$ and each partial derivative $(u_j)_{x_m}$ belong to $H^1_0$), and the boundary condition is satisfied classically on any smooth portion of $\partial \Omega$, again by elliptic regularity.

The eigenvalues are nonnegative, with
\[
\Lambda_j = \frac{\int_\Omega |D^2 u_j|^2 \, dx}{\int_\Omega u_j^2 \, dx} \geq 0 ,
\]
as we see by choosing $v=u_j$ in the weak formulation and writing $D^2 u = [u_{x_m x_n}]_{m,n=1}^d$ for the Hessian matrix.

Further, $\Lambda_j>0$ because: if $\Lambda_j=0$ then $(u_j)_{x_m x_n} \equiv 0$ by the last formula, so that $(u_j)_{x_m} \equiv 0$ by the Sobolev inequality for $H^1_0$ applied to $(u_j)_{x_m}$, and hence $u_j \equiv 0$ by the same Sobolev inequality, which gives a contradiction. Hence
\[
0<\Lambda_1 \leq \Lambda_2 \leq \Lambda_3 \leq \cdots \to \infty .
\]

\subsection*{Compact resolvents}

The essence of the proof of the discrete spectral Theorem~\ref{th:spec} is to show that the inverse operator $B$ is compact, which means for our differential operators that the inverse is a compact integral operator. For example, in the Neumann Laplacian application we see that $(-\Delta+1)^{-1}$ is compact from $L^2(\Omega)$ to $H^1(\Omega)$. So is $(-\Delta+\alpha)^{-1}$ for any positive $\alpha$, but $\alpha=0$ does not give an invertible operator because the Neumann Laplacian has nontrivial kernel, with $-\Delta(c)=0$ for every constant $c$.

Thus for the Neumann Laplacian, the \emph{resolvent operator}
\[
R_\lam = (-\Delta-\lam)^{-1}
\]
is compact whenever $\lam$ is negative.

\chapter{Natural boundary conditions} \label{ch:NBC}

\subsubsection*{Goal} To understand how the Neumann and Robin boundary conditions arise
``naturally'' from the weak eigenfunction equation.

\paragraph*{Dirichlet boundary conditions} \

\noindent are imposed directly by our choice of function space $H^1_0(\Omega)$, since each function in that space is a limit of functions with compact support in $\Omega$.

\paragraph*{Neumann boundary conditions} \

The weak form of the Neumann eigenequation for the Laplacian, from Chapter~\ref{ch:aONB}, is:
\begin{equation} \label{eq:NBC1}
\int_\Omega \nabla u \cdot \nabla v \, dx = \mu \int_\Omega u v \, dx
\qquad \forall v \in H^1(\Omega) .
\end{equation}
From this formula we showed that $-\Delta u = \mu u$ weakly and hence classically, by using only functions $v$ that vanish on the boundary, meaning $v \in H^1_0(\Omega)$.

To deduce the Neumann boundary condition $\partial u/\partial n = 0$, we will take $v$ \emph{not} to vanish on the boundary. Assume for simplicity that the boundary is smooth, so that $u$ extends smoothly to $\overline{\Omega}$. Green's formula (integration by parts) applied to \eqref{eq:NBC1} implies that
\[
\int_\Omega (-\Delta u) v \, dx + \int_{\partial \Omega} \frac{\partial
u}{\partial n} v \, dS = \int_\Omega (\mu u) v \, dx \qquad \forall v \in
C^\infty(\overline{\Omega}) .
\]
Since $-\Delta u = \mu u$, we deduce
\[
\int_{\partial \Omega} \frac{\partial u}{\partial n} v \, dS = 0 \qquad
\forall v \in C^\infty(\overline{\Omega}) .
\]
One may choose $v \in C^\infty(\overline{\Omega})$ to equal the normal derivative of $u$ on the boundary (meaning $v|_{\partial \Omega} = \tfrac{\partial
u}{\partial n}$), or alternatively one may use density of $C^\infty(\overline{\Omega}) \big|_{\partial \Omega}$ in $L^2(\partial \Omega)$; either way one concludes that
\[
\frac{\partial u}{\partial n} = 0 \qquad \text{on $\partial \Omega$,}
\]
which is the Neumann boundary condition.

\emph{Note.} If the boundary is only piecewise smooth, then one merely applies the above reasoning on the smooth portions of the boundary, to show the Neumann condition holds there.

\paragraph*{Robin boundary conditions} \

Integrating by parts in the Robin eigenfunction equation
\[
\int_\Omega \nabla u \cdot \nabla v \, dx + \sigma \int_{\partial \Omega}
uv \, dS = \rho \int_\Omega u v \, dx \qquad \forall v \in H^1(\Omega)
\]
(that is, applying Green's formula to this equation) and then using that $-\Delta u = \rho u$ gives that
\[
\int_{\partial \Omega} \big( \frac{\partial u}{\partial n} + \sigma u
\big) v \, dS = 0 \qquad \forall v \in C^\infty(\overline{\Omega})
.
\]
Like above, we obtain the Robin boundary condition
\[
\frac{\partial u}{\partial n} + \sigma u = 0 \qquad \text{on $\partial
\Omega$,}
\]
at least on smooth portions of the boundary.

\paragraph*{BiLaplacian --- natural boundary conditions} \

Natural boundary conditions for the biLaplacian can be derived similarly \cite[{\S}5]{C}. They are much more complicated than for the Laplacian.

\chapter[Magnetic Laplacian]{Application: ONB of eigenfunctions for the
Laplacian with magnetic field} \label{ch:maglaplace}

\subsubsection*{Goals} To apply the spectral theorem from
Chapter~\ref{ch:eONB} to the magnetic Laplacian (the
Schr\"{o}dinger operator for a particle in the presence of a classical magnetic field).

\subsubsection*{Magnetic Laplacian} \

Take a bounded domain $\Omega$ in $\Rd$, with $d=2$ or $d=3$. We seek an
ONB of eigenfunctions and eigenvalues for the \emph{magnetic Laplacian}
\begin{align*}
(i\nabla + \vec{A})^2 \, u & = \beta u \qquad \text{in $\Omega$,} \\
u & = 0 \qquad \ \text{on $\partial \Omega$,}
\end{align*}
where $u(x)$ is complex-valued and
\[
\vec{A} : \Rd \to \Rd
\]
is a given bounded vector field.

Physically, $\vec{A}$ represents the \emph{vector potential}, whose curl
equals the magnetic field: $\nabla \times \vec{A} = \vec{B}$. Note that in
$2$ dimensions, one extends $\vec{A}=(A_1,A_2)$ to a $3$-vector
$(A_1,A_2,0)$ before taking the curl, so that the field
$\vec{B}=(0,0,\tfrac{\partial A_2}{\partial x_1}-\tfrac{\partial
A_1}{\partial x_2})$ cuts vertically through the plane of the domain.
For a brief explanation of how the magnetic Laplacian arises from the correspondence between classical energy functions and quantum mechanical Hamiltonians, see \cite[p.~173]{RS2}.

Now we choose the Hilbert spaces and sesquilinear form. Consider only the
Dirichlet boundary condition, for simplicity:

\bigskip
\noindent $\mathcal{H}=L^2(\Omega;\C)$ (complex valued functions), with inner product
\[
\la u , v \ra_{L^2} = \int_\Omega u \overline{v} \, dx .
\]

\noindent $\mathcal{K}=H^1_0(\Omega;\C)$ with inner product
\[
\la u , v \ra_{H^1} = \int_\Omega [ \, \nabla u \cdot \overline{\nabla v}
+ u \overline{v}  \, ] \, dx
\]

Density: $\mathcal{K}$ contains $C^\infty_0$, which is dense in $L^2$.

Continuous imbedding $H^1_0 \hookrightarrow L^2$ is trivial, since $\lv u
\rv_{L^2} \leq \lv u \rv_{H^1}$, and the imbedding is compact by the
Rellich--Kondrachov Theorem \cite[Theorem 7.22]{GT}.

\smallskip
Sesquilinear form: define
\[
a(u,v) = \int_\Omega (i\nabla + \vec{A}) u \cdot \overline{(i\nabla +
\vec{A})v} \, dx + C \int_\Omega u\overline{v} \, dx , \qquad u,v \in
H^1_0(\Omega;\C) ,
\]
with constant $C= \lv \vec{A} \rv_{L^\infty}^2+\tfrac{1}{2}$. Clearly $a$ is
symmetric and continuous on $H^1_0$.

Ellipticity:
\begin{align*}
a(u,u)
& = \int_\Omega \big[ \, |\nabla u|^2 + 2 \Re (i\nabla u \cdot \vec{A}
\overline{u}) + |\vec{A}|^2 |u|^2 + C|u|^2 \, \big] \, dx \\
& \geq \int_\Omega \big[ \, |\nabla u|^2 - 2 |\nabla u| |\vec{A}| |u| + 2
|\vec{A}|^2 |u|^2 + \frac{1}{2}|u|^2 \, \big] \, dx \\
& \geq \int_\Omega \big[ \, \frac{1}{2} |\nabla u|^2 + \frac{1}{2} |u|^2
\, \big] \, dx \\
& = \frac{1}{2} \lv u \rv_{H^1}^2
\end{align*}

\medskip
The discrete spectral Theorem~\ref{th:spec} gives an ONB $\{ u_j \}$ for
$L^2(\Omega;\C)$ and corresponding eigenvalues which we denote
$\gamma_j=\beta_j+C>0$ satisfying
\[
\int_\Omega (i\nabla + \vec{A}) u_j \cdot \overline{(i\nabla + \vec{A})v)}
\, dx = \beta_j \int_\Omega u_j \overline{v} \, dx \qquad \forall v \in
H^1_0(\Omega;\C) .
\]
In particular,
\[
(i\nabla + \vec{A})^2 \, u_j = \beta_j u_j
\]
weakly (and hence classically, assuming smoothness of the vector potential
$\vec{A}$), so that $u_j$ is an eigenfunction of the magnetic Laplacian
$(i\nabla + \vec{A})^2$ with eigenvalue $\beta_j$. We have
\[
\beta_1 \leq \beta_2 \leq \beta_3 \leq \cdots \to \infty .
\]
The eigenvalues satisfy
\[
\beta_j = \frac{\int_\Omega |(i\nabla + \vec{A}) u_j|^2 \, dx}{\int_\Omega
|u_j|^2 \, dx} ,
\]
as we see by choosing $v=u_j$ in the weak formulation. Hence the
eigenvalues are all nonnegative.

In fact $\beta_1>0$ if the magnetic field vanishes nowhere, as we will show by proving the contrapositive. If $\beta_1=0$ then $(i\nabla + \vec{A}) u_1 \equiv 0$, which implies $\vec{A} = -i\nabla \log u_1$ wherever $u_1$ is nonzero. Then $\nabla \times \vec{A}=0$ wherever $u_1$ is nonzero, since the curl of a gradient vanishes identically. (Here we assume $u_1$ is twice continuously differentiable.) Thus the magnetic field vanishes somewhere, as we wanted to show.

\smallskip
\emph{Aside.} The preceding argument works regardless of the boundary condition. In the case of Dirichlet boundary conditions, one need not assume the magnetic field is nonvanishing, because the above argument and the reality of $\vec{A}$ together imply that if $\beta_1=0$ then $|u_1|$ is constant, which is impossible since $u_1=0$ on the boundary.

\paragraph*{Gauge invariance} \

Many different vector potentials can generate the same magnetic field. For example, in $2$ dimensions the potentials
\[
\vec{A} = (0,x_1) , \qquad \vec{A} = (-x_2,0) , \qquad \vec{A} = \frac{1}{2}(-x_2,x_1) ,
\]
all generate the same (constant) magnetic field: $\nabla \times \vec{A} = (0,0,1)$. Indeed, adding any gradient vector $\nabla f$ to the potential leaves the magnetic field unchanged, since the curl of a gradient equals zero. This phenomenon goes by the name of \emph{gauge invariance}.

How is the spectral theory of the magnetic Laplacian affected by gauge invariance? The sesquilinear form definitely changes when we replace $\vec{A}$ with $\vec{A}+\nabla f$. Fortunately, the new eigenfunctions are related to the old by a unitary transformation, as follows. Suppose $f$ is $C^1$-smooth on the closure of the domain. For any trial function $u \in H^1_0(\Omega;\C)$ we note that the modulated function $e^{if}u$ also belongs to $H^1_0(\Omega;\C)$, and that
\[
(i\nabla + \vec{A})u = (i\nabla + \vec{A} + \nabla f)(e^{if}u) .
\]
Thus if we write $a$ for the original sesquilinear form and $\widetilde{a}$ for the analogous form coming from the vector potential $\vec{A}+\nabla f$, we deduce
\[
a(u,v) = \widetilde{a}(e^{if}u,e^{if}v)
\]
for all trial functions $u,v$. Since also $\la u,v \ra_{L^2} = \la e^{if}u,e^{if}v \ra_{L^2}$, we find that the ONB of eigenfunctions $u_j$ associated with $a$ transforms to an ONB of eigenfunctions $e^{if}u_j$ associated with $\widetilde{a}$. The eigenvalues (energy levels) $\beta_j$ are unchanged by this transformation.

For geometric invariance of the spectrum with respect to rotations, reflections and translations, and for a discussion of the Neumann and Robin situations, see
\cite[Appendix A]{LLR}.

\subsubsection{Higher dimensions}

In dimensions $d \geq 4$ we identify the vector potential $\vec{A} : \Rd
\to \Rd$ with a $1$-form
\[
A = A_1 \, dx_1 + \cdots + A_d \, dx_d
\]
and obtain the magnetic field from the exterior derivative:
\[
B = dA .
\]
Apart from that, the spectral theory proceeds as in dimensions $2$ and $3$.

\chapter[Schr\"{o}dinger in confining well]{Application: ONB of eigenfunctions for Schr\"{o}dinger in a confining well} \label{ch:aSch}

\subsubsection*{Goal} To apply the spectral theorem from Chapter~\ref{ch:eONB} to the harmonic oscillator and more general confining potentials in higher dimensions.

\paragraph*{Schr\"{o}dinger operator with potential growing to infinity} \

We treat a locally bounded, real-valued potential $V(x)$ on $\Rd$ that grows at infinity:
\[
-C \leq V(x) \to \infty \qquad \text{as $|x| \to \infty$,}
\]
for some constant $C>0$. For example, $V(x)=|x|^2$ gives the harmonic oscillator.

We aim to prove existence of an ONB of eigenfunctions and eigenvalues for
\begin{align*}
(-\Delta+V)u & = Eu \qquad \text{in $\Rd$} \\
u & \to 0 \qquad \ \text{as $|x| \to \infty$}
\end{align*}

\bigskip
$\Omega=\Rd$

$\mathcal{H}=L^2(\Rd)$, inner product $\la u , v \ra_{L^2} = \int_\Rd uv \, dx$.

$\mathcal{K}=H^1(\Rd) \cap L^2(|V| \, dx)$ under the inner product
\[
\la u , v \ra_\mathcal{K} = \int_\Rd [ \, \nabla u \cdot \nabla v + (1+|V|)uv \, ] \, dx.
\]

Density: $\mathcal{K}$ contains $C^\infty_0$, which is dense in $L^2$.

Continuous imbedding $\mathcal{K} \hookrightarrow L^2$ is trivial, since $\lv u \rv_{L^2} \leq \lv u \rv_\mathcal{K}$. To prove the imbedding is compact:

\emph{Proof that imbedding is compact.} Suppose $\{ f_k \}$ is a bounded sequence in $\mathcal{K}$, say with $\lv f_k \rv_\mathcal{K}
\leq M$ for all $k$. We must prove the existence of a subsequence
converging in $L^2(\Rd)$.

The sequence is bounded in $H^1(B(R))$ for each ball $B(R) \subset \Rd$
that is centered at the origin. Take $R=1$. The Rellich--Kondrachov theorem
provides a subsequence that converges in $L^2 \big( B(1) \big)$. Repeating
with $R=2$ provides a sub-subsequence converging in $L^2 \big( B(2)
\big)$. Continue in this fashion and then consider the diagonal
subsequence, to obtain a subsequence that converges in $L^2 \big( B(R)
\big)$ for each $R>0$.

We will show this subsequence converges in $L^2(\Rd)$. Denote it by $\{
f_{k_\ell} \}$. Let $\e>0$. Since $V(x)$ grows to infinity, we may choose
$R$ so large that $V(x) \geq 1/\e$ when $|x| \geq R$. Then
\begin{align*}
\int_{\Rd \setminus B(R)} f_{k_\ell}^2 \, dx
& \leq \e \int_{\Rd \setminus B(R)} f_{k_\ell}^2 V \, dx \\
& \leq \e \lv f_{k_\ell} \rv_\mathcal{K}^2 \\
& \leq \e M^2
\end{align*}
for all $\ell$. Since also $\{ f_{k_\ell} \}$ converges on $B(R)$, we have
\[
\limsup_{\ell,m \to \infty} \lv f_{k_\ell} - f_{k_m} \rv_{L^2(\Rd)}
= \limsup_{\ell,m \to \infty} \lv f_{k_\ell} - f_{k_m} \rv_{L^2(\Rd \setminus B(R))}
\leq 2\sqrt{\e} M .
\]
Therefore $\{ f_{k_\ell} \}$ is Cauchy in $L^2(\Rd)$, and hence converges.

\smallskip
Sesquilinear form: define
\[
a(u,v) = \int_\Rd [ \, \nabla u \cdot \nabla v + Vuv \, ] \, dx + (2C+1) \int_\Rd uv \, dx , \qquad u,v \in \mathcal{K} .
\]
Clearly $a$ is symmetric and continuous on $\mathcal{K}$.

Ellipticity: $a(u,u) \geq \lv u \rv_\mathcal{K}^2$, since $V+2C+1 \geq 1+|V|$.

\medskip
The discrete spectral Theorem~\ref{th:spec} gives an ONB $\{ u_j \}$ for $L^2(\Rd)$ and corresponding eigenvalues which we denote $\gamma_j=E_j+2C+1>0$ satisfying
\[
\int_\Rd [ \, \nabla u_j \cdot \nabla v + Vu_j v ] \, dx = E_j \int_\Rd u_j v \, dx \qquad \forall v \in \mathcal{K} .
\]
In particular,
\[
-\Delta u_j + Vu_j = E_j u_j
\]
weakly (and hence classically, assuming smoothness of $V$), so that $u_j$ is an eigenfunction of the Schr\"{o}dinger operator $-\Delta + V$, with eigenvalue $E_j$. We have
\[
E_1 \leq E_2 \leq E_3 \leq \cdots \to \infty .
\]

The boundary condition $u_j \to 0$ at infinity is interpreted to mean, more precisely, that $u_j$ belongs to the space $H^1(\Rd) \cap L^2(|V| \, dx)$. This condition suffices to rule out the existence of any other eigenvalues for the harmonic oscillator, for example, as one can show by direct estimation \cite{S}.

The eigenvalues satisfy
\[
E_j = \frac{\int_\Rd \big( \, |\nabla u_j|^2 + Vu_j^2 \, \big) \, dx}{\int_\Rd u_j^2 \, dx} ,
\]
as we see by choosing $v=u_j$ in the weak formulation. Hence if $V \geq 0$ then the eigenvalues are all positive.

\chapter[Variational characterizations]{Variational characterizations of eigenvalues}

\subsubsection*{Goal} To obtain minimax and maximin characterizations of the eigenvalues of the sesquilinear form in Chapter~\ref{ch:eONB}.

\paragraph*{References} \cite{B} Section III.1.2

\paragraph*{Motivation and hypotheses.} How can one estimate the eigenvalues if the spectrum cannot be computed explicitly?
We will develop two complementary variational characterizations of eigenvalues. The intuition for these characterizations comes from
the special case of eigenvalues of a Hermitian (or real symmetric) matrix $A$, for which the sesquilinear form is $a(u,v)=Au \cdot \overline{v}$
and the first eigenvalue is
\[
\gamma_1 = \min_{v \neq 0} \frac{Av \cdot \overline{v}}{v \cdot \overline{v}} .
\]

We will work under the assumptions of the discrete spectral theorem in Chapter~\ref{ch:eONB}, for the sesquilinear form $a$.
Recall the ordering
\[
\gamma_1 \leq \gamma_2 \leq \gamma_3 \leq \cdots \to \infty .
\]

\subsection*{Poincar\'{e}'s minimax characterization of the eigenvalues} \

Define the \emph{Rayleigh quotient} of $u$ to be
\[
\frac{a(u,u)}{\la u,u \ra_\mathcal{H}} .
\]
We claim $\gamma_1$ equals the minimum value of the Rayleigh quotient:
\begin{equation} \label{eq:vc1}
\boxed{\gamma_1 = \min_{f \in \mathcal{K} \setminus \{ 0 \}} \frac{a(f,f)}{\la f,f \ra_\mathcal{H}}.}
\end{equation}
This characterization of the first eigenvalue is the \emph{Rayleigh principle}.

More generally, each eigenvalue is given by a minimax formula known as the \emph{Poincar\'{e} principle}:
\begin{equation} \label{eq:vc2}
\boxed{\gamma_j = \min_S \max_{f \in S \setminus \{ 0 \}} \frac{a(f,f)}{\la f,f \ra_\mathcal{H}}}
\end{equation}
where $S$ ranges over all $j$-dimensional subspaces of $\mathcal{K}$.

\medskip \noindent \emph{Remark.} The Rayleigh and Poincar\'{e} principles provide \emph{upper bounds} on eigenvalues,
since they expresses $\gamma_j$ as a minimum. More precisely, we obtain an upper bound on $\gamma_j$ by choosing $S$ to be any $j$-dimensional subspace and evaluating the maximum of the Rayleigh quotient over $f \in S$.

\medskip
\noindent \emph{Proof of Poincar\'{e} principle.} First we prove the Rayleigh principle for the first eigenvalue. Let $f \in \mathcal{K}$. Then $f$ can be expanded in terms of the ONB of eigenvectors as
\[
f = \sum_j c_j u_j
\]
where $c_j = \la f,u_j \ra_\mathcal{H}$. This series converges in both $\mathcal{H}$ and $\mathcal{K}$ (as we proved in Chapter~\ref{ch:eONB}). Hence we may substitute it into the Rayleigh quotient to obtain
\begin{align}
\frac{a(f,f)}{\la f,f \ra_\mathcal{H}}
& = \frac{\sum_{j,k} c_j \overline{c_k} a(u_j,u_k)}{\sum_{j,k} c_j \overline{c_k} \la u_j,u_k \ra_\mathcal{H}} \notag \\
& = \frac{\sum_j |c_j|^2 \gamma_j}{\sum_j |c_j|^2} \label{eq:poincare1}
\end{align}
since the eigenvectors $\{ u_j \}$ are orthonormal in $\mathcal{H}$ and the collection $\{ u_j/\sqrt{\gamma_j} \}$ is $a$-orthonormal in $\mathcal{K}$ (that is, $a(u_j,u_k)=\gamma_j \delta_{jk}$).
The expression \eqref{eq:poincare1} is obviously greater than or equal to $\gamma_1$, with equality when $f=u_1$, and so we have proved the Rayleigh principle \eqref{eq:vc1}.

Next we prove the minimax formula \eqref{eq:vc2} for $j=2$. (We leave the case of higher $j$-values as an exercise.)
Choose $S=\{ c_1 u_1 + c_2 u_2 : c_1,c_2 \text{\ scalars} \}$ to be the span of the first two eigenvectors. Then
\[
\max_{f \in S \setminus \{ 0 \}} \frac{a(f,f)}{\la f,f \ra_\mathcal{H}} = \max_{(c_1,c_2) \neq (0,0)} \frac{\sum_{j=1}^2 |c_j|^2 \gamma_j}{\sum_{j=1}^2 |c_j|^2} = \gamma_2 .
\]
Hence the minimum on the right side of \eqref{eq:vc2} is $\leq \gamma_2$.

To prove the opposite inequality, consider an arbitrary $2$-dimensional subspace $S \subset \mathcal{K}$.
This subspace contains a nonzero vector $g$ that is orthogonal to $u_1$ (since given a basis
$\{ v_1,v_2 \}$ for the subspace, we can find scalars $d_1,d_2$ not both zero such that $g=d_1 v_1 + d_2 v_2$ satisfies
$0 = d_1 \la v_1,u_1 \ra_\mathcal{H} + d_2 \la v_2,u_1 \ra_\mathcal{H} = \la g,u_1 \ra_\mathcal{H}$). Then $c_1=0$ in the expansion for $g$, and so by \eqref{eq:poincare1},
\[
\frac{a(g,g)}{\la g,g \ra_\mathcal{H}} = \frac{\sum_{j=2}^\infty |c_j|^2 \gamma_j}{\sum_{j=2}^\infty |c_j|^2} \geq \gamma_2 .
\]
Hence
\[
\max_{f \in S \setminus \{ 0 \}} \frac{a(f,f)}{\la f,f \ra_\mathcal{H}} \geq \frac{a(g,g)}{\la g,g \ra_\mathcal{H}} \geq \gamma_2 ,
\]
which implies that the minimum on the right side of \eqref{eq:vc2} is  $\geq \gamma_2$.
\qed

\paragraph*{Variational characterization of eigenvalue sums.}
The sum of the first $n$ eigenvalues has a simple ``minimum'' characterization, similar to the Rayleigh principle
for the first eigenvalue, but now involving pairwise orthogonal trial functions:
\begin{align}
& \gamma_1 + \cdots + \gamma_n \label{eq:vc3} \\
& = \min \Big\{ \frac{a(f_1,f_1)}{\la f_1,f_1 \ra_\mathcal{H}} + \cdots + \frac{a(f_n,f_n)}{\la f_n,f_n \ra_\mathcal{H}} : f_j \in \mathcal{K} \setminus \{ 0 \}, \la f_j,f_k \ra_\mathcal{H} = 0 \text{\ when $j \neq k$} \Big\} . \notag
\end{align}
See Bandle's book for the proof and related results \cite[Section III.1.2]{B}.

\subsection*{Courant's maximin characterization} \

The eigenvalues are given also by a maximin formula known as the \emph{Courant principle}:
\begin{equation} \label{eq:vc4}
\boxed{\gamma_j = \max_S \min_{f \in S^\perp \setminus \{ 0 \}} \frac{a(f,f)}{\la f,f \ra_\mathcal{H}}}
\end{equation}
where this time $S$ ranges over all $(j-1)$-dimensional subspaces of $\mathcal{K}$.

\medskip \noindent \emph{Remark.} The Courant principle provide \emph{lower bounds} on eigenvalues,
since it expresses $\gamma_j$ as a maximum. The lower bounds are difficult to compute, however, because $S^\perp$ is
an infinite dimensional space.

\medskip
\noindent \emph{Sketch of proof of Courant principle.} The Courant principle reduces to
Rayleigh's principle when $j=1$, since in that case $S$ is the zero subspace and $S^\perp=\mathcal{K}$.

Now take $j=2$ (we leave the higher values of $j$ as an exercise). For the ``$\leq$'' direction of the proof, we
choose $S$ to be the $1$-dimensional space spanned by the first eigenvector $u_1$. Then every $f \in S^\perp$
has $c_1=\la f , u_1 \ra_\mathcal{H}=0$ and so
\[
\gamma_2 \leq \min_{f \in S^\perp \setminus \{ 0 \}} \frac{a(f,f)}{\la f,f \ra_\mathcal{H}}
\]
by expanding $f=\sum_{j=2}^\infty c_j u_j$ and computing as in our proof of the Poincar\'{e} principle.

For the ``$\geq$'' direction of the proof, consider an arbitrary $1$-dimensional subspace $S$ of $\mathcal{K}$. Then $S^\perp$ contains some vector of the form $f=c_1 u_1 + c_2 u_2$ with at least one of $c_1$ or $c_2$ nonzero. Hence
\[
\min_{f \in S^\perp \setminus \{ 0 \}} \frac{a(f,f)}{\la f,f \ra_\mathcal{H}} \leq \frac{\sum_{j=1}^2 |c_j|^2 \gamma_j}{\sum_{j=1}^2 |c_j|^2} \leq \gamma_2 ,
\]
as desired.

\subsection*{Eigenvalues as critical values of the Rayleigh quotient}
Even if we did not know the existence of an ONB of eigenvectors we could still prove the Rayleigh principle,
by the following direct approach. Define $\gamma^*$ to equal the infimum of the Rayleigh quotient:
\[
\gamma^* = \inf_{f \in \mathcal{K} \setminus \{ 0 \}} \frac{a(f,f)}{\la f,f \ra_\mathcal{H}} .
\]
We will prove $\gamma^*$ is an eigenvalue. It follows that $\gamma^*$ is the lowest eigenvalue, $\gamma^*=\gamma_1$ (because if any eigenvector $f$ corresponded to a smaller eigenvalue, then the Rayleigh quotient of $f$ would be smaller than $\gamma^*$, a contradiction).

First, choose an infimizing sequence $\{ f_k \}$ normalized with $\lv f_k \rv_\mathcal{H} = 1$, so that
\[
a(f_k,f_k) \to \gamma^* .
\]
By weak compactness of closed balls in the Hilbert space $\mathcal{K}$, we may suppose $f_k$ converges weakly
in $\mathcal{K}$ to some $u \in \mathcal{K}$. Hence $f_k$ also converges weakly in $\mathcal{H}$ to $u$ (because if
$F(\cdot)$ is any bounded linear functional on $\mathcal{H}$ then it is also a bounded linear functional on $\mathcal{K}$).
We may further suppose $f_k$ converges in $\mathcal{H}$ to some $v \in \mathcal{H}$ (by compactness of the imbedding $\mathcal{K} \hookrightarrow \mathcal{H}$)
and then $f_k$ converges weakly in $\mathcal{H}$ to $v$, which forces $v=u$. To summarize: $f_k \rightharpoonup u$
weakly in $\mathcal{K}$ and $f_k \to u$ in $\mathcal{H}$. In particular, $\lv u \rv_\mathcal{H} = 1$. Therefore we have
\begin{align*}
0
& \leq a(f_k-u,f_k-u) \\
& = a(f_k,f_k) -2 \Re a(f_k,u) + a(u,u) \\
& \to \gamma^* - 2 \Re a(u,u) + a(u,u) \qquad \text{using weak convergence $f_k \rightharpoonup u$} \\
& = \gamma^* - a(u,u) \\
& \leq 0
\end{align*}
by definition of $\gamma^*$ as an infimum.

We have shown that the infimum defining $\gamma^*$ is actually a minimum,
\[
\gamma^* = \min_{f \in \mathcal{K} \setminus \{ 0 \}} \frac{a(f,f)}{\la f,f \ra_\mathcal{H}} ,
\]
and that the minimum is attained when $f=u$.

Our second task is to show $u$ is an eigenvector with eigenvalue $\gamma^*$. Let $v \in \mathcal{K}$ be arbitrary and use $f=u+\e v$ as a trial
function in the Rayleigh quotient; since $u$ gives the minimizer, the derivative at $\e=0$ must equal zero by
the first derivative test from calculus:
\[
0 = \left. \frac{d\ }{d\e} \, \frac{a(u+\e v, u+\e v)}{\la u+\e v, u+\e v \ra_\mathcal{H}} \, \right|_{\e=0} = 2\Re a(u,v) - \gamma^* 2\Re \la u,v \ra_\mathcal{H} .
\]
The same equation holds with $\Im$ instead of $\Re$, as we see by replacing $v$ with $iv$. (This last step is unnecessary when working with real Hilbert spaces, of course.) Hence
\[
a(u,v) = \gamma^* \la u,v \ra_\mathcal{H} \qquad \forall v \in \mathcal{K},
\]
which means $u$ is an eigenvector for the sesquilinear form $a$, with eigenvalue $\gamma^*$.

\emph{Aside.} The higher eigenvalues ($\gamma_j$ for $j>1$) can be obtained by a similar process, minimizing the Rayleigh quotient on the orthogonal complement of
the span of the preceding eigenfunctions $u_1,\ldots,u_{j-1}$. In particular,
\[
\gamma_2 = \min_{f \perp u_1} \frac{a(f,f)}{\la f,f \ra_\mathcal{H}}
\]
where $f \perp u_1$ means that the trial function $f \in \mathcal{K} \setminus \{ 0 \}$ is assumed orthogonal to $u_1$ in $\mathcal{H}$: $\la f,u_1 \ra_\mathcal{H} =0$.

\chapter[Monotonicity of eigenvalues]{Monotonicity properties of eigenvalues} \label{ch:monot}

\subsubsection*{Goal} To apply Poincar\'{e}'s minimax principle to the Laplacian and related operators, and hence to establish monotonicity results for Dirichlet and Neumann eigenvalues of the Laplacian, and a diamagnetic comparison for the magnetic Laplacian.

\paragraph*{References} \cite{B}

\subsection*{Laplacian, biLaplacian, and Schr\"{o}dinger operators}
Applying the Rayleigh principle \eqref{eq:vc1} to the examples in Chapters~\ref{ch:aONB}--\ref{ch:aSch} gives:
\begin{align*}
\lam_1 & = \min_{f \in H^1_0(\Omega)} \frac{\int_\Omega |\nabla f|^2 \, dx}{\int_\Omega f^2 \, dx} && \text{Dirichlet Laplacian on $\Omega$,} \\
\rho_1 & = \min_{f \in H^1(\Omega)} \frac{\int_\Omega |\nabla f|^2 \, dx + \sigma \int_{\partial \Omega} f^2 \, dS}{\int_\Omega f^2 \, dx} && \text{Robin Laplacian on $\Omega$,} \\
\mu_1 & = \min_{f \in H^1(\Omega)} \frac{\int_\Omega |\nabla f|^2 \, dx}{\int_\Omega f^2 \, dx} && \text{Neumann Laplacian on $\Omega$,}
\end{align*}
\begin{align*}
\Lambda_1 & = \min_{f \in H^2_0(\Omega)} \frac{\int_\Omega \sum_{m,n=1}^d f_{x_m x_n}^2 \, dx}{\int_\Omega f^2 \, dx} && \text{Dirichlet biLaplacian on $\Omega$} \\
& = \min_{f \in H^2_0(\Omega)} \frac{\int_\Omega (\Delta f)^2 \, dx}{\int_\Omega f^2 \, dx} , \\
\beta_1 & = \min_{f \in H^1_0(\Omega)} \frac{\int_\Rd |i\nabla f + \vec{A}f|^2 \, dx}{\int_\Rd |f|^2 \, dx} && \text{magnetic Laplacian} \\
E_1 & = \min_{f \in H^1(\Rd) \cap L^2(|V| \, dx)} \frac{\int_\Rd \big( |\nabla f|^2 + Vf^2 \big) \, dx}{\int_\Rd f^2 \, dx} && \text{Schr\"{o}dinger with potential} \\
& && \text{$V(x)$ growing to infinity.}
\end{align*}
The Poincar\'{e} principle applies too, giving formulas for the higher eigenvalues and hence implying certain monotonicity relations, as follows.

\subsubsection*{Neumann $\leq$ Robin $\leq$ Dirichlet}
Free membranes give lower tones than partially free and fixed membranes:
\begin{theorem}[Neumann--Robin--Dirichlet comparison] \label{th:nleqd}
Let $\Omega$ be a bounded domain in $\Rd$ with Lipschitz boundary, and fix $\sigma>0$.

Then the Neumann eigenvalues of the Laplacian lie below their Robin counterparts, which in turn lie below the Dirichlet eigenvalues:
\[
\mu_j \leq \rho_j \leq \lam_j \qquad \forall j \geq 1 .
\]
\end{theorem}
\begin{proof}
Poincar\'{e}'s minimax principle gives the formulas
\begin{align*}
\mu_j & = \min_U \max_{f \in U \setminus \{ 0 \}} \frac{\int_\Omega |\nabla f|^2 \, dx}{\int_\Omega f^2 \, dx} \\
\rho_j & = \min_T \max_{f \in T \setminus \{ 0 \}} \frac{\int_\Omega |\nabla f|^2 \, dx + \sigma \int_{\partial \Omega} f^2 \, dS}{\int_\Omega f^2 \, dx} \\
\lam_j & = \min_S \max_{f \in S \setminus \{ 0 \}} \frac{\int_\Omega |\nabla f|^2 \, dx}{\int_\Omega f^2 \, dx}
\end{align*}
where $S$ ranges over all $j$-dimensional subspaces of $H^1_0(\Omega)$, and $T$ and $U$ range over all $j$-dimensional subspaces of $H^1(\Omega)$.

Clearly $\mu_j \leq \rho_j$. Further, every subspace $S$ is also a valid $T$, since $H^1_0 \subset H^1$. Thus the minimum for $\rho_j$ is taken over a larger class of subspaces. Since for $f \in H^1_0$ the boundary term vanishes in the Rayleigh quotient for $\rho_j$, we conclude that $\rho_j \leq \lam_j$.
\end{proof}

\subsubsection*{Domain monotonicity for Dirichlet spectrum}
Making a drum smaller increases its frequencies of vibration:
\begin{theorem} \label{th:monotDir}
Let $\Omega$ and $\widetilde{\Omega}$ be bounded domains in $\Rd$, and denote the eigenvalues of the Dirichlet Laplacian on these domains by $\lam_j$ and $\widetilde{\lam}_j$, respectively.

If $\Omega \supset \widetilde{\Omega}$ then
\[
\lam_j \leq \widetilde{\lam}_j \qquad \forall j \geq 1.
\]
\end{theorem}
\begin{proof}
Poincar\'{e}'s minimax principle gives that
\begin{align*}
\lam_j & = \min_S \max_{f \in S \setminus \{ 0 \}} \frac{\int_\Omega |\nabla f|^2 \, dx}{\int_\Omega f^2 \, dx} \\
\widetilde{\lam}_j & = \min_{\widetilde{S}} \max_{f \in \widetilde{S} \setminus \{ 0 \}} \frac{\int_{\widetilde{\Omega}} |\nabla f|^2 \, dx}{\int_{\widetilde{\Omega}} f^2 \, dx}
\end{align*}
where $S$ ranges over all $j$-dimensional subspaces of $H^1_0(\Omega)$ and $\widetilde{S}$ ranges over all $j$-dimensional subspaces of $H^1_0(\widetilde{\Omega})$.

Every subspace $\widetilde{S}$ is also a valid $S$, since $H^1_0(\widetilde{\Omega}) \subset H^1_0(\Omega)$ (noting that any approximating function in $C^\infty_0(\widetilde{\Omega})$ belongs also to $C^\infty_0(\Omega)$ by extension by $0$.) Therefore $\lam_j \leq \widetilde{\lam}_j$.
\end{proof}

\subsubsection*{Restricted reverse monotonicity for Neumann spectrum}
The monotonicity proof breaks down in the Neumann case because $H^1(\widetilde{\Omega})$ is not a subspace of $H^1(\Omega)$. More precisely, while one can extend a function in $H^1(\widetilde{\Omega})$ to belong to $H^1(\Omega)$, the extended function must generally be nonzero outside $\widetilde{\Omega}$, and so its $L^2$ norm and Dirichlet integral will differ from those of the original function.

Furthermore, counterexamples to domain monotonicity are easy to construct for Neumann eigenvalues, as the figure below shows
with a rectangle contained in a square. In that example, the square has side length $1$ and hence $\mu_2=\pi^2$, while the rectangle has
side length $\sqrt{2}(0.9)$ and so $\widetilde{\mu}_2 = \pi^2/(1.62)$, which is smaller than $\mu_2$.
\begin{figure}[h]
\begin{center}
\includegraphics[width=0.3\textwidth]{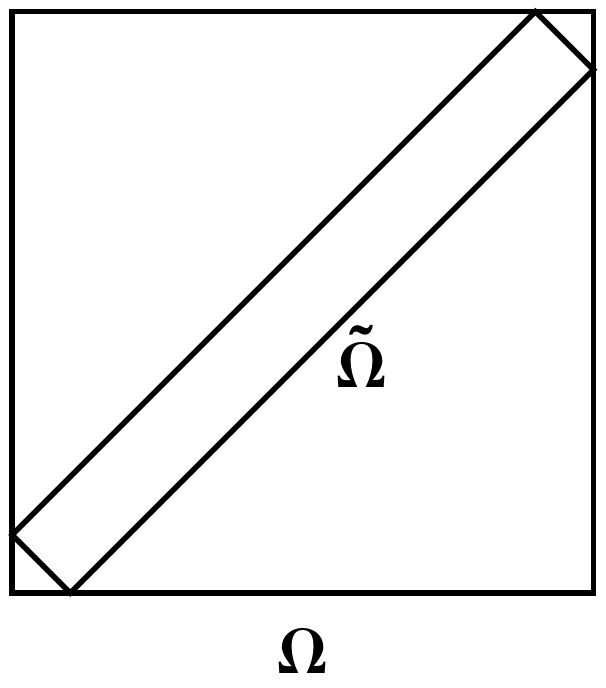}
\end{center}
\end{figure}

Nonetheless, monotonicity does holds in a certain restricted situation, although the inequality is reversed --- the smaller drum has lower tones:
\begin{theorem} \label{th:monotNeu}
Let $\Omega$ and $\widetilde{\Omega}$ be bounded Lipschitz domains in $\Rd$, and denote the eigenvalues of the Neumann Laplacian on these domains by $\mu_j$ and $\widetilde{\mu}_j$, respectively.

If $\widetilde{\Omega} \subset \Omega$ and $\Omega \setminus \widetilde{\Omega}$ has measure zero, then
\[
\widetilde{\mu}_j \leq \mu_j \qquad \forall j \geq 1.
\]
\end{theorem}
One might imagine the smaller domain $\widetilde{\Omega}$ as being constructed by removing a hypersurface of measure zero from $\Omega$, thus introducing an additional boundary surface.
Reverse monotonicity then makes perfect sense, because the additional boundary, on which values are not
specified for the eigenfunctions, enables the eigenfunctions to ``relax'' and hence lowers the eigenvalues.
\begin{figure}[h]
\begin{center}
\includegraphics[width=0.4\textwidth]{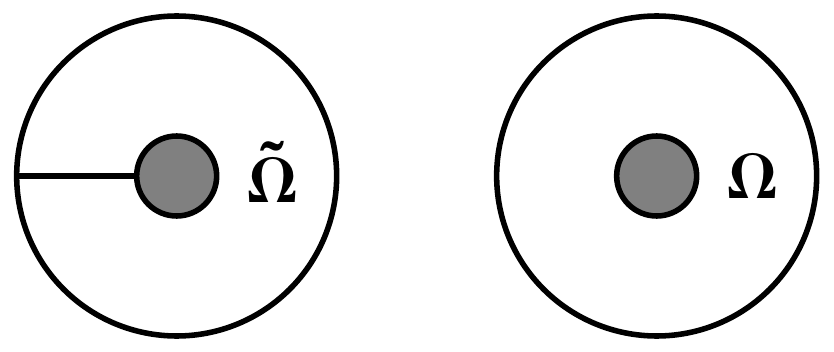}
\end{center}
\end{figure}

Introducing additional boundary surfaces to a Dirichlet problem would have the opposite effect: the eigenfunctions would be further constrained, and the eigenvalues raised.
\begin{proof}
Poincar\'{e}'s minimax principle gives that
\begin{align*}
\mu_j & = \min_S \max_{f \in S \setminus \{ 0 \}} \frac{\int_\Omega |\nabla f|^2 \, dx}{\int_\Omega f^2 \, dx} \\
\widetilde{\mu}_j & = \min_{\widetilde{S}} \max_{f \in \widetilde{S} \setminus \{ 0 \}} \frac{\int_\Omega |\nabla f|^2 \, dx}{\int_\Omega f^2 \, dx}
\end{align*}
where $S$ ranges over all $j$-dimensional subspaces of $H^1(\Omega)$ and $\widetilde{S}$ ranges over all $j$-dimensional subspaces of $H^1(\widetilde{\Omega})$.

Every subspace $S$ is also a valid $\widetilde{S}$, since each $f \in H^1(\Omega)$ restricts to a function in $H^1(\widetilde{\Omega})$ that has the same $H^1$-norm (using here that $\Omega \setminus \widetilde{\Omega}$ has measure zero). Therefore $\widetilde{\mu}_j \leq \mu_j$.
\end{proof}

\subsubsection*{Diamagnetic comparison for the magnetic Laplacian}

Imposing a magnetic field always raises the ground state energy.
\begin{theorem}[Diamagnetic comparison] \label{th:diamag}
\[
\beta_1 \geq \lam_1
\]
\end{theorem}
First we prove a pointwise comparison.
\begin{lemma}[Diamagnetic inequality] \label{le:diamag}
\[
|(i\nabla+\vec{A})f| \geq \big| \nabla |f| \big|
\]
\end{lemma}
\begin{proof}[Proof of Lemma~\ref{le:diamag}]
Write $f$ in polar form as $f=Re^{i\Theta}$. Then
\begin{align*}
|i\nabla f +\vec{A}f|^2
& = |i e^{i\Theta} \nabla R - Re^{i\Theta} \nabla \Theta + \vec{A}Re^{i\Theta}|^2 \\
& = |i \nabla R - R \nabla \Theta + \vec{A} R|^2 \\
& = |\nabla R|^2 + R^2 |\nabla \Theta - \vec{A}|^2 \\
& \geq |\nabla R|^2 = \big| \nabla |f| \big|^2 .
\end{align*}
\end{proof}
\begin{proof}[Proof of Theorem~\ref{th:diamag}]
The proof is immediate from the diamagnetic inequality in Lemma~\ref{le:diamag} and the Rayleigh principles for $\beta_1$ and $\lam_1$ at the beginning of this chapter. Note we can assume $f \geq 0$ in the Rayleigh principle for $\lam_1$, since the first Dirichlet eigenfunction can be taken nonnegative \cite[Theorem 8.38]{GT}.
\end{proof}

\chapter[Weyl's asymptotic]{Weyl's asymptotic for high eigenvalues} \label{ch:weyl}

\subsubsection*{Goal} To determine the rate of growth of eigenvalues of the Laplacian.

\paragraph*{References} \cite{A}; \cite{CH} Section VI.4

\subsubsection*{Notation}

The asymptotic notation $\alpha_j \sim \beta_j$ means
\[
\lim_{j \to \infty} \frac{\alpha_j}{\beta_j} = 1 .
\]
Write $V_d$ for the volume of the unit ball in $d$-dimensions.

\subsection*{Growth of eigenvalues}
The eigenvalues of the Laplacian grow at a rate $c j^{2/d}$ where the constant depends only on the volume of the domain, independent of the boundary conditions.
\begin{theorem}[Weyl's law] \label{th:weyl}
Let $\Omega$ be a bounded domain in $\Rd$ with piecewise smooth boundary. As $j \to \infty$ the eigenvalues grow according to:
\[
\lam_j \sim \rho_j \sim \mu_j \sim
\begin{cases}
\left( \pi j/|\Omega| \right)^2 & (d=1) \\
4\pi j/|\Omega| & (d=2) \\
\big( 6\pi^2 j/|\Omega| \big)^{2/3} & (d=3)
\end{cases}
\]
and more generally,
\[
\lam_j \sim \rho_j \sim \mu_j \sim 4\pi^2 \left( \frac{j}{V_d |\Omega|} \right)^{\! \! 2/d} \qquad (d \geq 1).
\]
Here $|\Omega|$ denotes the $d$-dimensional volume of the domain, in other words its length when $d=1$ and area when $d=2$.
\end{theorem}
In $1$ dimension the theorem is proved by the explicit formulas for the eigenvalues in Chapter~\ref{ch:lce}. We will prove the theorem in $2$ dimensions, by a technique known as ``Dirichlet--Neumann bracketing''. The higher dimensional proof is similar.

An alternative proof using small-time heat kernel asymptotics can be found (for example) in the survey paper by Arendt \emph{et al.} \cite[\S1.6]{A}.

\begin{proof}[Proof of Weyl aymptotic --- Step 1: rectangular domains]
In view of the Neu\-mann--Robin--Dirichlet comparison (Theorem~\ref{th:nleqd}), we need only prove Weyl's law for the Neumann and Dirichlet eigenvalues. We provided a proof in Proposition~\ref{pr:rectweyl}, for rectangles.
\end{proof}

\begin{proof}[Proof of Weyl aymptotic --- Step 2: finite union of rectangles]
Next we suppose $R_1,\ldots,R_n$ are disjoint rectangular domains and put
\begin{align*}
\widetilde{\Omega} & = \cup_{m=1}^n R_m \, , \\
\Omega & = \Interior \left( \cup_{m=1}^n \overline{R_m} \right) .
\end{align*}
\begin{figure}[h]
\begin{center}
\includegraphics[width=0.4\textwidth]{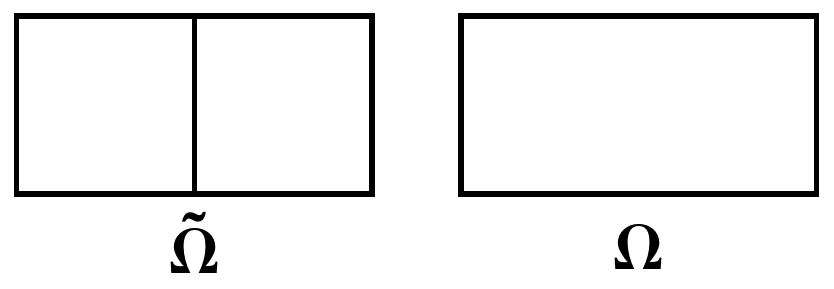}
\end{center}
\end{figure}
For example, if $R_1$ and $R_2$ are adjacent squares of side length $1$, then $\widetilde{\Omega}$ is the
disjoint union of those squares whereas $\Omega$ is the $2 \times 1$ rectangular domain formed from the interior of their union.

Admittedly $\widetilde{\Omega}$ is not connected, but the spectral theory of the Laplacian remains valid on a finite union of disjoint domains: the eigenfunctions are simply the eigenfunctions of each of the component domains extended to be zero on the other components, and the spectrum equals the union of the spectra of the individual components. (On an infinite union of disjoint domains, on the other hand, one would lose compactness of the imbedding $H^1 \hookrightarrow L^2$, and the zero eigenvalue of the Neumann Laplacian would have infinite multiplicity.)

Write $\widetilde{\lam}_j$ and $\widetilde{\mu}_j$ for the Dirichlet and Neumann eigenvalues of $\widetilde{\Omega}$.

Then by the restricted reverse Neumann monotonicity (Theorem~\ref{th:monotNeu}), Neumann--Robin--Dirichlet comparison (Theorem~\ref{th:nleqd}) and Dirichlet monotonicity (Theorem~\ref{th:monotDir}), we deduce that
\[
\widetilde{\mu}_j \leq \mu_j \leq \rho_j \leq \lam_j \leq \widetilde{\lam}_j \qquad \forall j \geq 1.
\]
Hence if we can prove Weyl's law
\begin{equation} \label{eq:weylunion}
\widetilde{\mu}_j \sim \widetilde{\lam}_j \sim \frac{4\pi j}{|\Omega|}
\end{equation}
for the union-of-rectangles domain $\widetilde{\Omega}$, then Weyl's law will follow for the original domain $\Omega$.

Define the eigenvalue counting functions of the rectangle $R_m$ to be
\begin{align*}
N_{\text{Neu}}(\alpha;R_m) & = \# \{ j \geq 1 : \mu_j(R_m) \leq \alpha \} , \\
N_{\text{Dir}}(\alpha;R_m) & = \# \{ j \geq 1 : \lam_j(R_m) \leq \alpha \} .
\end{align*}
We know from Weyl's law for rectangles (Step 1 of the proof above) that
\begin{equation} \label{eq:weylrect}
N_{\text{Neu}}(\alpha;R_m) \sim N_{\text{Dir}}(\alpha;R_m) \sim \frac{|R_m|}{4\pi} \alpha
\end{equation}
as $\alpha \to \infty$.

The spectrum of $\widetilde{\Omega}$ is the union of the spectra of the $R_m$, and so (here comes the key step in the proof!) the eigenvalue counting functions of $\widetilde{\Omega}$ equal the sums of the corresponding counting functions of the rectangles:
\begin{align*}
N_{\text{Neu}}(\alpha;\widetilde{\Omega}) = \sum_{m=1}^n N_{\text{Neu}}(\alpha;R_m) , \\
N_{\text{Dir}}(\alpha;\widetilde{\Omega}) = \sum_{m=1}^n N_{\text{Dir}}(\alpha;R_m) .
\end{align*}
Combining these sums with the asymptotic \eqref{eq:weylrect} shows that
\[
N_{\text{Neu}}(\alpha;\widetilde{\Omega}) \sim \left( \sum_{m=1}^n \frac{|R_m|}{4\pi} \right) \alpha = \frac{|\Omega|}{4\pi} \alpha
\]
and similarly
\[
N_{\text{Dir}}(\alpha;\widetilde{\Omega}) \sim \frac{|\Omega|}{4\pi} \alpha
\]
as $\alpha \to \infty$. We can invert these last two asymptotic formulas with the help of Lemma~\ref{le:inversion}, thus obtaining Weyl's law \eqref{eq:weylunion} for $\widetilde{\Omega}$.
\end{proof}

\begin{proof}[Proof of Weyl aymptotic --- Step 3: approximation of arbitrary domains]
Lastly we suppose $\Omega$ is an arbitrary domain with piecewise smooth boundary. The idea is to approximate $\Omega$
with a union-of-rectangles domain such as in Step 2, such that the volume of the approximating domain
is within $\e$ of the volume of $\Omega$. We refer to the text of Courant and Hilbert for the detailed proof \cite[{\S}VI.4.4]{CH}.
\end{proof}

\chapter[P\'{o}lya's conjecture]{P\'{o}lya's conjecture and the Berezin--Li--Yau Theorem}

\subsubsection*{Goal} To describe Polya's conjecture about Weyl's law, and to state the ``tiling domain'' and ``summed'' versions that are known to hold.

\paragraph*{References} \cite{K,L,P}

\subsubsection*{P\'{o}lya's conjecture}

Weyl's law (Theorem~\ref{th:weyl}) says that
\[
\lam_j \sim \frac{4\pi j}{|\Omega|} \sim \mu_j \qquad \text{as $j \to \infty$,}
\]
for a bounded plane domain $\Omega$ with piecewise smooth boundary. (We restrict to plane domains, in this chapter, for simplicity.)

P\'{o}lya conjectured that these asymptotic formulas hold as inequalities.
\begin{conjecture}[\cite{P}, 1960]
\[
\lam_j \geq \frac{4\pi j}{|\Omega|} \geq \mu_j \qquad \forall j \geq 1.
\]
\end{conjecture}
The conjecture remains open even for a disk.

P\'{o}lya proved the Dirichlet part of the inequality for tiling domains \cite{P}, and Kellner did the same for the Neumann part \cite{K}. Recall that a ``tiling domain'' covers the plane with congruent copies of itself (translations, rotations and reflections). For example, parallelograms and triangles are tiling domains, as are many variants of these domains (a fact that M. C. Escher exploited in his artistic creations).

P\'{o}lya and Kellner's proofs are remarkably simple, using a rescaling argument together with Weyl's law.

For arbitrary domains, P\'{o}lya's conjecture has been proved only for $\lam_1,\lam_2$ (see \cite[Th.~3.2.1 and (4.3)]{H}) and for $\mu_1,\mu_2,\mu_3$ (see \cite{GNP}). The conjecture remains open for $j \geq 3$ (Dirichlet) and $j \geq 4$ (Neumann).

\subsubsection*{Berezin--Li--Yau results}

The major progress for arbitrary domains has been on a ``summed'' version of the conjecture. (Quite often in analysis, summing or integrating an expression produces a significantly more tractable quantity.) Li and Yau \cite{LY} proved that
\[
\sum_{k=1}^j \lam_k \geq \frac{2\pi j^2}{|\Omega|} ,
\]
which is only slightly smaller than the quantity $(2\pi/|\Omega|)j(j+1)$ that one gets by summing the left side of the P\'{o}lya conjecture. An immediate consequence is a Weyl-type inequality for Dirichlet eigenvalues:
\[
\lam_j \geq \frac{2\pi j}{|\Omega|}
\]
by combining the very rough estimate $j \lam_j \geq \sum_{k=1}^j \lam_k$ with the Li--Yau inequality. The last formula has $2\pi$ whereas P\'{o}lya's conjecture demands $4\pi$, and so we see the conjecture is true up to a factor of $2$, at worst.

Similar results hold for Neumann eigenvalues.

A somewhat more general approach had been obtained earlier by Berezin. For more information, consult the work of Laptev \cite{L} and a list of open problems from recent conferences \cite{AIM}.

\chapter[Reaction--diffusion stability]{Case study: stability of steady states for reaction--diffusion PDEs}
\label{ch:reacdiff}

\subsubsection*{Goal} To linearize a nonlinear reaction--diffusion PDE around a steady state, and study the spectral theory of the linearized operator by time-map methods.

\paragraph*{References} \cite{Sch} Section 4.1

\subsubsection*{Reaction--diffusion PDEs}
Assume throughout this section that $f(y)$ is a smooth function on $\R$. Let $X>0$. We study the reaction--diffusion PDE
\begin{equation} \label{eq:reacdiff}
u_t = u_{xx} + f(u)
\end{equation}
on the interval $(0,X)$ with Dirichlet boundary conditions $u(0)=u(X)=0$. Physical interpretations include: (i) $u=$temperature and $f=$rate of heat generation, (ii) $u=$chemical concentration and $f=$reaction rate of chemical creation.

Intuitively, the $2$nd order diffusion term in the PDE is stabilizing (since $u_t=u_{xx}$ is the usual diffusion equation), whereas the $0$th order reaction term can be destabilizing (since solutions to $u_t=f(u)$ will grow, when $f$ is positive). Thus the reaction--diffusion PDE features a competition between stabilizing and destabilizing effects. This competition leads to nonconstant steady states, and interesting stability behavior.

\paragraph*{Steady states.} If $U(x)$ is a steady state, then
\begin{equation} \label{eq:reacdiffss}
U^{\prime \prime} + f(U) = 0 , \qquad 0<x<X .
\end{equation}
More than one steady state can exist. For example if $f(0)=0$ then $U \equiv 0$ is a steady state, but nonconstant steady states might exist too, such as $U(x)=\sin x$ when $X=\pi$ and $f(y)=y$.

\subsubsection*{Linearized PDE} We perturb a steady state by considering
\[
u=U+\e \phi
\]
where the perturbation $\phi(x,t)$ is assumed to satisfy the Dirichlet BC $\phi=0$ at $x=0,L$, for each $t$. Substituting $u$ into the equation \eqref{eq:reacdiff} gives
\begin{align*}
0+\e \phi_t
& = (U_{xx}+ \e \phi_{xx}) + f(U+\e \phi) \\
& = U_{xx} + \e \phi_{xx} + f(U) + f^\prime(U) \e \phi + O(\e^2) .
\end{align*}
The leading terms, of order $\e^0$, equal zero by the steady state equation for $U$. We discard terms of order $\e^2$ and higher. The remaining terms, of order $\e^1$, give the linearized equation:
\begin{equation} \label{eq:reaclin1}
\phi_t = \phi_{xx} + f^\prime(U) \phi .
\end{equation}
That is,
\[
\phi_t = - L\phi
\]
where $L$ is the symmetric linear operator
\[
Lw = -w_{xx} - f^\prime(U) w .
\]
Separation of variables gives (formally) solutions of the form
\[
\phi = \sum_j c_j e^{-\tau_j t} w_j(x) ,
\]
where the eigenvalues $\tau_j$ and Dirichlet eigenfunctions $w_j$ satisfy
\[
Lw_j = \tau_j w_j
\]
with $w_j(0)=w_j(X)=0$.

Thus the steady state $U$ of the reaction--diffusion PDE is
\[
\text{\textbf{linearly unstable} if $\tau_1<0$}
\]
because the perturbation $\phi$ grows to infinity, whereas the steady state is
\[
\text{\textbf{linearly stable} if $\tau_1 \geq 0$}
\]
because $\phi$ remains bounded in that case.

To make these claims rigorous, we study the spectrum of $L$.

\subsubsection*{Spectrum of $L$}

We take:

$\Omega=(0,X)$

$\mathcal{H}=L^2(0,X)$, inner product $\la u,v \ra_{L^2} = \int_0^X uv \, dx$

$\mathcal{K}=H^1_0(0,X)$, inner product
\[
\la u,v \ra_{H^1} = \int_0^X (u^\prime v^\prime + uv) \, dx
\]

Compact imbedding $H^1_0 \hookrightarrow L^2$ by Rellich--Kondrachov

Symmetric sesquilinear form
\[
a(u,v) = \int_0^X \big( u^\prime v^\prime - f^\prime(U)uv + Cuv \big) \, dx
\]
where $C>0$ is chosen larger than $\lv f^\prime \rv_{L^\infty}+1$. Proof of ellipticity:
\[
a(u,u) \geq \int_0^X \big( (u^\prime)^2 + u^2 \big) \, dx = \lv u \rv_{H^1}^2
\]
by choice of $C$.

\medskip
The discrete spectral Theorem~\ref{th:spec} now yields an ONB of eigenfunctions $\{ w_j \}$ with eigenvalues $\gamma_j$ such that
\[
a(w_j,v) = \gamma_j \la w_j,v \ra_{L^2} \qquad \forall v \in H^1_0(0,X) .
\]
Writing $\gamma_j=\tau_j+C$ we get
\[
\int_0^X \big( w_j^\prime v^\prime - f^\prime(U) w_j v \big) \, dx = \tau_j \int_0^X w_j v \, dx \qquad \forall v \in H^1_0(0,X) .
\]
These eigenfunctions satisfy $Lw_j = \tau_j w_j$ weakly, and hence also classically.

\subsubsection*{Stability of the zero steady state.} Assume $f(0)=0$, so that $U \equiv 0$ is a steady state.  Its stability is easily determined, as follows.

The linearized operator is $Lw=-w^{\prime \prime} - f^\prime(0)w$, which on the interval $(0,X)$ has Dirichlet eigenvalues
\[
\tau_j = \big( \frac{j\pi}{X} \big)^2 - f^\prime(0) .
\]
Thus the zero steady state is linearly unstable if and only if
\[
\big( \frac{\pi}{X} \big)^2 < f^\prime(0) .
\]
Thus we may call the reaction--diffusion PDE ``long-wave unstable'' when $f^\prime(0)>0$, because then the zero steady state is unstable with respect to perturbations of sufficiently long wavelength $X$. On short intervals, the Dirichlet BCs are strong enough to stabilize the steady state.

\subsubsection*{Sufficient conditions for linearized instability of nonconstant steady states}

Our first instability criterion is \emph{structural}, meaning it depends on properties of the reaction function $f$ rather than on properties of the particular steady state $U$.
\begin{theorem} \label{th:reactinst}
Assume the steady state $U$ is nonconstant, and that $f(0)=0, f^{\prime \prime}(0)=0$ and $f^{\prime \prime \prime}>0$. Then $\tau_1<0$.
\end{theorem}
For example, the theorem shows that nonconstant steady states are unstable when $f(y)=y^3$.
\begin{proof}
First we collect facts about boundary values, to be used later in the proof when we integrate by parts:
\begin{align*}
U & = 0 \text{\ at $x=0,X$} && \text{by the Dirichlet BC,} \\
f(U) & = 0 \text{\ at $x=0,X$} && \text{since $f(0)=0$,} \\
U^{\prime \prime} & = 0 \text{\ at $x=0,X$} && \text{because $U^{\prime \prime} = -f(U)$} \\
f^{\prime \prime}(U) & = 0 \text{\ at $x=0,X$} && \text{since $f^{\prime \prime}(0)=0$.}
\end{align*}

The Rayleigh principle for $L$ says that
\[
\tau_1 = \min \Big\{ \frac{\int_0^X \big( (w^\prime)^2 - f^\prime(U) w^2 \big) \, dx}{\int_0^X w^2 \, dx} :
w \in H^1_0(0,X) \Big\} .
\]
We choose a trial function
\[
w = U^{\prime \prime},
\]
which is not the zero function, since $U$ is nonconstant. Then the numerator of the Rayleigh quotient for $w$ is
\begin{align*}
& \int_0^X \big( (U^{\prime \prime \prime})^2 - f^\prime(U) (U^{\prime \prime})^2 \big) \, dx \\
& = \int_0^X \big( -U^{\prime \prime \prime \prime} - f^\prime(U) U^{\prime \prime} \big) U^{\prime \prime} \, dx && \text{by parts} \\
& = \int_0^X f^{\prime \prime}(U) (U^\prime)^2 U^{\prime \prime} \, dx && \text{by the steady state equation \eqref{eq:reacdiffss}} \\
& = \frac{1}{3} \int_0^X f^{\prime \prime}(U) \big[ (U^\prime)^3 \big]^\prime \, dx \\
& = - \frac{1}{3} \int_0^X f^{\prime \prime \prime}(U) (U^\prime)^4 \, dx && \text{by parts} \\
& < 0
\end{align*}
since $f^{\prime \prime \prime}>0$ and $U$ is nonconstant. Hence $\tau_1 < 0$, by the Rayleigh principle.
\end{proof}

\noindent \emph{Motivation for the choice of trial function.} Our trial function $w=U^{\prime \prime}$ corresponds to a perturbation $u=U+\e \phi=U+\e U^{\prime \prime}$, which tends (when $\e>0$) to push the steady state towards the constant function. The opposite perturbation ($\e<0$) would tend to make the solution grow even further away from the constant steady state.

\medskip
The next instability criterion, rather than being structural, depends on particular properties of the steady state.
\begin{theorem}[\protect{\cite[Proposition 4.1.2]{Sch}}] \label{th:reactsign}
Assume the nonconstant steady state $U$ changes sign on $(0,X)$. Then $\tau_1<0$.
\end{theorem}
For example, suppose $f(y)=y$ so that the steady state equation is $U^{\prime \prime}+U=0$. If $X=2\pi$ then the steady state $U=\sin x$ is linearly unstable, by the theorem. Of course, for that example we can compute the spectrum of $L$ exactly: the lowest eigenfunction is $w=\sin(x/2)$ with eigenvalue $\tau_1 = \big( \tfrac{1}{2} \big)^2-1<0$.
\begin{proof}
If $U$ changes sign then it has a positive local maximum and a negative local minimum in $(0,X)$, recalling that $U=0$ at the endpoints. Obviously $U^\prime$ must be nonzero at some point between these local extrema, and so there exist points $0<x_1<x_2<X$ such that
\[
U^\prime(x_1)=U^\prime(x_2)=0
\]
and $U^\prime \neq 0$ on $(x_1,x_2)$. Define a trial function
\[
w =
\begin{cases}
U^\prime & \text{on $(x_1,x_2)$,} \\
0 & \text{elsewhere.}
\end{cases}
\]
(We motivate this choice of trial function at the end of the proof.) Then $w$ is piecewise smooth, and is continuous since $w=U^\prime=0$ at $x_1$ and $x_2$. Therefore $w \in H^1_0(0,X)$, and $w \not \equiv 0$ since $U^\prime \neq 0$ on $(x_1,x_2)$.

The numerator of the Rayleigh quotient for $w$ is
\begin{align*}
\int_0^X \big( (w^\prime)^2 - f^\prime(U) w^2 \big) \,  dx
& = \int_{x_1}^{x_2} \big( -w^{\prime \prime} - f^\prime(U)w \big) w \, dx \qquad \text{by parts} \\
& = 0
\end{align*}
since
\begin{equation} \label{eq:reactrial}
-w^{\prime \prime} = - U^{\prime \prime \prime}  =\big( f(U) \big)^\prime = f^\prime(U) U^\prime = f^\prime(U)w .
\end{equation}
Hence $\tau_1 \leq 0$, by the Rayleigh principle for the first eigenvalue.

Suppose $\tau_1=0$, so that the Rayleigh quotient of $w$ equals $\tau_1$. Then $w$ must be an eigenfunction with eigenvalue $\tau_1$ (because substituting $w=\sum_j c_j w_j$ into the Rayleigh quotient would give a value larger than $\tau_1$, if $c_j$ were nonzero for any term with eigenvalue larger than $\tau_1$).

Since the eigenfunction $w$ must be smooth, the slopes of $w$ from the left and the right at $x_2$ must agree, which means $w^\prime(x_2)=w^\prime(x_2 +)=0$. Thus $w(x_2)=w^\prime(x_2)=0$ and $w$ satisfies the second order linear ODE \eqref{eq:reactrial} on $(x_1,x_2)$. Therefore $w \equiv 0$ on $(x_1,x_2)$, by uniqueness, which contradicts our construction of $w$. We conclude $\tau_1 < 0$.
\end{proof}

\noindent \emph{Motivation for the choice of trial function.} The steady state equation reads $U^{\prime \prime} + f(U)=0$, and differentiating shows that $U^\prime$ lies in the nullspace of the linearized operator $L$:
\[
L U^\prime = -(U^\prime)^{\prime \prime} - f^\prime(U) U^\prime = 0 .
\]
In other words, $U^\prime$ is an eigenfunction with eigenvalue $0$, which almost proves instability (since instability would correspond to a negative eigenvalue). Of course, the eigenfunction $U^\prime$ does not satisfy the Dirichlet boundary conditions at the endpoints, and hence we must restrict to the subinterval $(x_1,x_2)$, in the proof above, in order to obtain a valid trial function.

\subsubsection*{Time maps and linearized stability}

Next we derive instability criteria  that are almost necessary and sufficient. These conditions depend on the \emph{time map} for a family of steady states.

Parameterize the steady states by their slope at the left endpoint: given $s \neq 0$, write $U_s(x)$ for the steady state on $\R$ (if it exists) satisfying
\[
U_s(0)=0, \qquad U_s^\prime(0)=s,  \qquad \text{$U_s(x)=0$ for some $x>0$.}
\]
Define the \emph{time map} to give the first point or ``time'' $x$ at which the steady state hits the axis:
\[
T(s) = \min \{ x>0 : U_s(x)=0 \} .
\]

If $U_s$ exists for some $s \neq 0$ then it exists for all nonzero $s$-values in a neighborhood, and the time map is smooth on that neighborhood \cite[Proposition 4.1.1]{Sch}. The time map can be determined numerically by plotting solutions with different initial slopes, as the figures below show. In the first figure the time map is decreasing, whereas in the second it increases.

\begin{figure}[h]
\begin{center}
\includegraphics[width=0.4\textwidth]{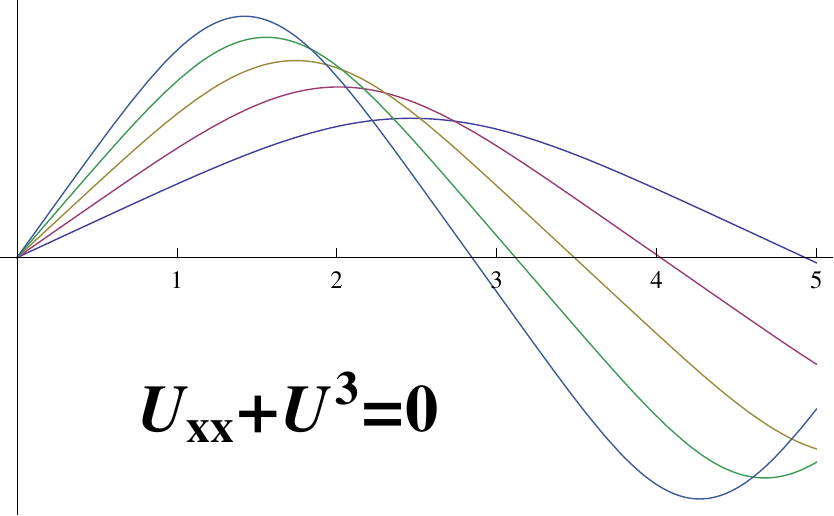} \qquad
\includegraphics[width=0.4\textwidth]{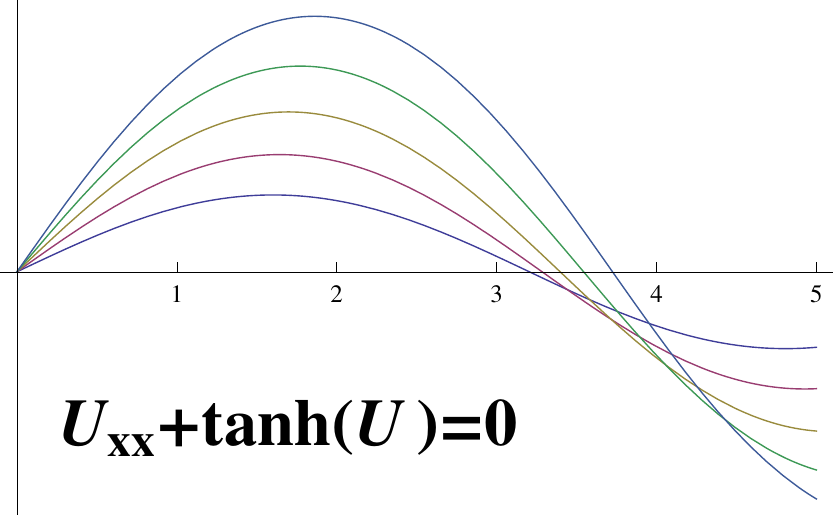}
\end{center}
\end{figure}

Monotonicity of the time maps determines stability of the steady state:
\begin{theorem}[\protect{\cite[Proposition 4.1.3]{Sch}}] \label{th:reacttime}
The steady state $U_s$ is linearly unstable on the interval $(0,T(s))$ if $sT^\prime(s)<0$, and is linearly stable if $sT^\prime(s)>0$.
\end{theorem}
\begin{proof}
We begin by differentiating the family of steady states with respect to the parameter $s$, and obtaining some properties of that function. Then we treat the ``instability'' and ``stability'' parts of the theorem separately.

Write $s_0 \neq 0$ for a specific value of $s$, in order to reduce notational confusion. Let $X=T(s_0)$.
Define a function
\[
v = \frac{\partial U_s}{\partial s} \Big|_{s=s_0}
\]
on $(0,X)$, where we use that $U_s(x)$ is jointly smooth in $(x,s)$. Then
\begin{equation} \label{eq:reactime1}
v^{\prime \prime} + f^\prime(U) v = 0
\end{equation}
as one sees by differentiating the steady state equation \eqref{eq:reacdiffss} with respect to $s$, and writing $U$ for $U_{s_0}$.

At the left endpoint we have
\[
v(0)=0, \qquad v^\prime(0)=1 ,
\]
because $U_s(0)=0, U_s^\prime(0)=s$ for all $s$.

We do not expect $v$ to vanish at the right endpoint, but we can calculate its value there to be
\[
v(X) = s_0 T^\prime(s_0) ,
\]
as follows. First, differentiating the equation $0=U_s(T(s))$ gives that
\begin{align*}
0
& = \frac{\partial \ }{\partial s} U_s(T(s)) \\
& = \frac{\partial U_s}{\partial s} \big( T(s) \big) + U_s^\prime(T(s))T^\prime(s) \\
& = v \big( T(s) \big) + U_s^\prime(T(s))T^\prime(s) .
\end{align*}
Note the steady state $U_s$ is symmetric about the midpoint of the interval $(0,T(s))$ (exercise; use that $U_s=0$ at both endpoints and that the steady state equation is invariant under $x \mapsto -x$, so that steady states must be symmetric about any local maximum point). Thus $U_s^\prime(T(s))=-U_s^\prime(0)=-s$, and evaluating the last displayed formula at $s=s_0$ then gives that $0=v(X)-s_0 T^\prime(s_0)$, as we wanted.

\medskip
\emph{Proof of instability.}
Assume $s_0 T^\prime(s_0)<0$. Then $v(X)<0$. Since $v^\prime(0)=1$ we know $v(x)$ is positive for small values of $x$, and so some $x_2 \in (0,X)$ exists at which $v(x_2)=0$. Define a trial function
\[
w =
\begin{cases}
v & \text{on $[0,x_2)$,} \\
0 & \text{elsewhere.}
\end{cases}
\]
Then $w$ is piecewise smooth, and is continuous since $v=0$ at $x_2$. Note $w(0)=0$. Therefore $w \in H^1_0(0,X)$, and $w \not \equiv 0$.

Hence $\tau_1 < 0$ by arguing as in the proof of Theorem~\ref{th:reactsign}, except with $x_1=0$.

\medskip \noindent [\emph{Motivation for the choice of trial function.} Differentiating the steady state equation $U^{\prime \prime} + f(U)=0$ with respect to $s$ shows that $\partial U/\partial s$ is an eigenfunction with eigenvalue zero:
\[
L \big( \frac{\partial U}{\partial s} \big) = -\big( \frac{\partial U}{\partial s} \big)^{\prime \prime} - f^\prime(U) \frac{\partial U}{\partial s} = 0 .
\]
In other words, $\partial U/\partial s$ lies in the nullspace of the linearized operator. It does not satisfy the Dirichlet boundary condition at the right endpoint, but we handled that issue in the proof above by restricting to the subinterval $(0,x_2)$, in order to obtain a valid trial function.]

\medskip
\emph{Proof of stability.} Assume $s_0 T^\prime(s_0)>0$, so that $v(X)>0$. Define $\sigma=-v^\prime(X)/v(X)$. Then
\[
v(0)=0, \qquad v^\prime(X)+ \sigma v(X)=0 ,
\]
which is a mixed Dirichlet--Robin boundary condition. We will show later that $v$ is a first eigenfunction for $L$, under this mixed condition, with eigenvalue is $\rho_1=0$ (since $Lv=0$ by \eqref{eq:reactime1}).

By adapting our Dirichlet-to-Robin monotonicity result (Theorem~\ref{th:nleqd}) one deduces that
\[
\tau_1 \geq \rho_1 = 0 ,
\]
which gives linearized stability of the steady state $U$.

To show $v$ is a first eigenfunction for $L$, as used above, we first show $v$ is positive on $(0,X)$. Apply the steady state equation \eqref{eq:reacdiffss} to $U_s$, and multiply by $U_s^\prime$ and integrate to obtain the energy equation
\begin{equation} \label{eq:reacenergy}
\frac{1}{2} (U_s^\prime)^2 + F(U_s) = \frac{1}{2} s^2 ,
\end{equation}
where $F$ is an antiderivative of $f$ chosen with $F(0)=0$. Differentiating with respect to $s$ at $s=s_0$ gives that
\[
U^\prime v^\prime + f(U)v = s_0.
\]
Hence if $v$ vanishes at some $x_0 \in (0,X)$ then $U^\prime v^\prime = s_0 \neq 0$ and so $v^\prime(x_0) \neq 0$. Thus at any two successive zeros of $v$, we know $v^\prime$ has opposite signs. Therefore $U^\prime$ has opposite signs too, because $U^\prime v^\prime = s_0$ at the zeros. It is straightforward to show from \eqref{eq:reacenergy} that $U$ increases on $[0,X/2]$ and decreases on $[X/2,X]$, and so after the zero of $v$ at $x=0$ the next zero (if it exists) can only be $>X/2$, and the one after that must be $>X$. Since we know $v(x)$ is positive for small $x$ and that $v(X)>0$, we conclude $v$ has no zeros in $(0,X)$ and hence is positive there.

The first eigenfunction of $L$ with mixed Dirichlet--Robin boundary condition is positive, and it is the unique positive eigenfunction (adapt the argument in \cite[Theorem 8.38]{GT}). Since the eigenfunction $v$ is positive, we conclude that it is the first Dirichlet--Robin eigenfunction, as desired.
\end{proof}

\chapter[Thin fluid film stability]{Case study: stability of steady states for thin fluid film PDEs}

\subsubsection*{Goal} To linearize a particular nonlinear PDE around a steady state, and develop the spectral theory of the linearized operator.

\paragraph*{References} \cite{LP1,LP2}

\subsubsection*{Thin fluid film PDE}

The evolution of a thin layer of fluid (such as paint) on a flat substrate (such as the ceiling) can be modeled using the \emph{thin fluid film PDE}:
\[
h_t = - \big( f(h) h_{xxx} \big)_x - \big( g(h) h_x \big)_x
\]
where $h(x,t)>0$ measures the thickness of the fluid, and the smooth, positive coefficient functions $f$ and $g$ represent surface tension and gravitational effects (or substrate-fluid interactions). For simplicity we assume $f \equiv 1$, so that the equation becomes
\begin{equation} \label{eq:thinfilm}
h_t = - h_{xxxx} - \big( g(h) h_x \big)_x .
\end{equation}
We will treat the case of general $g$, but readers are welcome to focus on the special case $g(y)=y^p$ for some $p \in \R$.

Solutions are known to exist for small time, given positive smooth initial data. But films can ``rupture'' in finite time, meaning $h(x,t) \searrow 0$ as $t \nearrow T$, for some coefficient functions $g$ (for example, for $g \equiv 1$).

Intuitively, the $4$th order ``surface tension'' term in the PDE  is stabilizing (since $h_t=-h_{xxxx}$ is the usual $4$th order diffusion equation) whereas the $2$nd order ``gravity'' term is destabilizing (since $h_t=-h_{xx}$ is the backwards heat equation). Thus the thin film PDE features a competition between stabilizing and destabilizing effects. This competition leads to nonconstant steady states, and interesting stability behavior.

\paragraph*{Periodic BCs and conservation of fluid.} Fix $X>0$ and assume $h$ is $X$-periodic with respect to $x$. Then the total volume of fluid is conserved, since
\begin{align*}
\frac{d\ }{dt} \int_0^X h(x,t) \, dx
& = - \int_0^X \big( h_{xxxx} +\big( g(h) h_x \big)_x \big) \, dx \\
& = - \big( h_{xxx} + g(h) h_x \big) \Big|_{x=0}^{x=X} \\
& = 0
\end{align*}
by periodicity.

\paragraph*{Nonconstant steady states.} Every constant function is a steady state of \eqref{eq:thinfilm}. We discuss the stability of these steady states at the end of the chapter.

To find \emph{nonconstant} steady states, substitute $h=H(x)$ and solve:
\begin{align}
-H_{xxxx} - ( g(H) H_x )_x & = 0 \label{eq:filmss} \\
H_{xxx} + g(H) H_x & = \alpha \notag \\
H_{xx} + G(H) & = \beta + \alpha x \notag \\
H_{xx} + G(H) & = \beta \notag
\end{align}
where $G$ is an antiderivative of $g$; here $\alpha=0$ is forced because the left side of the equation ($H_{xx} + G(H)$) is periodic. This last equation describes a nonlinear oscillator, and it is well known how to construct solutions (one multiplies by $H_x$ and integrates). For example, when $g \equiv (2\pi/X)^2$ we have steady states $H(x)=(\text{const.})+\cos(2\pi x/X)$. For the general case see \cite{LP1}.

Assume from now on that $H(x)$ is a nonconstant steady state with \text{period $X$}.

\subsubsection*{Linearized PDE} We perturb a steady state by considering
\[
h=H+\e \phi
\]
where the perturbation $\phi(x,t)$ is assumed to have mean value zero $\int_0^X \phi(x,t) \, dx = 0$), so that fluid is conserved. Substituting $h$ into the equation \eqref{eq:thinfilm} gives
\begin{align*}
0+\e \phi_t
& = - (H_{xxxx}+\e \phi_{xxxx}) - \big( g(H+\e \phi) (H_x + \e \phi_x) \big)_x \\
& = - H_{xxxx} - \big( g(H) H_x \big)_x - \e \big[ \phi_{xxx} + g(H) \phi_x + g^\prime(H) H_x \phi \big]_x + O(\e^2) .
\end{align*}
The leading terms, of order $\e^0$, equal zero by the steady state equation for $H$. We discard terms of order $\e^2$ and higher. The remaining terms, of order $\e^1$, give the linearized equation:
\begin{equation} \label{eq:thinlin1}
\phi_t = - \big[ \phi_{xx} + g(H) \phi \big]_{xx} .
\end{equation}
Unfortunately, the operator on the right side is not symmetric (meaning it does not equal its formal adjoint). To make it symmetric, we ``integrate up'' the equation, as follows. Write
\[
\phi=\psi_x
\]
where $\psi$ is $X$-periodic (since $\phi$ has mean value zero). We may suppose $\psi$ has mean value zero at each time, by adding to $\psi$ a suitable function of $t$.

Substituting $\phi=\psi_x$ into \eqref{eq:thinlin1} gives that
\begin{align*}
\psi_{tx} & = - \big[ \psi_{xxx} + g(H) \psi_x \big]_{xx} \\
\psi_t & = - \big[ \psi_{xxx} + g(H) \psi_x \big]_x
\end{align*}
(noting the constant of integration must equal $0$, by integrating both sides and using periodicity). Thus
\[
\psi_t = -L\psi
\]
where $L$ is the symmetric operator
\[
Lw = w_{xxxx}+ \big( g(H) w_x \big)_x .
\]
Separation of variables gives (formally) solutions of the form
\[
\psi = \sum_j c_j e^{-\tau_j t} w_j(x) , \qquad \phi = \sum_j c_j e^{-\tau_j t} w_j^\prime(x) ,
\]
where the eigenvalues $\tau_j$ and periodic eigenfunctions $w_j$ satisfy
\[
Lw_j = \tau_j w_j .
\]
We conclude that the steady state $H$ of the thin fluid film PDE is
\[
\text{\textbf{linearly unstable} if $\tau_1<0$}
\]
because the perturbation $\phi$ grows to infinity, whereas the steady state is
\[
\text{\textbf{linearly stable} if $\tau_1 \geq 0$}
\]
because $\phi$ remains bounded in that case. Remember these stability claims relate only to mean zero (volume preserving) perturbations.

To make these claims more rigorous, we need to understand the eigenvalue problem for $L$.

\subsubsection*{Spectrum of $L$}

We take:

$\Omega=\T = \R / (X\Z)=$ torus of length $X$, so that functions on $\Omega$ are $X$-periodic

$\mathcal{H}=L^2(\T)$, inner product $\la u,v \ra_{L^2} = \int_0^X uv \, dx$

$\mathcal{K}=H^2(\T) \cap \{ u : \int_0^X u \, dx = 0 \}$, with inner product
\[
\la u,v \ra_{H^2} = \int_0^X (u^{\prime \prime} v^{\prime \prime} + u^\prime v^\prime + uv) \, dx
\]

Compact imbedding $\mathcal{K} \hookrightarrow L^2$ by Rellich--Kondrachov

Symmetric sesquilinear form
\[
a(u,v) = \int_0^X \big( u^{\prime \prime} v^{\prime \prime} - g(H) u^\prime v^\prime + Cuv \big) \, dx
\]
where $C>0$ is a sufficiently large constant to be chosen below.

Proof of ellipticity: The quantity $a(u,u)$ has a term of the form $-(u^\prime)^2$, whereas for $\lv u \rv_{H^2}^2$ we need $+(u^\prime)^2$. To get around this obstacle we ``hide'' the $-(u^\prime)^2$ term inside the terms of the form $(u^{\prime \prime})^2$ and $u^2$. Specifically,
\begin{align}
\int_0^X (u^\prime)^2 \, dx
& = - \int_0^X u^{\prime \prime} u \, dx \notag \\
& \leq \int_0^X \big( \delta (u^{\prime \prime})^2 + (4\delta)^{-1} u^2 \big) \, dx \label{eq:cauchydelta}
\end{align}
for any $\delta >0$. Here we used ``Cauchy-with-$\delta$, which is the observation that for any $\alpha,\beta \in \R$,
\[
0 \leq \big( \sqrt{\delta} \alpha \pm (4\delta)^{-1/2} \beta \big)^2 \qquad \Longrightarrow \qquad
|\alpha \beta| \leq \delta \alpha^2 + (4\delta)^{-1} \beta^2 .
\]
Next,
\begin{align*}
& a(u,u) \\
& \geq \int_0^X \big( (u^{\prime \prime})^2 - \big( \lv g(H) \rv_{L^\infty} + \frac{1}{2} \big) (u^\prime)^2 +  \frac{1}{2} (u^\prime)^2 + Cu^2 \big) \, dx \\
& \geq \int_0^X \Big( \big[ 1 - \big( \lv g(H) \rv_{L^\infty} + \frac{1}{2} \big) \delta \big] (u^{\prime \prime})^2 +  \frac{1}{2} (u^\prime)^2 + \big[ C-\big( \lv g(H) \rv_{L^\infty} + \frac{1}{2} \big) (4\delta)^{-1} \big] u^2 \Big) \, dx \\
& \qquad \qquad \qquad \qquad \text{by \eqref{eq:cauchydelta}} \\
& \geq \frac{1}{2} \lv u \rv_{H^2}^2
\end{align*}
provided we choose $\delta$ sufficiently small (depending on $H$) and then choose $C$ sufficiently large. Thus ellipticity holds.

\medskip
The discrete spectral Theorem~\ref{th:spec} now yields an ONB of eigenfunctions $\{ w_j \}$ with eigenvalues $\gamma_j$ such that
\[
a(w_j,v) = \gamma_j \la w_j,v \ra_{L^2} \qquad \forall v \in \mathcal{K} .
\]
Writing $\gamma_j=\tau_j+C$ we get
\[
\int_0^X \big( w_j^{\prime \prime} v^{\prime \prime} - g(H) w_j^\prime v^\prime \big) \, dx = \tau_j \int_0^X w_j v \, dx \qquad \forall v \in \mathcal{K} .
\]
These eigenfunctions satisfy $Lw_j = \tau_j w_j$ weakly, and hence also classically (by elliptic regularity, since $H$ and $g$ are smooth).

\paragraph*{Zero eigenvalue due to translational symmetry.} We will show that $\tau=0$ is always an eigenvalue, with eigenfunction $u=H-\overline{H}$, where the constant $\overline{H}$ equals the mean value of the steady state $H$. Indeed,
\[
Lu = L(H-\overline{H}) = H_{xxxx}+ \big( g(H) H_x \big)_x = 0
\]
by the steady state equation \eqref{eq:filmss}.

This zero eigenvalue arises from a \emph{translational perturbation} of the steady state, because choosing
\[
h=H(x+\e)=H(x)+\e H^\prime(x)+O(\e^2)
\]
gives rise to $\phi=H^\prime$ and hence $\psi=H-\overline{H}$.

\subsubsection*{Sufficient condition for linearized instability of nonconstant steady state $H$}
\begin{theorem}[\protect{\cite[Th.~3]{LP2}}] \label{th:filminst}
If $g$ is strictly convex then $\tau_1<0$.
\end{theorem}
For example, the theorem shows that nonconstant steady states are unstable with respect to volume-preserving perturbations if $g(y)=y^p$ with either $p > 1$ or $p<0$.

Incidentally, the theorem is essentially the same as Theorem~\ref{th:reactinst} for the reaction--diffusion PDE, simply writing $g$ instead of $f^\prime$ and noting that our periodic boundary conditions take care of the boundary terms in the integrations by parts.
\begin{proof}
The Rayleigh principle for $L$ says that
\[
\tau_1 = \min \Big\{ \frac{\int_0^X \big( (w^{\prime \prime})^2 - g(H) (w^\prime)^2 \big) \, dx}{\int_0^X w^2 \, dx} :
w \in H^2(\T) \setminus \{ 0 \}, \int_0^X w \, dx = 0 \Big\} .
\]
We choose
\[
w = H^\prime ,
\]
which is not the zero function since $H$ is nonconstant, and note $w$ has mean value zero (as required), by periodicity of $H$. Then the numerator of the Rayleigh quotient for $w$ is
\begin{align*}
& \int_0^X \big( (H^{\prime \prime \prime})^2 - g(H) (H^{\prime \prime})^2 \big) \, dx \\
& = \int_0^X \big( -H^{\prime \prime \prime \prime} - g(H) H^{\prime \prime} \big) H^{\prime \prime} \, dx && \text{by parts} \\
& = \int_0^X g^\prime(H) (H^\prime)^2 H^{\prime \prime} \, dx && \text{by the steady state equation \eqref{eq:filmss}} \\
& = \frac{1}{3} \int_0^X g^\prime(H) \big[ (H^\prime)^3 \big]^\prime \, dx \\
& = - \frac{1}{3} \int_0^X g^{\prime \prime}(H) (H^\prime)^4 \, dx && \text{by parts} \\
& < 0
\end{align*}
by convexity of $g$ and since $H^\prime \not \equiv 0$. Hence $\tau_1 < 0$, by the Rayleigh principle.
\end{proof}

\noindent \emph{Motivation for the choice of trial function.} Our trial function $w=H^\prime$ corresponds to a perturbation $\phi=H^{\prime \prime}$. This perturbation $h=H+\e \phi=H+\e H^{\prime \prime}$ tends to push the steady state towards the constant function. The opposite perturbation would tend to push the steady state towards a ``droplet'' solution that equals $0$ at some point. Thus our instability proof in Theorem~\ref{th:filminst} suggests (in the language of dynamical systems) that a heteroclinic connection might exist between the nonconstant steady state and the constant steady state, and similarly between the nonconstant steady state and a droplet steady state.

\paragraph*{Linear stability of nonconstant steady states.} It is more difficult to prove stability results, because lower bounds on the first eigenvalue are more difficult to prove (generally) than upper bounds.

See \cite[{\S}3.2]{LP2} for some results when $g(y)=y^p,0<p \leq 3/4$, based on time-map monotonicity ideas from the theory of reaction diffusion equations (see Chapter~\ref{ch:reacdiff}).

\subsubsection*{Stability of constant steady states.} Let $\overline{H}>0$ be constant. Then $H \equiv \overline{H}$ is a constant steady state. Its stability is easily determined, as follows.

Linearizing gives
\[
\phi_t = - \phi_{xxxx} - g(\overline{H}) \phi_{xx}
\]
by \eqref{eq:thinlin1}, where the right side is linear and symmetric with \emph{constant} coefficients.

We substitute the periodic Fourier mode $\phi=e^{-\tau t} \exp(2\pi i k x/X)$, where $k \in \Z$ (and $k \neq 0$ since our perturbations have mean value zero), obtaining the eigenvalue
\[
\tau = \big( \frac{2\pi k}{X} \big)^2 \Big( \big( \frac{2\pi k}{X} \big)^2 - g(\overline{H}) \Big) .
\]

If $g \leq 0$ (which means the second order term in the thin film PDE behaves like a forwards heat equation), then $\tau \geq 0$ for each $k$, and so all constant steady states are linearly stable.

If $g>0$ and $\big( \tfrac{2\pi}{X} \big)^2 \geq g(\overline{H})$ then $\tau \geq 0$ for each $k$, and so the constant steady states $\overline{H}$ is linearly stable.

If $g>0$ and $\big( \tfrac{2\pi}{X} \big)^2 < g(\overline{H})$ then the constant steady states $\overline{H}$ is linearly unstable with respect to the $k=1$ mode (and possibly other modes too). In particular, this occurs if $X$ is large enough. Hence we call the thin film PDE ``long-wave unstable'' if $g>0$, since constant steady states are then unstable with respect to perturbations of sufficiently long wavelength.

%
%
%
%

\part{Continuous Spectrum}
\label{part:continuous}

\chapter*{Looking ahead to continuous spectrum}

The discrete spectral theory in Part~\ref{part:discrete} of the course generated, in each application,
\begin{itemize}
\item eigenfunctions $\{ u_j \}$ with ``discrete'' spectrum $\lam_1,\lam_2,\lam_3,\ldots$ satisfying $Lu_j=\lam_j u_j$ where $L$ is a symmetric differential operator, together with
\item a spectral decomposition (or ``resolution'') of each $f \in L^2$ into a sum of eigenfunctions: $f=\sum_j \la f , u_j \ra u_j$.
\end{itemize}
These constructions depended heavily on symmetry of the differential operator $L$ (which ensured symmetry of the sesquilinear form $a$) and on compactness of the imbedding of the Hilbert space $\mathcal{K}$ into $\mathcal{H}$.

For the remainder of the course we retain the symmetry assumption on the operator, but drop the compact imbedding assumption. The resulting ``continuous'' spectrum leads to a decomposition of $f \in L^2$ into an integral of ``almost eigenfunctions''.

We begin with examples, and later put the examples in context by developing some general spectral theory for unbounded, selfadjoint differential operators.

\chapter[Laplacian on whole space]{Computable example: Laplacian (free Schr\"{o}dinger) on all of space}

\label{ch:freeSchr}

\subsubsection*{Goal} To determine for the Laplacian on Euclidean space its continuous spectrum $[0,\infty)$, and the associated spectral decomposition of $L^2$.

\subsubsection*{Spectral decomposition}

The Laplacian $L=-\Delta$ on a bounded domain has discrete spectrum, as we saw in Chapters~\ref{ch:lce} and \ref{ch:aONB}. When the domain expands to be all of space, though, the Laplacian has no eigenvalues at all. For example in $1$ dimension, solutions of $-u^{\prime \prime} =  \lam u$ are linear combinations of $e^{\pm i \sqrt{\lam} x}$, which  oscillates if $\lam>0$, or is constant if $\lam=0$, or grows in one direction or the other if $\lam \in \C \setminus [0,\infty)$. In none of these situations does $u$ belong to $L^2$. (In all dimensions we can argue as follows: if $-\Delta u = \lam u$ and $u \in L^2$ then by taking Fourier transforms, $4\pi^2 |\xi|^2 \widehat{u}(\xi)=\widehat{u}(\xi)$ a.e., and so $\widehat{u}=0$ a.e. Thus no $L^2$-eigenfunctions exist.)

A fundamental difference between the whole space case and the case of bounded domains is that the imbedding $H^1(\Rd) \hookrightarrow L^2(\Rd)$ is not compact. For example, given any nonzero $f \in H^1(\R)$, the functions $f(\cdot - k)$ are bounded in $L^2(\R)$, but have no $L^2$-convergent subsequence as $k \to \infty$. Hence the discrete spectral theorem (Theorem~\ref{th:spec}) is inapplicable.

Nevertheless, the Laplacian $-\Delta$ on $\Rd$ has:
\begin{enumerate}
\item generalized eigenfunctions
\[
v_\omega(x) = e^{2\pi i \omega \cdot x}, \qquad \omega \in \Rd ,
\]
(note that $v_\omega$ is bounded, but it is not an eigenfunction since $v_\omega \not \in L^2$)
which satisfy the eigenfunction equation $-\Delta v_\omega = \lam v_\omega$ with generalized eigenvalue
\[
    \lam = \lam(\omega) = 4\pi^2 |\omega|^2 ,
\]
\begin{figure}[h]
\begin{center}
\includegraphics[width=0.4\textwidth]{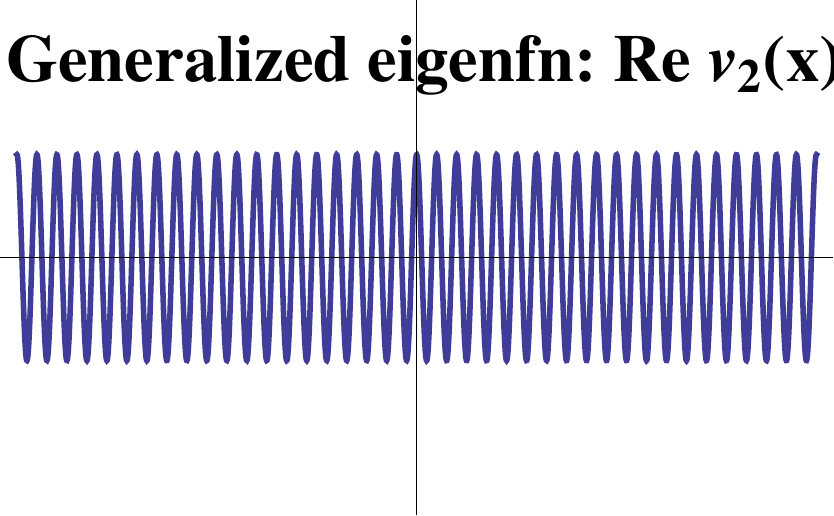}
\end{center}
\end{figure}
\item and a spectral decomposition
\[
f = \int_\Rd \la f , v_\omega \ra \, v_\omega \, d\omega , \qquad \forall f \in L^2(\Rd) .
\]
\end{enumerate}
\emph{Proof of spectral decomposition.} Since
\[
\la f , v_\omega \ra = \int_\Rd f(x) e^{-2\pi i \omega \cdot x} \, dx = \widehat{f}(\omega) ,
\]
the spectral decomposition simply says
\[
f(x) = \int_\Rd \widehat{f}(\omega) e^{2\pi i \omega \cdot x} \, d\omega ,
\]
which is the Fourier inversion formula.

\subsubsection*{Application of spectral decomposition}

One may solve evolution equations by separating variables: for example, the heat equation $u_t=\Delta u$ with initial condition $h(x)$ has solution
\[
u(x,t) = \int_\Rd \widehat{h}(\omega) e^{-\lam(\omega)t} v_\omega(x) \, d\omega .
\]
Note the analogy to the series solution by separation of variables, in the case of discrete spectrum.

\emph{Aside.} Typically, one evaluates the last integral (an inverse Fourier transform) and thus obtains a convolution of the initial data $h$ and the fundamental solution of the heat equation (which is the inverse transform of $e^{-\lam(\omega)t}$).

\subsubsection*{Continuous spectrum $=[0,\infty)$}

The generalized eigenvalue $\lam \geq 0$ is ``almost'' an eigenvalue, in two senses:
\begin{itemize}
\item the eigenfunction equation $(-\Delta - \lam)u=0$ does not have a solution in $L^2$, but it does have a solution $v_\omega \in L^\infty$,
\item a \textbf{Weyl sequence} exists for $-\Delta$ and $\lam$, meaning there exist functions $w_n$ such that
    \begin{itemize}
    \item[(W1)] $\lv (-\Delta-\lam)w_n \rv_{L^2} \to 0$ as $n \to \infty$,
    \item[(W2)] $\lv w_n \rv_{L^2}=1$,
    \item[(W3)] $w_n \rightharpoonup 0$ weakly in $L^2$ as $n \to \infty$.
    \end{itemize}
\end{itemize}
We prove existence of a Weyl sequence in the Proposition below. Later we will define the \textbf{continuous spectrum} to consist of those $\lam$-values for which a Weyl sequence exists. Thus the continuous spectrum of $-\Delta$ is precisely the nonnegative real axis. Recall it is those values of $\lam$ that entered into our spectral decomposition earlier in the chapter.
\begin{figure}[h]
\begin{center}
\includegraphics[width=0.4\textwidth]{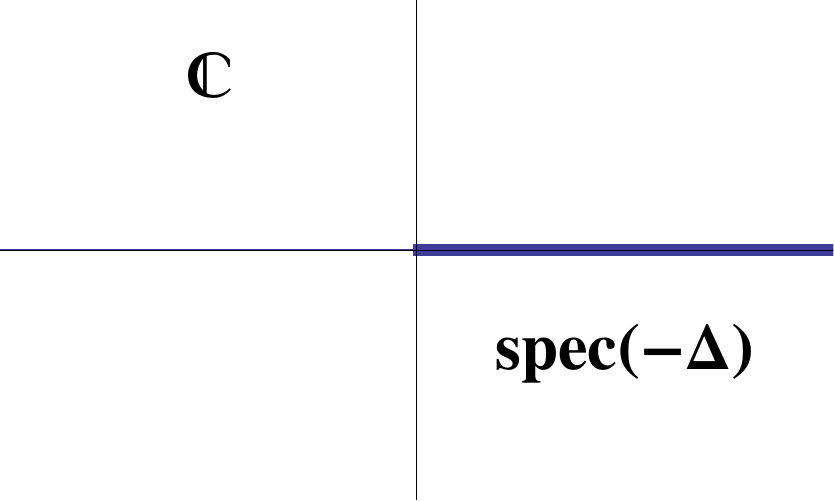}
\end{center}
\end{figure}

\medskip
\noindent \emph{Remark.} Existence of a Weyl sequence ensures that $(-\Delta-\lam)$ does not have a bounded inverse from $L^2 \to L^2$, for if we write $f_n=(-\Delta - \lam)w_n$ then
\[
\frac{\lv (-\Delta-\lam)^{-1}f_n \rv_{L^2}}{\lv f_n \rv_{L^2}} = \frac{\lv w_n \rv_{L^2}}{\lv (-\Delta-\lam)w_n \rv_{L^2}} \to \infty
\]
as $n \to \infty$, by (W1) and (W2). In this way, existence of a Weyl sequence is similar to existence of an eigenfunction, which also prevents invertibility of $(-\Delta-\lam)$.

\begin{proposition}[Weyl sequences for negative Laplacian] \label{le:weylexist}
A Weyl sequence exists for $-\Delta$ and $\lam \in \C$ if and only if $\lam \in [0,\infty)$.
\end{proposition}
\begin{proof}
``$\Longleftarrow$'' Fix $\lam \in [0,\infty)$ and choose $\omega \in \Rd$ with $4\pi |\omega|^2=\lam$. Take a cut-off function $\kappa \in C^\infty_0(\Rd)$ such that $\kappa \equiv 1$ on the unit ball $B(1)$ and $\kappa \equiv 0$ on $\Rd \setminus B(2)$. Define a cut-off version of the generalized eigenfunction, by
\[
w_n = c_n \kappa(\frac{x}{n}) v_\omega(x)
\]
where the normalizing constant is
\[
c_n = \frac{1}{n^{d/2} \lv \kappa \rv_{L^2}} .
\]

\begin{figure}[h]
\begin{center}
\includegraphics[width=0.4\textwidth]{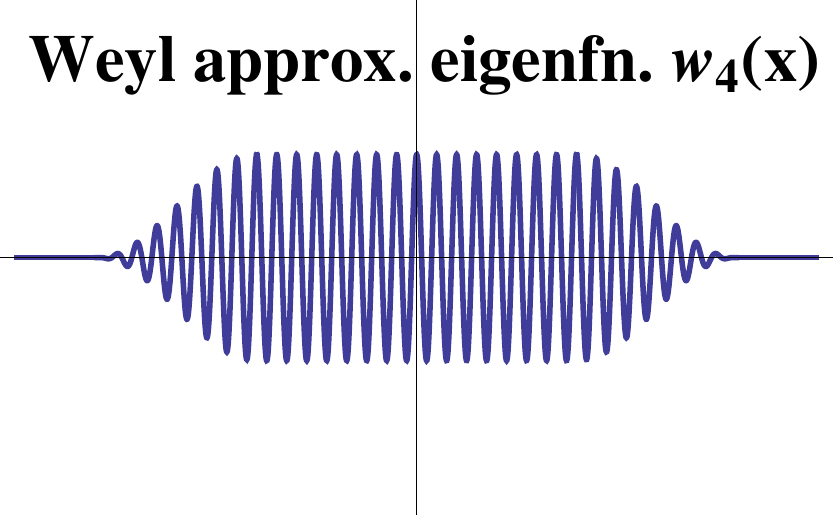}
\end{center}
\end{figure}

First we prove (W1). We have
\begin{align*}
& (\lam+\Delta)w_n \\
& = c_n(\lam v_\omega + \Delta v_\omega) \kappa(\frac{x}{n}) + 2\frac{c_n}{n} \nabla v_\omega(x) \cdot (\nabla \kappa)(\frac{x}{n}) + \frac{c_n}{n^2} v_\omega(x) (\Delta \kappa)(\frac{x}{n}) .
\end{align*}
The first term vanishes because $\Delta v_\omega = -4\pi |\omega|^2 v_\omega$ pointwise. In the third term, note that $v_\omega$ is a bounded function, and that a change of variable shows
\[
\frac{c_n}{n^2} \lv (\Delta \kappa)(\frac{\cdot}{n}) \rv_{L^2} = \frac{1}{n^2} \frac{\lv \Delta \kappa \rv_{L^2}}{\lv \kappa \rv_{L^2}} \to 0 .
\]
The second term similarly vanishes in the limit, as $n \to \infty$. Hence $(\lam+\Delta)w_n \to 0$ in $L^2$, which is (W1).

For (W2) we simply observe that $|v_\omega(x)|=1$ pointwise, so that $\lv w_n \rv_{L^2}=1$ by a change of variable, using the definition of $c_n$.

To prove (W3), take $f \in L^2$ and let $R>0$. We decompose $f$ into ``near'' and ``far'' components, as $f=g+h$ where $g=f 1_{B(R)}$ and $h=f 1_{\Rd \setminus B(R)}$. Then
\[
\la f , w_n \ra_{L^2} = \la g , w_n \ra_{L^2} + \la h , w_n \ra_{L^2} .
\]
We have
\[
\big| \la g , w_n \ra_{L^2} \big| \leq c_n \lv \kappa \rv_{L^\infty} \lv g \rv_{L^1} \to 0
\]
as $n \to \infty$, since $c_n \to 0$. Also, by Cauchy--Schwarz and (W2) we see
\[
\limsup_{n \to \infty} \big| \la h , w_n \ra_{L^2} \big| \leq \lv h \rv_{L^2} .
\]
This last quantity can be made arbitrarily small by letting $R \to \infty$, and so $\lim_{n \to \infty} \la f , w_n \ra_{L^2} = 0$. That is, $w_n \rightharpoonup 0$ weakly.

\medskip
``$\Longrightarrow$'' Assume $\lam \in \C \setminus [0,\infty)$, and let
\[
\delta = \text{dist}\, \big( \lam, [0,\infty) \big)
\]
so that $\delta > 0$.

Suppose (W1) holds, and write $g_n = (-\Delta - \lam)w_n$. Then
\begin{align*}
\widehat{g_n}(\xi) & = (4\pi^2 |\xi|^2 - \lam) \widehat{w_n}(\xi) \\
\widehat{w_n}(\xi) & = \frac{1}{(4\pi^2 |\xi|^2 - \lam)}
\widehat{g_n}(\xi) \\
|\widehat{w_n}(\xi)| & \leq \delta^{-1} |\widehat{g_n}(\xi)| \\
\end{align*}
and hence
\begin{align*}
\lv w_n \rv_{L^2} = \lv \widehat{w_n} \rv_{L^2}
& \leq \delta^{-1} \lv \widehat{g_n} \rv_{L^2} \\
& = \delta^{-1} \lv g_n \rv_{L^2} \\
& \to 0
\end{align*}
by (W1). Thus (W2) does not hold.

\smallskip
(\emph{Aside.} The calculations above show, in fact, that $(-\Delta-\lam)^{-1}$ is bounded from $L^2 \to L^2$ with norm bound $\delta^{-1}$, when $\lam \notin [0,\infty)$.)
\end{proof}

\chapter[Schr\"{o}dinger with $-2\sech^2$ potential]{Computable example: Schr\"{o}dinger with a bounded potential well}

\label{ch:sechsquared}

\subsubsection*{Goal} To show that the Schr\"{o}dinger operator
\[
L= - \tfrac{d^2\ }{d x^2} - 2\sech^2 x
\]
in $1$ dimension has a single negative eigenvalue (discrete spectrum) as well as nonnegative continuous spectrum $[0,\infty)$. The spectral decomposition will show the potential is reflectionless.

\begin{figure}[h]
\begin{center}
\includegraphics[width=0.4\textwidth]{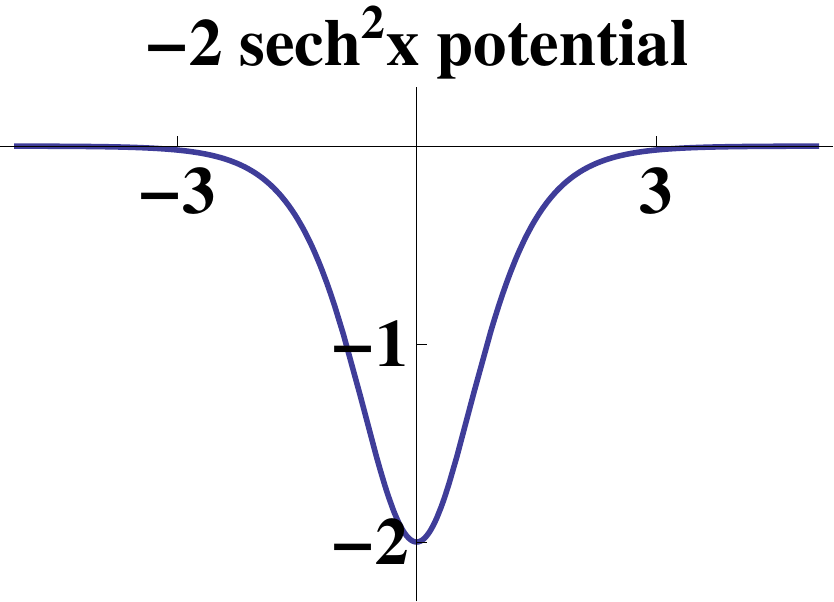} \quad
\includegraphics[width=0.4\textwidth]{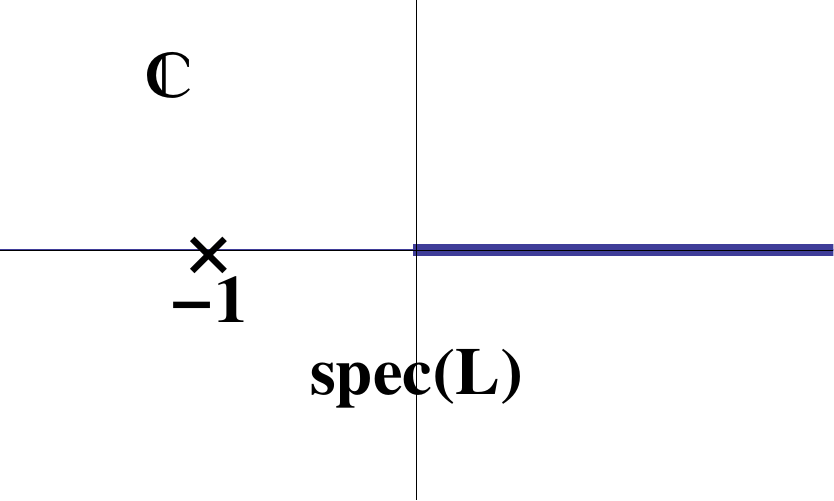}
\end{center}
\end{figure}

\paragraph*{Reference} \cite{Ke} Section 7.5

\subsubsection*{Discrete spectrum = $\{ -1 \}$}

We claim $-1$ is an eigenvalue of $L$ with eigenfunction $\sech x$. This fact can be checked directly, but we will proceed more systematically by factoring the Schr\"{o}dinger operator with the help of the first order operators
\begin{align*}
L^+ & = - \frac{d\ }{d x} + \tanh x , \\
L^- & = \ \frac{d\ }{d x} + \tanh x .
\end{align*}
We compute
\begin{align*}
L^+ L^- - 1
& = \big( - \frac{d\ }{d x} + \tanh x \big) \big( \frac{d\ }{d x} + \tanh x \big) - 1 \\
& = - \frac{d^2\ }{d x^2} - (\tanh x)^\prime + \tanh^2 x - 1 \\
& = - \frac{d^2\ }{d x^2} - 2\sech^2 x \\
& = L
\end{align*}
since $(\tanh)^\prime = \sech^2$ and $1-\tanh^2 = \sech^2$. Thus
\begin{equation} \label{eq:Lcommute}
L = L^+ L^- - 1 .
\end{equation}

It follows that functions in the kernel of $L^-$ are eigenfunctions of $L$ with eigenvalue $\lam=-1$. To find the kernel we solve:
\begin{align*}
L^- v & = 0 \\
v^\prime + (\tanh x)v & = 0 \\
(\cosh x) v^\prime + (\sinh x)v & = 0 \\
(\cosh x)v & = \text{const.} \\
v & = c \sech x
\end{align*}
Clearly $\sech x \in L^2(\R)$, since $\sech$ decays exponentially. Thus $-1$ lies in the discrete spectrum of $L$, with eigenfunction $\sech x$.

\begin{figure}[h]
\begin{center}
\includegraphics[width=0.4\textwidth]{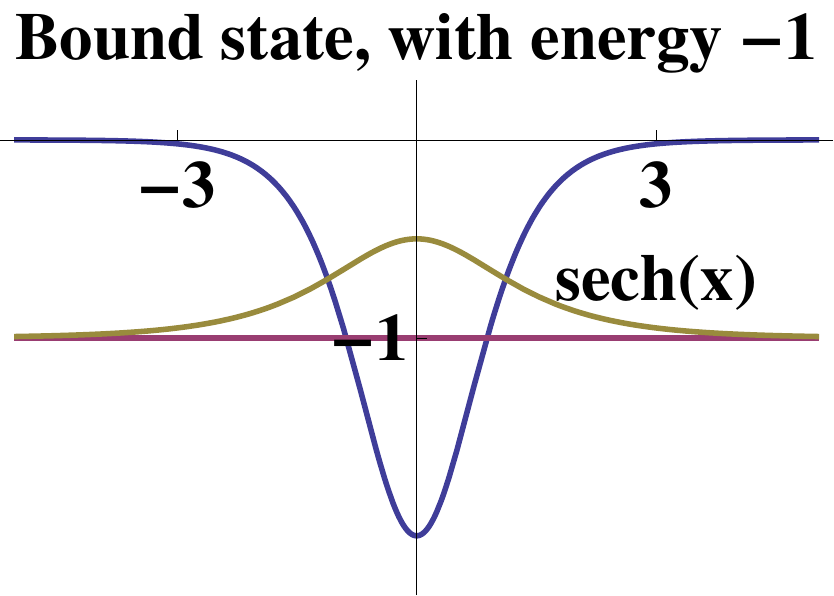}
\end{center}
\end{figure}

Are there any other eigenvalues? No! Argue as follows. By composing $L^+$ and $L^-$ in the reverse order we find
\begin{align}
L^- L^+ - 1
& = \big( \frac{d\ }{d x} + \tanh x \big) \big( - \frac{d\ }{d x} + \tanh x \big) - 1 \notag \\
& = - \frac{d^2\ }{d x^2} + (\tanh x)^\prime + \tanh^2 x - 1 \notag \\
& = - \frac{d^2\ }{d x^2} . \label{eq:reversederiv}
\end{align}
From \eqref{eq:Lcommute} and \eqref{eq:reversederiv} we deduce
\[
- \frac{d^2\ }{d x^2} L^- = L^- L .
\]
Thus if $Lv=\lam v$ then $- \tfrac{d^2\ }{d x^2} (L^- v) = L^- L v = \lam (L^- v)$. By solving for $L^- v$ in terms of $e^{\pm i \sqrt{\lam}x}$, and then integrating to obtain $v$, we conclude after some thought (omitted) that the only way for $v$ to belong to $L^2(\R)$ is to have $L^- v = 0$ and hence $v=c \sech x$, so that $\lam=-1$.

\subsubsection*{Continuous spectrum $\supset [0,\infty)$}

Let $\lam \in [0,\infty)$. Generalized eigenfunctions with $Lv=\lam v$ certainly exist: choose $\omega \in \R$ with $\lam=4\pi^2 \omega^2$ and define
\[
v(x) = L^+ (e^{2\pi i \omega x})  = (\tanh x - 2\pi i\omega) e^{2\pi i \omega x} ,
\]
which is bounded but not square integrable. We compute
\begin{align*}
Lv
& = (L^+ L^- - 1)L^+ (e^{2\pi i \omega x}) && \text{by \eqref{eq:Lcommute}} \\
& = L^+ (L^- L^+ - 1) (e^{2\pi i \omega x}) \\
& = - L^+ \frac{d^2\ }{d x^2} (e^{2\pi i \omega x}) && \text{by \eqref{eq:reversederiv}} \\
& = L^+ (4\pi^2 \omega^2 e^{2\pi i \omega x}) \\
& = \lam v ,
\end{align*}
which verifies that $v(x)$ is a generalized eigenfunction.

We can further prove existence of a Weyl sequence for $L$ and $\lam$ by adapting Lemma~\ref{le:weylexist} ``$\Longleftarrow$'', using the same Weyl functions $w_n(x)$ as for the free Schr\"{o}dinger operator $-\Delta$. The only new step in the proof, for proving $\lv (L-\lam)w_n \rv_{L^2} \to 0$ in (W1), is to observe that
\begin{align*}
|2\sech^2 x \, w_n(x)|
&  = 2 c_n |\kappa(\frac{x}{n}) e^{2\pi i \omega x}| \sech^2 x \\
&  \leq 2 c_n \lv \kappa \rv_{L^\infty} \sech^2 x \\
& \to 0
\end{align*}
in $L^2(\R)$ as $n \to \infty$, because $c_n \to 0$. (\emph{Note.} This part of the proof works not only for the $\sech^2$ potential, but for any potential belonging to $L^2$.)

We have shown that the continuous spectrum contains $[0,\infty)$. We will prove the reverse inclusion at the end of the chapter.

\paragraph*{Generalized eigenfunctions as traveling waves.} The eigenfunction (``bound state'') $v(x)=\sech x$ with eigenvalue (``energy'') $-1$ produces a standing wavefunction
\[
u=e^{it} \sech x
\]
satisfying the time-dependent Schr\"{o}dinger equation
\[
iu_t = Lu .
\]

The generalized eigenfunction
\begin{equation} \label{eq:geneig}
v(x) = (\tanh x - 2\pi i\omega) e^{2\pi i \omega x}
\end{equation}
with generalized eigenvalue $\lam=4 \pi^2 \omega^2$ similarly produces a standing wave
\[
u = e^{- i 4 \pi^2 \omega^2 t} (\tanh x - 2\pi i\omega) e^{2\pi i \omega x} .
\]

\begin{figure}[h]
\begin{center}
\includegraphics[width=0.4\textwidth]{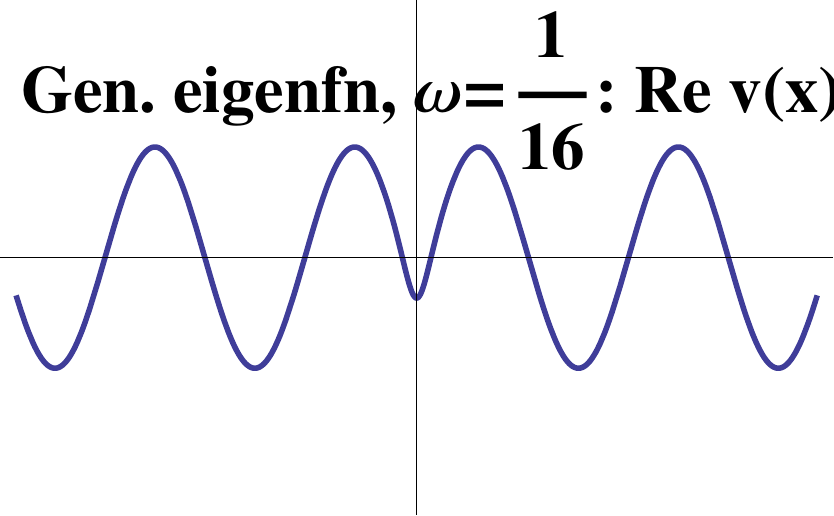}
\end{center}
\end{figure}

More usefully, we rewrite this formula as a traveling plane wave multiplied by an $x$-dependent amplitude:
\begin{equation} \label{eq:rightwave}
u = (\tanh x - 2\pi i\omega) e^{2\pi i \omega (x - 2\pi \omega t)} .
\end{equation}
The amplitude factor serves to quantify the effect of the potential on the traveling wave: in the absence of a potential, the amplitude would be identically $1$, since the plane wave $e^{2\pi i \omega (x - 2\pi \omega t)}$ solves the free Schr\"{o}dinger equation $iu_t = -\Delta u$.

\paragraph*{Reflectionless nature of the potential, and a nod to scattering theory.}
One calls the potential $-2\sech^2 x$ ``reflectionless'' because the right-moving wave in \eqref{eq:rightwave} passes through the potential with none of its energy reflected into a left-moving wave. In other words, the generalized eigenfunction \eqref{eq:geneig} has the form $ce^{2\pi i \omega x}$ both as $x \to -\infty$ and as $x \to \infty$ (with different constants, it turns out, although the constants are equal in magnitude).

This reflectionless property is unusual. A typical Schr\"{o}dinger potential would produce generalized eigenfunctions equalling approximately
\[
c_I e^{2\pi i \omega x} + c_R e^{-2\pi i \omega x} \qquad \text{as $x \to -\infty$}
\]
and
\[
c_T e^{2\pi i \omega x} \qquad \text{as $x \to \infty$}
\]
(or similarly with the roles of $\pm \infty$ interchanged). Here $|c_I|$ is the amplitude of the \emph{incident} right-moving wave, $|c_R|$ is the amplitude of the left-moving wave \emph{reflected} by the potential, and $|c_T|$ is the amplitude of the right-moving wave \emph{transmitted} through the potential. Conservation of $L^2$-energy demands that
\[
|c_I|^2 = |c_R|^2 + |c_T|^2 .
\]
For a gentle introduction to this ``scattering theory'' see \cite[Section 7.5]{Ke}. Then one can proceed to the book-length treatment in \cite{RS3}.

\subsubsection*{Spectral decomposition of $L^2$}
Analogous to an orthonormal expansion in terms of eigenfunctions, we have:
\begin{theorem} \label{th:sechdecomp}
\[
f = \frac{1}{2} \la f , \sech \ra \sech + \int_\R \la f , L^+ v_\omega \ra L^+ v_\omega \, \frac{d\omega}{1+4\pi^2 \omega^2}, \qquad \forall f \in L^2(\R),
\]
where $L^+ v_\omega(x)=(\tanh x - 2\pi i\omega) e^{2\pi i \omega x}$ is the generalized eigenfunction at frequency $\omega$.
\end{theorem}
The discrete part of the decomposition has the same form as the continuous part, in fact, because $\sech=-L^+(\sinh)$.
\begin{proof}
We will sketch the main idea of the proof, and leave it to the reader to make the argument rigorous.

By analogy with an orthonormal expansion in the discrete case, we assume that $f \in L^2(\R)$ has a decomposition in terms of the eigenfunction $\sech x$ and the generalized eigenfunctions $L^+ v_\omega$ in the form
\[
f = c \la f , \sech \ra \sech + \int_\R m_f(\omega) \la f , L^+ v_\omega \ra L^+ v_\omega \, d\omega ,
\]
where the coefficient $c$ and multiplier $m_f(\omega)$ are to be determined.

Taking the inner product with $\sech x$ implies that $c=\tfrac{1}{2}$, since $\lv \sech \rv_{L^2(\R)}^2 = 2$ and $\la L^+ v_\omega , \sech \ra = \la v_\omega , L^- \sech \ra = 0$.

Next we annihilate the $\sech$ term by applying $L^-$ to both sides:
\[
L^- f = L^- \Big( \int_\R m_f(\omega) \la f , L^+ v_\omega \ra L^+ v_\omega \, d\omega \Big) .
\]
Note that by integration by parts,
\[
\la f , L^+ v_\omega \ra = \la L^- f , v_\omega \ra = \widehat{(L^- f)}(\omega) .
\]
Hence
\begin{align*}
L^- f
& = L^- \Big( \int_\R m_f(\omega) \widehat{(L^- f)}(\omega) L^+ v_\omega \, d\omega \Big) \\
& = \int_\R m_f(\omega) \widehat{(L^- f)}(\omega) L^- L^+ v_\omega \, d\omega \\
& = \int_\R m_f(\omega) \widehat{(L^- f)}(\omega) (1+4\pi^2 \omega^2) v_\omega \, d\omega
\end{align*}
by \eqref{eq:reversederiv}. Thus the multiplier should be $m_f(\omega) =1/(1+4\pi^2 \omega^2)$, in order for Fourier inversion to hold. This argument shows the necessity of the formula in the theorem, and one can show sufficiency by suitably reversing the steps.
\end{proof}

The theorem implies a Plancherel type identity.
\begin{corollary} \label{co:sechplancherel}
\[
\lv f \rv_{L^2}^2 = \frac{1}{2} |\la f , \sech \ra|^2 + \int_\R | \la f , L^+ v_\omega \ra |^2 \, \frac{d\omega}{1+4\pi^2 \omega^2}, \qquad \forall f \in L^2(\R) .
\]
\end{corollary}
\begin{proof}
Take the inner product of $f$ with the formula in Theorem~\ref{th:sechdecomp}.
\end{proof}

\subsubsection*{Continuous spectrum $= [0,\infty)$}

Earlier we showed that the continuous spectrum contains $[0,\infty)$. For the reverse containment, suppose $\lam \notin [0,\infty)$ and $\lam \neq -1$. Then $L-\lam$ is invertible on $L^2$, with
\[
(L-\lam)^{-1} f = - \frac{1}{\lam+1} \frac{1}{2} \la f , \sech \ra \sech + \int_\R \frac{\la f , L^+ v_\omega \ra}{4\pi^2 \omega^2 - \lam} L^+ v_\omega \, \frac{d\omega}{1+4\pi^2 \omega^2}
\]
as one sees by applying $L-\lam$ to both sides and recalling Theorem~\ref{th:sechdecomp}. To check the boundedness of this inverse, note that
\begin{align*}
\lv (L-\lam)^{-1} f \rv_{L^2}^2
& = \frac{1}{|\lam+1|^2} \frac{1}{2} |\la f , \sech \ra|^2 + \int_\R \frac{| \la f , L^+ v_\omega \ra |^2}{|4\pi^2 \omega^2 - \lam|^2} \, \frac{d\omega}{1+4\pi^2 \omega^2} \\
& \leq \frac{1}{|\lam+1|^2} \frac{1}{2} |\la f , \sech \ra|^2 + \frac{1}{\text{dist} \, \big(\lam,[0,\infty) \big)^2} \int_\R | \la f , L^+ v_\omega \ra |^2 \, \frac{d\omega}{1+4\pi^2 \omega^2} \\
& \leq (\text{const.}) \lv f \rv_{L^2}^2 ,
\end{align*}
where we used Corollary~\ref{co:sechplancherel}.

The boundedness of $(L-\lam)^{-1}$ implies that the Weyl conditions (W1) and (W2) cannot both hold. Thus no Weyl sequence can exist for $\lam$, so that $\lam$ does not belong to the continuous spectrum.

Next suppose $\lam = -1$. If a Weyl sequence $w_n$ exists, then
\[
\la w_n , \sech \ra_{L^2} \to 0 \qquad \text{as $n \to \infty$,}
\]
by the weak convergence in (W3). Hence if we project away from the $\lam=-1$ eigenspace by defining
\[
y_n = w_n - \frac{1}{2} \la w_n , \sech \ra_{L^2} \sech \qquad \text{and} \qquad z_n = y_n/\lv y_n \rv_{L^2} ,
\]
then we find $\lv y_n \rv_{L^2} \to 1$ and $\lv z_n \rv_{L^2} = 1$, with $\la z_n , \sech \ra_{L^2} = 0$. Also
\[
(L+1)z_n = (L+1)y_n/\lv y_n \rv_{L^2} = (L+1)w_n/\lv y_n \rv_{L^2} \to 0
\]
in $L^2$. Thus $z_n$ satisfies (W1) and (W2) and lies in the orthogonal complement of the eigenspace spanned by $\sech$. A contradiction now follows from the boundedness of $(L+1)^{-1}$ on that orthogonal complement (with the boundedness being proved by the same argument as above for $\lam \neq -1$). This contradiction shows that no such Weyl sequence $w_n$ can exist, and so $-1$ does not belong to the continuous spectrum.

\medskip \emph{Note.} The parallels with our derivation of the continuous spectrum for the Laplacian in Chapter~\ref{ch:freeSchr} are instructive.

\chapter[Selfadjoint operators]{Selfadjoint, unbounded linear operators}

\label{ch:selfadjoint}

\subsubsection*{Goal} To develop the theory of unbounded linear operators on a Hilbert space, and to define selfadjointness for such operators.

\paragraph*{References} \cite{GS} Sections 1.5, 2.4

\cite{HS} Chapters 4, 5

\subsubsection*{Motivation}
Now we should develop some general theory, to provide context for the examples computed in Chapters~\ref{ch:freeSchr} and \ref{ch:sechsquared}.

We begin with a basic principle of calculus:
\begin{quote}
integration makes functions better, while differentiation makes them worse.
\end{quote}
More precisely, integral operators are bounded (generally speaking), while differential operators are unbounded. For example, $e^{2\pi i inx}$ has norm $1$ in $L^2[0,1]$ while its derivative $\tfrac{d\ }{dx} e^{2\pi i inx}=2\pi in e^{2\pi inx}$ has norm that grows with $n$. The unboundedness of such operators prevents us from applying the spectral theory of bounded operators on a Hilbert space.

Further, differential operators are usually defined only on a (dense) subspace of our natural function spaces. In particular, we saw in our study of discrete spectra that the Laplacian is most naturally studied using the Sobolev space $H^1$, even though the Laplacian involves two derivatives and $H^1$-functions are guaranteed only to possess a single derivative.

To meet these challenges, we will develop the theory of densely defined, unbounded linear operators, along with the notion of adjoints and selfadjointness for such operators.

\subsubsection*{Domains and inverses of (unbounded) operators}

Take a complex Hilbert space $\mathcal{H}$ with inner product $\la \cdot , \cdot \ra$. Suppose $A$ is a linear operator (not necessarily bounded) from a subspace $D(A) \subset \mathcal{H}$ into $\mathcal{H}$:
\[
A : D(A) \to \mathcal{H} .
\]
Call $D(A)$ the \textbf{domain} of $A$.

An operator $B$ with domain $D(B)$ is called the \textbf{inverse} of $A$ if
\begin{itemize}
\item $D(B) = \Ran(A), D(A)=\Ran(B)$, and
\item $BA=\text{id}_{\Ran(B)}, AB=\text{id}_{\Ran(A)}$.
\end{itemize}
Write $A^{-1}$ for this inverse, if it exists. Obviously $A^{-1}$ is unique, if it exists, because in that case $A$ is bijective.

Further say $A$ is \textbf{invertible} if $A^{-1}$ exists and is bounded on $\mathcal{H}$ (meaning that $A^{-1}$ exists, $\Ran(A)=\mathcal{H}$, and $A^{-1}:\mathcal{H} \to \mathcal{H}$ is a bounded linear operator).

\medskip
\noindent \emph{Example.} Consider the operator $A=-\Delta+1$ with domain $H^2(\Rd) \subset L^2(\Rd)$. Invertibility is proved using the Fourier transform: let $D(B)=L^2(\Rd)$, and define a bounded operator $B : L^2 \to L^2$ by
\[
\widehat{Bf}(\xi) = (1+4\pi^2 |\xi|^2)^{-1} \widehat{f}(\xi) .
\]
One can check that $\Ran(B) = H^2(\Rd) = D(A)$. Notice $BA=\text{id}_{H^2}, AB=\text{id}_{L^2}$. The second identity implies that $\Ran(A)=L^2=D(B)$.

\subsubsection*{Adjoint of an (unbounded) operator}

Call $A$ \textbf{symmetric} if
\begin{equation} \label{eq:symmetric}
\la Af,g \ra = \la f,Ag \ra , \qquad \forall f,g \in D(A) .
\end{equation}
Symmetry is a simpler concept than \textbf{selfadjointness}, which requires the operator and its adjoint to have the same domain, as we now explain.

First we define a subspace
\[
D(A^*) = \{ f \in \mathcal{H} : \text{the linear functional $g \mapsto \la f,Ag \ra$ is bounded on $D(A)$} \} .
\]
Assume from now on that $A$ is \textbf{densely defined}, meaning $D(A)$ is dense in $\mathcal{H}$. Then for each $f \in D(A^*)$, the bounded linear functional $g \mapsto \la f,Ag \ra$ is defined on a dense subspace of $\mathcal{H}$ and hence extends uniquely to a bounded linear functional on all of $\mathcal{H}$. By the Riesz Representation Theorem, that linear functional can be represented as the inner product of $g$ against a unique element of $\mathcal{H}$, which we call $A^* f$. Hence
\begin{equation} \label{eq:selfadjoint}
 \la f,Ag \ra = \la A^* f, g \ra , \qquad \forall f \in D(A^*), \quad g \in D(A)  .
\end{equation}
Clearly this operator $A^* : D(A^*) \to \mathcal{H}$ is linear. We call it the \textbf{adjoint} of $A$.

\begin{lemma}
If $A$ is a densely defined linear operator and $\lam \in \C$, then $(A-\lam)^*=A^* - \overline{\lam}$.
\end{lemma}
We leave the (easy) proof to the reader. Implicit in the proof is that domains are unchanged by subtracting a constant: $D(A-\lam)=D(A)$ and $D\big(( (A-\lam)^* \big) = D(A^*)$.

The kernel of the adjoint complements the range of the original operator, as follows.
\begin{proposition} \label{pr:complement}
If $A$ is a densely defined linear operator then $\overline{\Ran(A)} \oplus \ker(A^*) = \mathcal{H}$.
\end{proposition}
\begin{proof}
Clearly $\ker(A^*) \subset \Ran(A)^\perp$, because if $f \in \ker(A^*)$ then $A^*f = 0$ and so for all $g \in D(A)$ we have
\[
\la f,Ag \ra = \la A^* f,g \ra = 0 .
\]

To prove the reverse inclusion, $\Ran(A)^\perp \subset \ker(A^*)$, suppose $h \in \Ran(A)^\perp$. For all $g \in D(A)$ we have $\la h,Ag \ra = 0$. In particular, $h \in D(A^*)$. Hence
\[
\la A^* h,g \ra = \la h,Ag \ra = 0 \qquad \forall g \in D(A) ,
\]
and so from density of $D(A)$ we conclude $A^* h = 0$. That is, $h \in \ker(A^*)$.

We have shown $\Ran(A)^\perp = \ker(A^*)$, and so (since the orthogonal complement is unaffected by taking the closure) $\overline{\Ran(A)}^\perp = \ker(A^*)$. The proposition follows immediately.
\end{proof}

We will need later that the graph of the adjoint, $\{ (f,A^* f) : f \in D(A^*) \}$, is closed in $\mathcal{H} \times \mathcal{H}$.
\begin{theorem} \label{pr:closedadjoint}
If $A$ is a densely defined linear operator then $A^*$ is a closed operator.
\end{theorem}
\begin{proof}
Suppose $f_n \in D(A^*)$ with $f_n \to f ,A^* f_n \to g$, for some $f,g \in \mathcal{H}$. To prove the graph of $A^*$ is closed, we must show $f \in D(A^*)$ with $A^* f=g$.

For each $h \in D(A)$ we have
\[
\la f,Ah \ra = \lim_n \la f_n, Ah\ra = \lim_n \la A^* f_n, h \ra = \la g,h \ra .
\]
Thus the map $h \mapsto \la f , Ah \ra$ is bounded for $h \in D(A)$. Hence $f \in D(A^*)$, and using the last calculation we see
\[
\la A^* f, h \ra = \la f,Ah \ra = \la g,h \ra
\]
for all $h \in D(A)$. Density of the domain implies $A^* f = g$, as we wanted.
\end{proof}

\subsubsection*{Selfadjointness}

Call $A$ \textbf{selfadjoint} if $A^*=A$, meaning $D(A^*)=D(A)$ and $A^*=A$ on their common domain.

Selfadjoint operators have closed graphs, due to closedness of the adjoint in Theorem~\ref{pr:closedadjoint}. Thus:
\begin{proposition} \label{pr:closedself}
If a densely defined linear operator $A$ is selfadjoint then it is closed.
\end{proposition}

The relation between selftadjointness and symmetry is clear:
\begin{proposition}
The densely defined linear operator $A$ is selfadjoint if and only if it is symmetric and $D(A)=D(A^*)$.
\end{proposition}
\begin{proof}
``$\Longrightarrow$''
If $A^*=A$ then the adjoint relation \eqref{eq:selfadjoint} reduces immediately to the symmetry relation \eqref{eq:symmetric}.

\medskip
``$\Longleftarrow$''  The symmetry relation \eqref{eq:symmetric} together with the adjoint relation \eqref{eq:selfadjoint} implies that $\la Af, g \ra = \la A^* f, g \ra$ for all $f,g \in D(A)=D(A^*)$. Since $D(A)$ is dense in $\mathcal{H}$, we conclude $Af=A^* f$.
\end{proof}
For bounded operators, selfadjointness and symmetry are equivalent.
\begin{lemma}
If a linear operator $A$ is bounded on $\mathcal{H}$, then it is selfadjoint if and only if it is symmetric.
\end{lemma}
\begin{proof}
Boundedness of $A$ ensures that $D(A^*)=\mathcal{H}=D(A)$, and so the adjoint relation \eqref{eq:selfadjoint} holds for all $f,g \in \mathcal{H}$. Thus $A^*=A$ is equivalent to symmetry.
\end{proof}

\subsubsection*{Example: selfadjointness for Schr\"{o}dinger operators}

Let $L=-\Delta+V$ be a Schr\"{o}dinger operator with potential $V(x)$ that is bounded and real-valued. Choose the domain to be $D(L)=H^2(\Rd)$ in the Hilbert space $L^2(\Rd)$. This Schr\"{o}dinger operator is selfadjoint.
\begin{proof}
Density of $D(L)$ follows from density in $L^2$ of the smooth functions with compact support.

Our main task is to determine the domain of $L^*$. Fix $f,g \in H^2(\Rd)$. From the integration by parts formula $\la f,\Delta g \ra_{L^2}=\la \Delta f,g \ra_{L^2}$ (which one may alternatively prove with the help of the Fourier transform), one deduces that
\[
| \la f, \Delta g \ra_{L^2} | = | \la \Delta f, g \ra_{L^2} | \leq \lv f \rv_{H^2} \lv g \rv_{L^2} .
\]
Also $| \la f, Vg \ra_{L^2} | \leq \lv f \rv_{L^2} \lv V \rv_{L^\infty} \lv g \rv_{L^2}$.
Hence the linear functional $g \mapsto \la f , Lg \ra_{L^2}$ is bounded on $g \in D(L)$. Therefore $f \in D(L^*)$, which tells us $H^2(\Rd) \subset D(L^*)$.

To prove the reverse inclusion, fix $f \in D(L^*)$. Then
\[
| \la f , Lg \ra_{L^2} | \leq (\text{const.}) \lv g \rv_{L^2} , \qquad \forall g \in D(L) = H^2(\Rd) .
\]
Since the potential $V$ is bounded, the last formula still holds if we replace $V$ with $1$, so that
\[
| \la f , (-\Delta+1) g \ra_{L^2} | \leq (\text{const.}) \lv g \rv_{L^2} , \qquad \forall g \in H^2(\Rd) .
\]
Taking Fourier transforms gives
\[
| \la \widehat{f} , (1+4\pi^2|\xi|^2) \widehat{g} \ra_{L^2} | \leq (\text{const.}) \lv \widehat{g} \rv_{L^2} , \qquad \forall g \in H^2(\Rd) .
\]
In particular, we may suppose $\widehat{g} = h \in C^\infty_0(\Rd)$, since every such $\widehat{g}$ gives $g \in H^2(\Rd)$. Hence
\[
| \la (1+4\pi^2|\xi|^2) \widehat{f} , h \ra_{L^2} | \leq (\text{const.}) \lv h \rv_{L^2} , \qquad \forall h \in C^\infty_0(\Rd) .
\]
Taking the supremum of the left side over all $h$ with $L^2$-norm equal to $1$ shows that
\[
\lv(1+4\pi^2|\xi|^2) \widehat{f} \rv_{L^2} \leq (\text{const.})
\]
Hence $(1+|\xi|)^2 \, \widehat{f} \in L^2(\Rd)$, which means $f \in H^2(\Rd)$. Thus $D(L^*) \subset H^2(\Rd)$.

Now that we know the domains of $L$ and $L^*$ agree, we have only to check symmetry, and that is straightforward. When $f,g \in H^2(\Rd)$ we have
\begin{align*}
\la Lf,g \ra
& = - \la \Delta f,g \ra_{L^2} + \la Vf,g \ra_{L^2} \\
& = - \la f, \Delta g \ra_{L^2} + \la f,Vg \ra_{L^2} \\
& = \la f , Lg \ra_{L^2}
\end{align*}
where we integrated by parts and used that $V(x)$ is real-valued.
\end{proof}

\chapter{Spectra: discrete and continuous}

\label{ch:spectrum}

\subsubsection*{Goal} To develop the spectral theory of selfadjoint unbounded linear operators.

\paragraph*{References} \cite{GS} Sections 2.4, 5.1

\cite{HS} Chapters 1, 5, 7

\cite{Ru} Chapter 13

\subsection*{Resolvent set, and spectrum}

Let $A$ be a densely defined linear operator on a complex Hilbert space $\mathcal{H}$, as in the preceding chapter. The operator $A-\lam$ has domain $D(A)$, for each constant $\lam \in \C$. Define the \textbf{resolvent set}
\[
\res(A) = \{ \lam \in \C : \text{$A-\lam$ is invertible (has a bounded inverse defined on $\mathcal{H}$)} \} .
\]
For $\lam$ in the resolvent set, we call the inverse $(A-\lam)^{-1}$ the \textbf{resolvent operator}.

The \textbf{spectrum} is defined as the complement of the resolvent set:
\[
\spec(A) = \C \setminus \res(A).
\]
For example, if $\lam$ is an eigenvalue of $A$ then $\lam \in \spec(A)$, because if $Af=\lam f$ for some $f \neq 0$, then $(A-\lam)f=0$ and so $A-\lam$ is not injective, and hence is not invertible.

\begin{proposition}[\protect{\cite[Theorem 1.2]{HS}}]
The resolvent set is open, and hence the spectrum is closed.
\end{proposition}
We omit the proof.

The next result generalizes the fact that Hermitian matrices have only real eigenvalues.
\begin{theorem} \label{th:reality}
If $A$ is selfadjoint then its spectrum is real: $\spec(A) \subset \R$.
\end{theorem}
\begin{proof}
We prove the contrapositive. Suppose $\lam \in \C$ has nonzero imaginary part, $\Im \lam \neq 0$. We will show $\lam \in \res(A)$.

The first step is to show $A-\lam$ is injective. For all $f \in D(A)$,
\[
\lv (A-\lam)f \rv^2 = \lv Af \rv^2 - 2 (\Re \lam) \la f,Af \ra + |\lam|^2 \lv f \rv^2
\]
and so
\begin{align}
\lv (A-\lam)f \rv^2
& \geq \lv Af \rv^2 - 2 |\Re \lam| \lv f \rv \lv Af \rv + |\lam|^2 \lv f \rv^2 \notag \\
& = \big( \lv Af \rv - |\Re \lam| \lv f \rv \big)^2 + |\Im \lam|^2 \lv f \rv^2 \notag \\
& \geq |\Im \lam|^2 \lv f \rv^2 . \label{eq:inverseest}
\end{align}
The last inequality implies that $A-\lam$ is injective, using here that $|\Im \lam| > 0$. That is, $\ker(A-\lam)=\{ 0 \}$.

Selfadjointness ($A^*=A$) now gives $\ker(A^* - \lam)=0$, and so $\overline{\Ran(A-\lam)}=\mathcal{H}$ by Proposition~\ref{pr:complement}. That is, $A-\lam$ has dense range.

Next we show $\Ran(A-\lam)=\mathcal{H}$. Let $g \in \mathcal{H}$. By density of the range, we may take a sequence $f_n \in D(A)$ such that $(A-\lam)f_n \to g$. The sequence $f_n$ is Cauchy, in view of \eqref{eq:inverseest}. Hence the sequence $(f_n,(A-\lam)f_n)$ is Cauchy in $\mathcal{H} \times \mathcal{H}$, and so converges to $(f,g)$ for some $f \in \mathcal{H}$. Note each ordered pair $(f_n,(A-\lam)f_n)$ lies in the graph of $A-\lam$, and this graph is closed by Proposition~\ref{pr:closedself} (relying here on selfadjointness  again). Therefore $(f,g)$ belongs to the graph of $A-\lam$, and so $g \in \Ran(A-\lam)$. Thus $A-\lam$ has full range.

To summarize: we have shown $A-\lam$ is injective and surjective, and so it has an inverse operator
\[
(A-\lam)^{-1} : \mathcal{H} \to D(A) \subset \mathcal{H} .
\]
This inverse is bounded with
\[
\lv (A-\lam)^{-1}g \rv \leq |\Im \lam|^{-1} \lv g \rv , \qquad \forall g \in \mathcal{H} ,
\]
by taking $f=(A-\lam)^{-1}g$ in estimate \eqref{eq:inverseest}. The proof is thus complete.
\end{proof}

\subsubsection*{Characterizing the spectrum}

We will characterize the spectrum in terms of approximate eigenfunctions. Given a number $\lam \in \C$ and a sequence $w_n \in D(A)$, consider three conditions:
 \begin{itemize}
    \item[(W1)] $\lv (A-\lam)w_n \rv_\mathcal{H} \to 0$ as $n \to \infty$,
    \item[(W2)] $\lv w_n \rv_\mathcal{H}=1$,
    \item[(W3)] $w_n \rightharpoonup 0$ weakly in $\mathcal{H}$ as $n \to \infty$.
 \end{itemize}
(We considered these conditions in Chapter~\ref{ch:freeSchr} for the special case of the Laplacian).

Condition (W1) says $w_n$ is an ``approximate eigenfunction'', and condition (W2) simply normalizes the sequence. These conditions characterize the spectrum, for a selfadjoint operator.
\begin{theorem} \label{th:basicchar}
If $A$ is selfadjoint then
\[
\spec(A) = \{  \lam \in \C : \text{(W1) and (W2) hold for some sequence $w_n \in D(A)$} \} .
\]
\end{theorem}
\begin{proof}
``$\supset$'' Assume (W1) and (W2) hold for $\lam$, and that $A-\lam$ has an inverse defined on $\mathcal{H}$. Then for $f_n=(A - \lam)w_n$ we find
\[
\frac{\lv (A-\lam)^{-1}f_n \rv_\mathcal{H}}{\lv f_n \rv_\mathcal{H}} = \frac{\lv w_n \rv_\mathcal{H}}{\lv (A-\lam)w_n \rv_\mathcal{H}} \to \infty
\]
as $n \to \infty$, by (W1) and (W2). Thus the inverse operator is not bounded, and so $\lam \in \spec(A)$.

``$\subset$'' Assume $\lam \in \spec(A)$, so that $\lam$ is real by Theorem~\ref{th:reality}. If $\lam$ is an eigenvalue, say with normalized eigenvector $f$, then we simply choose $w_n=f$ for each $n$, and (W1) and (W2) hold trivially.

Suppose $\lam$ is not an eigenvalue. Then $A-\lam$ is injective, hence so is $(A-\lam)^*$, which equals $A-\lam$ by selfadjointness of $A$ and reality of $\lam$. Thus $\ker \big( (A-\lam)^* \big) = \{ 0 \}$, and so $\Ran(A-\lam)$ is dense in $\mathcal{H}$ by Proposition~\ref{pr:complement}.

Injectivity ensures that $(A-\lam)^{-1}$ exists on $\Ran(A-\lam)$. If it is unbounded there, then we may choose a sequence $f_n \in \Ran(A-\lam)$ with $\lv (A-\lam)^{-1}f_n \rv_\mathcal{H} = 1$ and $\lv f_n \rv_\mathcal{H} \to 0$. Letting $w_n=(A-\lam)^{-1}f_n$ gives (W1) and (W2) as desired. Suppose on the other hand that $(A-\lam)^{-1}$ is bounded on $\Ran(A-\lam)$. Then the argument in the proof of Theorem~\ref{th:reality} shows that $\Ran(A-\lam)=\mathcal{H}$, which means $\lam$ belongs to the resolvent set, and not the spectrum. Thus this case cannot occur.
\end{proof}

\subsection*{Discrete and continuous spectra}

Define the \textbf{discrete spectrum}
\begin{align*}
& \spec_{disc}(A) \\
& = \{ \lam \in \spec(A) :\text{$\lam$ is an isolated eigenvalue of $A$ having finite multiplicity} \} ,
\end{align*}
where ``isolated'' means that some neighborhood of $\lam$ in the complex plane intersects $\spec(A)$ only at $\lam$. By ``multiplicity'' we mean the geometric multiplicity (dimension of the eigenspace); if $A$ is not selfadjoint then we should use instead the algebraic multiplicity \cite{HS}.

Next define the \textbf{continuous spectrum}
\begin{align*}
& \spec_{cont}(A) \\
& = \{  \lam \in \C : \text{(W1), (W2) and (W3) hold for some sequence $w_n \in D(A)$} \} .
\end{align*}
The continuous spectrum lies within the spectrum, by Theorem~\ref{th:basicchar}. The characterization in that theorem required only (W1) and (W2), whereas the continuous spectrum imposes in addition the ``weak convergence'' condition (W3).

A \textbf{Weyl sequence} for $A$ and $\lam$ is a sequence $w_n \in D(A)$ such that (W1), (W2) and (W3) hold. Thus the preceding definition says the continuous spectrum consists of $\lam$-values for which Weyl sequences exist.

The continuous spectrum can contain eigenvalues that are not isolated (``imbedded eigenvalues'') or which have infinite multiplicity.

\medskip
A famous theorem of Weyl says that for selfadjoint operators, the entire spectrum is covered by the discrete and continuous spectra.
\begin{theorem} \label{th:weyldecomp}
If $A$ is selfadjoint then
\[
\spec(A) = \spec_{disc}(A) \cup \spec_{cont}(A) .
\]
(Further, the discrete and continuous spectra are disjoint.)
\end{theorem}
We omit the proof. See \cite[Theorem 7.2]{HS}.

\subsection*{Applications to Schr\"{o}dinger operators}

The continuous spectrum of the Laplacian $-\Delta$ equals $[0,\infty)$, and the spectrum contains no eigenvalues, as we saw in Chapter~\ref{ch:freeSchr}.

The hydrogen atom too has continuous spectrum $[0,\infty)$, with its Schr\"{o}dinger operator $L=-\Delta-2/|x|$ on $\R^3$ having domain $H^2(\Rd) \subset L^2(\Rd)$; see \cite[Section 8.7]{T}. The discrete spectrum $\{ -1/n^2 : n \geq 1 \}$ of the hydrogen atom was stated in Chapter~\ref{ch:soce}.

As the hydrogen atom example suggests, potentials vanishing at infinity generate continuous spectrum that includes all nonnegative numbers:
\begin{theorem} \label{th:kato}
Assume $V(x)$ is real-valued, continuous, and vanishes at infinity ($V(x) \to 0$ as $|x| \to \infty$).

Then the Schr\"{o}dinger operator $-\Delta + V$ is selfadjoint (with domain $H^2(\Rd) \subset L^2(\Rd)$) and has continuous spectrum $=[0,\infty)$.
\end{theorem}
For a proof see \cite[Corollary 14.10]{HS}, where a stronger theorem is proved that covers also the Coulomb potential $-2/|x|$ for the hydrogen atom. Note the Coulomb potential vanishes at infinity but is discontinuous at the origin, where it blows up. The stronger version of the theorem requires (instead of continuity and vanishing at infinity) that for each $\e>0$, the potential $V(x)$ be decomposable as $V=V_2+V_\infty$ where $V_2 \in L^2$ and $\lv V_\infty \rv_{L^\infty} < \e$. This decomposition can easily be verified for the Coulomb potential, by ``cutting off'' the potential near infinity.

Theorem~\ref{th:kato} implies that any isolated eigenvalues of $L$ must lie on the negative real axis (possibly accumulating at $0$). For example, the $-2\sech^2$ potential in Chapter~\ref{ch:sechsquared} generates a negative eigenvalue at $-1$.

\subsubsection*{Connection to generalized eigenvalues and eigenfunctions}

Just as the discrete spectrum is characterized by eigenfunctions in $L^2$, so the full spectrum is characterized by existence of a generalized eigenfunction that grows at most polynomially at infinity.
\begin{theorem}
Assume $V(x)$ is real-valued and bounded on $\Rd$. Then the Schr\"{o}dinger operator $-\Delta + V$ has spectrum
\begin{align*}
\spec(-\Delta+V) & = \\
\text{closure of\ } & \{ \lam \in \C : \text{$(-\Delta+V)u=\lam u$ for some polynomially bounded $u$} \} .
\end{align*}
\end{theorem}
We omit the proof; see \cite[Theorem 5.22]{GS}.

\subsubsection*{Further reading}

A wealth of information on spectral theory, especially for Schr\"{o}dinger operators, can be found in the books \cite{GS,HS,RS2,RS4}.

\chapter{Discrete spectrum revisited}

\label{ch:revisited}

\subsubsection*{Goal} To fit the discrete spectral Theorem~\ref{th:spec} (from Part~\ref{part:discrete} of the course) into the spectral theory of selfadjoint operators and, in particular, to prove the absence of continuous spectrum in that situation.

\subsection*{Discrete spectral theorem}

The discrete spectral Theorem~\ref{th:spec} concerns a symmetric,
elliptic, bounded sesquilinear form $a(u,v)$ on an infinite dimensional
Hilbert space $\mathcal{K}$, where $\mathcal{K}$ imbeds compactly and densely into the Hilbert
space $\mathcal{H}$. The theorem guarantees existence of an ONB for $\mathcal{H}$ consisting
of eigenvectors of $a$:
\[
a(u_j,v) = \gamma_j \la u_j , v \ra_\mathcal{H} \qquad \forall v \in \mathcal{K} ,
\]
where the eigenvalues satisfy
\[
0 < \gamma_1 \leq \gamma_2 \leq \gamma_3 \leq \cdots \to \infty .
\]
We want to interpret these eigenvalues as the discrete spectrum of some
selfadjoint, densely defined linear operator on $\mathcal{H}$. By doing so, we will
link the discrete spectral theory in Part~\ref{part:discrete} of the
course with the spectral theory of unbounded operators in
Part~\ref{part:continuous}.

Our tasks are to identify the operator $A$ and its domain, to prove $A$ is
symmetric, to determine the domain of the adjoint, to conclude
selfadjointness, and finally to show that the spectrum of $A$ consists
precisely of the eigenvalues $\gamma_j$.

\subsubsection*{Operator $A$ and its domain}
In the proof of Theorem~\ref{th:spec} we found a bounded, selfadjoint
linear operator $B : \mathcal{H} \to \mathcal{K} \subset \mathcal{H}$ with eigenvalues $1/\gamma_j$ and eigenvectors $u_j$:
\[
B u_j = \frac{1}{\gamma_j} u_j .
\]
We showed $B$ is injective (meaning its eigenvalues are nonzero). Notice
$B$ has dense range because its eigenvectors $u_j$ span $\mathcal{H}$.

(\emph{Aside.} This operator $B$ relates to the sesquilinear form $a$ by satisfying $a(Bf,v) = \la f , v \ra_\mathcal{H}$ for all $v \in \mathcal{K}$. We will not need that formula below.)

Define
\[
A = B^{-1} : \Ran(B) \to \mathcal{H} .
\]
Then $A$ is a linear operator, and its domain
\[
D(A) = \Ran(B)
\]
is dense in $\mathcal{H}$.

\subsubsection*{Symmetry of $A$}
Let $u,v \in D(A)$. Then
\begin{align*}
\la Au,v \ra_\mathcal{H}
& = \la Au,BAv \ra_\mathcal{H} && \text{since $BA=\text{Id}$,} \\
& = \la BAu,Av \ra_\mathcal{H} && \text{since $B$ is selfadjoint,} \\
& = \la u,Av \ra_\mathcal{H} && \text{since $BA=\text{Id}$.}
\end{align*}

\subsubsection*{Domain of the adjoint}
First we show $D(A) \subset D(A^*)$. Let $u \in D(A)$. For all $v \in
D(A)$ we have
\begin{align*}
| \la u,Av \ra_\mathcal{H} |
& = | \la Au,v \ra_\mathcal{H} |  && \text{by symmetry} \\
& \leq \lv Au \rv_\mathcal{H} \lv v \rv_\mathcal{H} .
\end{align*}
Hence the functional $v \mapsto \la u,Av \ra_\mathcal{H}$ is bounded on $D(A)$ with
respect to the $\mathcal{H}$-norm, so that $u$ belongs to the domain of the adjoint
$A^*$.

Next we show $D(A^*) \subset D(A)$. Let $u \in D(A^*) \subset \mathcal{H}$. We have
\[
| \la u,Av \ra_\mathcal{H} | \leq (\text{const.}) \lv v \rv_\mathcal{H} \qquad \forall v \in
D(A) = \Ran(B).
\]
Writing $v=Bg$ gives
\[
| \la u,g \ra_\mathcal{H} | \leq (\text{const.}) \lv Bg \rv_\mathcal{H} \qquad \forall g \in \mathcal{H} .
\]
One can express $u$ in terms of the ONB as $u = \sum_j d_j u_j$. Fix $J
\geq 1$ and choose $g = \sum_{j=1}^J \gamma_j^2 d_j u_j \in \mathcal{H}$, so that
$Bg = \sum_{j=1}^J \gamma_j d_j u_j$. We deduce from the last inequality
that
\[
\sum_{j=1}^J \gamma_j^2 |d_j|^2 \leq (\text{const.}) \big( \sum_{j=1}^J
\gamma_j^2 |d_j|^2 \big)^{1/2} ,
\]
and so
\[
\sum_{j=1}^J \gamma_j^2 |d_j|^2 \leq (\text{const.})^2
\]
Letting $J \to \infty$ implies that
\[
\sum_j \gamma_j^2 |d_j|^2 \leq (\text{const.})^2
\]
and so the sequence $\{ \gamma_j d_j \}$ belongs to $\ell^2$. Put
$f=\sum_j \gamma_j d_j u_j \in \mathcal{H}$. Then $Bf = \sum_j d_j u_j = u$, and so
$u \in \Ran(B) = D(A)$, as desired.

\subsubsection*{Selfadjointness, and discreteness of the spectrum}
\begin{theorem} \label{th:pure}
$A$ is selfadjoint, with domain
\[
D(A) = \Ran(B) = \big\{ \sum_j \gamma_j^{-1} c_j u_j : \{ c_j \} \in
\ell^2 \big\} .
\]
Furthermore, $\spec(A)= \spec_{disc}(A) = \{ \gamma_j :  j \geq 1 \}$.
\end{theorem}
\begin{proof}
We have shown above that $A$ is symmetric and $D(A^*)=D(A)$, which
together imply that $A$ is selfadjoint.

We will show that if
\[
\lam \in \C \setminus \{ \gamma_1,\gamma_2,\gamma_3,\ldots \}
\]
then $A-\lam$ is invertible, so that $\lam$ belongs to the resolvent set.
Thus the spectrum consists of precisely the eigenvalues $\gamma_j$. Note
each eigenvalue has finite multiplicity by Theorem~\ref{th:spec}, and is
isolated from the rest of the spectrum; hence $A$ has purely discrete
spectrum.

The inverse of $A-\lam$ can be defined explicitly, as follows. Define a
bounded operator $C : \mathcal{H} \to \mathcal{H}$ on $f=\sum_j c_j u_j \in \mathcal{H}$ by
\[
Cf = \sum_j (\gamma_j - \lam)^{-1} c_j u_j ,
\]
where we note that $(\gamma_j - \lam)^{-1}$ is bounded for all $j$, and in
fact approaches $0$ as $j \to \infty$, because $|\gamma_j - \lam|$ is
never zero and tends to $\infty$ as $j \to \infty$. This new operator has
range $\Ran(C)=\Ran(B)$, because $(\gamma_j - \lam)^{-1}$ is comparable to
$\gamma_j^{-1}$ (referring here to the characterization of $\Ran(B)$ in
Theorem~\ref{th:pure}). Thus $\Ran(C)=D(A)$.

Clearly  $(A-\lam)Cf=f$ by definition of $A$, and so $\Ran(A-\lam)=\mathcal{H}$.
Similarly one finds that $C(A-\lam)u=u$ for all $u \in D(A)$. Thus
$C$ is the inverse operator of $A-\lam$. Because $C$ is bounded on all of $\mathcal{H}$ we conclude $A-\lam$ is invertible, according to the definition in Chapter~\ref{ch:selfadjoint}.
\end{proof}

\subsection*{Example: Laplacian on a bounded domain}

To animate the preceding theory, let us consider the Laplacian on a
bounded domain $\Omega \subset \Rd$, with Dirichlet boundary conditions.
We work with the Hilbert spaces
\[
\mathcal{H} = L^2(\Omega), \qquad \mathcal{K} = H^1_0(\Omega) ,
\]
and the sesquilinear form
\[
a(u,v) = \int_\Omega \nabla u \cdot \nabla v \, dx + \int_\Omega uv \, dx
= \la u,v \ra_{H^1} ,
\]
which in Chapter~\ref{ch:aONB} gave eigenfunctions satisfying
$(-\Delta+1)u=(\lam+1) u$ weakly. In this setting, $u=Bf$ means that
$(-\Delta +1)u=f$ weakly. Note $B : L^2(\Omega) \to H^1_0(\Omega)$, and
recall that $A=B^{-1}$.
\begin{proposition}
The domain of the operator $A$ contains $H^2(\Omega) \cap H^1_0(\Omega)$, and
\[
A = -\Delta + 1
\]
on $H^2(\Omega) \cap H^1_0(\Omega)$.

Furthermore, if $\partial \Omega$ is smooth then $D(A) = H^2(\Omega) \cap
H^1_0(\Omega)$, in which case $A=-\Delta+1$ on all of its domain.
\end{proposition}
\begin{proof}
For all $u \in H^2(\Omega) \cap H^1_0(\Omega), v \in H^1_0(\Omega)$, we have
\begin{align*}
\la u,v \ra_{H^1}
& = \la -\Delta u +u,v \ra_{L^2} && \text{by parts} \\
& = \la B(-\Delta u + u),v \ra_{H^1} && \text{by definition of $B$.}
\end{align*}
Since both $u$ and $B(-\Delta u + u)$ belong to $H^1_0(\Omega)$, and $v
\in H^1_0(\Omega)$ is arbitrary, we conclude from above that $u=B(-\Delta
u + u)$. Therefore $u \in \Ran(B)=D(A)$, and so $H^2(\Omega) \cap
H^1_0(\Omega) \subset D(A)$.

Further, we find $Au=-\Delta u + u$ because $B=A^{-1}$, and so
\[
A = -\Delta + 1 \qquad \text{on $H^2(\Omega) \cap H^1_0(\Omega)$.}
\]

Finally we note that if $\partial \Omega$ is $C^2$-smooth then by elliptic regularity the weak solution $u$ of
$(-\Delta+1)u=f$ belongs to $H^2(\Omega)$, so that $\Ran(B) \subset
H^2(\Omega) \cap H^1_0(\Omega)$. Thus
\[
D(A) = H^2(\Omega) \cap H^1_0(\Omega)
\]
when $\partial \Omega$ is smooth enough. In that case $A=-\Delta+1$ on
all of its domain.
\end{proof}

\end{document}